\documentclass[11pt]{article}\UseRawInputEncoding
\usepackage{amssymb,amsfonts,amsmath,latexsym,epsf,tikz,url}
\usepackage{subcaption}
\usepackage{epsfig}
\usepackage{float}
\usepackage{pst-grad} % For gradients
\usepackage{pst-plot} % For axes
\usepackage[space]{grffile} % For spaces in paths
\usepackage{etoolbox} % For spaces in paths
\makeatletter % For spaces in paths
\patchcmd\Gread@eps{\@inputcheck#1 }{\@inputcheck"#1"\relax}{}{}
\makeatother
\usepackage[color=green!30]{todonotes}
\newcommand{\N}[0]{\mathbb{N}}

\newtheorem{theorem}{Theorem}[section]
\newtheorem{proposition}[theorem]{Proposition}

\newtheorem{corollary}[theorem]{Corollary}
\newtheorem{lemma}[theorem]{Lemma}
\newtheorem{remark}[theorem]{Remark}
\newtheorem{example}[theorem]{Example}
\newtheorem{definition}[theorem]{Definition}

\newcommand{\qed}{\hfill $\square$\medskip}

\begin{document}

\title{Super Domination: Graph Classes, Products and Enumeration}

\author{
Nima Ghanbari$^{1,}$\footnote{Corresponding author}
\and
Gerold J{\"a}ger$^{2}$
\and
Tuomo Lehtil{\"a}$^{3}$
}

\date{\today}

\maketitle

\begin{center}
$^1$Department of Informatics, University of Bergen, P.O. Box 7803, 5020 Bergen, Norway\\
$^2$Department of Mathematics and Mathematical Statistics,  University of Ume\aa, SE-901-87 Ume\aa, Sweden\\
$^3$Department of Mathematics and Statistics, University of Turku, Turku FI-20014, Finland\\
\medskip
{\tt  $^1$Nima.Ghanbari@uib.no ~~ $^2$gerold.jager@umu.se  ~~ $^3$tualeh@utu.fi}
\end{center}

\medskip
%%%%%%%%%%%%%%ABSTRACT%%%%%%%%%%%%%%%%%%%%%%%%%%%%%%%%%%%%%%%%%%%%%%%%%%%%%%%%%%%%

\begin{abstract}
The dominating set problem (DSP) is one of the most famous problems
in combinatorial optimization. It is defined as follows.
For a given simple graph $G=(V,E)$, a dominating set of $G$ is a subset
$S\subseteq V$ such that every vertex in $ V \setminus S$ is adjacent to at 
least one vertex in $S$. Furthermore, the DSP is the problem of finding a minimum-size dominating set 
and the corresponding minimum size, the domination number of $G$.

In this, work we investigate a variant of the DSP, the super dominating set 
problem (SDSP),
which has attracted much attention during the last years.
A dominating set $S$ is called a super dominating set of $G$, if for every 
vertex  $u\in \overline{S}=V \setminus S$, there exists a $v\in S$ such that 
$N(v)\cap \overline{S}=\{u\}$.  Analogously, the SDSP is to find a minimum-size super 
dominating set, and the corresponding minimum size, the super domination 
number of $G$.  
The decision variants of both the DSP and the SDSP have shown to be $\mathcal{NP}$-hard.

In this paper, we present tight bounds for the super domination number of 
the neighbourhood corona product, $r$-gluing, and the Haj\'{o}s sum of  two 
graphs.
Additionally, we present infinite families of graphs attaining our bounds.
Finally, we give the exact number of minimum size super dominating sets for 
some  graph classes. In particular, the number of super dominating sets for 
cycles has quite surprising properties as it varies between values  of the 
set $\{4,n,2n,\frac{5n^2-10n}{8}\}$ based on $n\mod4$.
\end{abstract}

\noindent{\bf Keywords:} domination number, super dominating set, 
neighbourhood corona product, $r$-gluing,
Haj\'{o}s sum.

\medskip
\noindent{\bf AMS Subj.\ Class.:} 05C38, 05C69, 05C76

%%%%%%%%%%%%%%%%%%%%%%%%%%%%%%%%%%%%%%%%%%%%%%%%%%%%%%%%%%%%%%%%%%%%%%%%%%%%%%%%%
%%%%%%%%%%%%%%%%%%%%%%%%%%%%%%%%%%%%%%%%%%%%%%%%%%%%%%%%%%%%%%%%%%%%%%%%%%%%%%%%%
\section{Introduction}
 
\paragraph{Notations and definitions.}

Let $G = (V,E)$ be a  graph with vertex set $V$ and edge set $E$. Throughout 
this paper, we consider finite undirected graphs without loops.
For each vertex $v\in V$, the set $N(v) = N_G(v)=\{u\in V \mid uv \in E\}$ refers to the 
\textit{open neighbourhood} of $v$ in $G$ and the set $N[v]=N_G[v]=N_G(v)\cup \{v\}$ refers to 
the \textit{closed neighbourhood} of $v$ in $G$. If the graph $G$ 
is clear from the context, we will omit the corresponding index~$G$. The \textit{degree of $v$} is the cardinality of $N(v)$. 
Throughout this paper, for a set $ S \subseteq V(G) $, the expression 
$ \overline{S} $ always stands for $ V(G) \setminus S $.

\medskip

A set $S\subseteq V$ is called a  \textit{dominating set} if every vertex in 
$\overline{S}= V \setminus S$ is adjacent to at least one vertex in $S$.
The  \textit{domination number} $\gamma(G)$ is the cardinality of a minimum size
dominating set in $G$. For a detailed treatment of domination theory, we refer the reader 
to~\cite{domination}. 

\medskip

A dominating set $S$ of $G$ is called a \textit{super dominating set} of $G$, if for 
every vertex  $u\in \overline{S}$, there exists a $v\in S$ such that $N(v)\cap 
\overline{S}=\{u\}$. In this case, we say that $v$ super dominates $u$.
The \textit{super domination number} $G$ is the cardinality of a minimum size super 
dominating set of $G$, denoted by $\gamma_{sp}(G)$~\cite{Lemans}. 
We refer the reader to~\cite{Alf,Dett,Nima,Nima1,Kri,Kle,Zhu} 
for more details on super dominating sets of a graph.
Some applications can be found in~\cite{Lemans}.

In this work, we often call vertices  ``black'' (abbreviated by ``B'') 
if they are contained in a given super dominating set and 
``white'' (abbreviated ``W'') if they are \textit{not} contained in it.

\medskip

\paragraph{Graph classes.}

We will consider the following well-known graph classes: \textit{path graph}, 
\textit{cycle graph}, \textit{star graph}, \textit{complete graph}, \textit{complete bipartite graph}. A \textit{friendship graph} $F_n$ is a collection
of $n$ triangles, where all triangles have one vertex, the central vertex, in common
(see Figure~\ref{friend}).

\begin{figure}
\begin{center}
\psscalebox{0.6 0.6}
{
\begin{pspicture}(0,-7.215)(20.277115,-1.245)
\psline[linecolor=black, linewidth=0.04](1.7971154,-1.815)(1.7971154,-1.815)
\psdots[linecolor=black, dotsize=0.4](8.997115,-1.815)
\psdots[linecolor=black, dotsize=0.4](10.5971155,-1.815)
\psdots[linecolor=black, dotsize=0.4](9.797115,-4.215)
\psdots[linecolor=black, dotsize=0.4](8.997115,-6.615)
\psdots[linecolor=black, dotsize=0.4](10.5971155,-6.615)
\psline[linecolor=black, linewidth=0.08](8.997115,-1.815)(10.5971155,-1.815)(8.997115,-6.615)(10.5971155,-6.615)(8.997115,-1.815)(8.997115,-1.815)
\psdots[linecolor=black, dotsize=0.4](12.197115,-3.415)
\psdots[linecolor=black, dotsize=0.4](12.197115,-5.015)
\psdots[linecolor=black, dotsize=0.4](7.397115,-3.415)
\psdots[linecolor=black, dotsize=0.4](7.397115,-5.015)
\psline[linecolor=black, linewidth=0.08](12.197115,-5.015)(7.397115,-3.415)(7.397115,-5.015)(12.197115,-3.415)(12.197115,-5.015)(12.197115,-5.015)
\psdots[linecolor=black, dotsize=0.4](0.1971154,-3.415)
\psdots[linecolor=black, dotsize=0.4](0.1971154,-5.015)
\psdots[linecolor=black, dotsize=0.4](2.5971155,-4.215)
\psline[linecolor=black, linewidth=0.08](2.5971155,-4.215)(0.1971154,-3.415)(0.1971154,-5.015)(2.5971155,-4.215)(2.5971155,-4.215)
\psdots[linecolor=black, dotsize=0.4](3.3971155,-1.815)
\psdots[linecolor=black, dotsize=0.4](4.5971155,-2.615)
\psdots[linecolor=black, dotsize=0.4](3.3971155,-6.615)
\psdots[linecolor=black, dotsize=0.4](4.5971155,-5.815)
\psline[linecolor=black, linewidth=0.08](2.5971155,-4.215)(4.5971155,-5.815)(3.3971155,-6.615)(3.3971155,-6.615)
\psline[linecolor=black, linewidth=0.08](2.5971155,-4.215)(3.3971155,-6.615)(3.3971155,-6.615)
\psline[linecolor=black, linewidth=0.08](2.5971155,-4.215)(3.3971155,-1.815)(4.5971155,-2.615)(2.5971155,-4.215)(2.5971155,-4.215)
\psdots[linecolor=black, dotsize=0.4](15.397116,-2.615)
\psdots[linecolor=black, dotsize=0.4](16.597115,-1.815)
\psdots[linecolor=black, dotsize=0.4](17.397116,-4.215)
\psdots[linecolor=black, dotsize=0.4](15.397116,-5.815)
\psdots[linecolor=black, dotsize=0.4](16.597115,-6.615)
\psdots[linecolor=black, dotsize=0.4](19.397116,-5.815)
\psdots[linecolor=black, dotsize=0.4](18.197115,-6.615)
\psdots[linecolor=black, dotsize=0.4](14.997115,-3.415)
\psdots[linecolor=black, dotsize=0.4](14.997115,-5.015)
\psdots[linecolor=black, dotsize=0.4](18.197115,-1.815)
\psdots[linecolor=black, dotsize=0.4](19.397116,-2.615)
\psdots[linecolor=black, dotsize=0.1](18.997116,-3.815)
\psdots[linecolor=black, dotsize=0.1](18.997116,-4.215)
\psdots[linecolor=black, dotsize=0.1](18.997116,-4.615)
\rput[bl](17.697115,-4.295){$x$}
\rput[bl](16.137115,-1.595){$u_1$}
\rput[bl](14.857116,-2.375){$u_2$}
\psline[linecolor=black, linewidth=0.08](17.397116,-4.215)(19.397116,-2.615)(18.197115,-1.815)(17.397116,-4.215)(16.597115,-1.815)(15.397116,-2.615)(17.397116,-4.215)(14.997115,-3.415)(14.997115,-5.015)(17.397116,-4.215)(15.397116,-5.815)(16.597115,-6.615)(17.397116,-4.215)(18.197115,-6.615)(19.397116,-5.815)(17.397116,-4.215)(17.397116,-4.215)
\rput[bl](14.277116,-3.495){$u_3$}
\rput[bl](14.297115,-5.095){$u_4$}
\rput[bl](14.817116,-6.195){$u_5$}
\rput[bl](16.217115,-7.175){$u_6$}
\rput[bl](18.037115,-7.215){$u_7$}
\rput[bl](19.517115,-6.295){$u_8$}
\rput[bl](18.097115,-1.495){$u_{2n}$}
\rput[bl](19.337116,-2.315){$u_{2n-1}$}
\end{pspicture}
}
\end{center}
\caption{Friendship graphs $F_3$, $F_4$ and $F_n$, respectively.}\label{friend}
\end{figure}
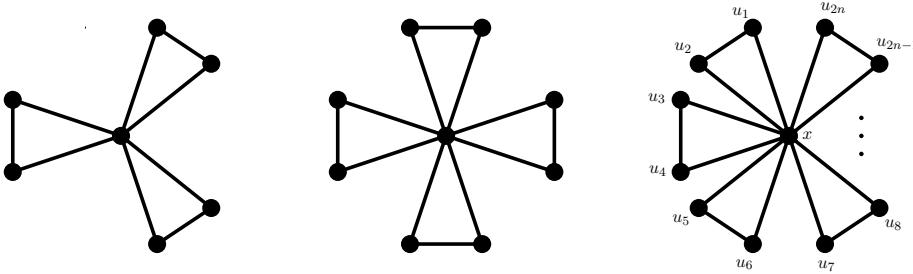

%\begin{figure}
%\centering
%\begin{subfigure}{0.32\linewidth}
%\includegraphics[width=0.9\textwidth]{f3.png}
%\caption{$F_3$.}
%\end{subfigure}
%\hfill
%\begin{subfigure}{0.32\linewidth}
%\includegraphics[width=0.9\textwidth]{f4.png}
%\caption{$F_4$.}
%\end{subfigure}
%\hfill
%\begin{subfigure}{0.32\linewidth}
%\includegraphics[width=0.9\textwidth]{fn.png}
%\caption{$F_n$}
%\end{subfigure}
%\caption{Examples for friendship graphs.}
%\label{friend}
%\end{figure}

Given two simple graphs  $G_1$ and $G_2$, the \textit{corona product} of 
$G_1$ and $G_2$, denoted by $G_1\circ G_2$ is the graph arising from the
disjoint union of $G_1$ with $| V(G_1) |$ copies of $G_2$, by adding edges between
the $i$-th vertex of $G_1$ and all vertices of the $i$-th copy of $G_2$ \cite{Har}.
The \textit{neighbourhood corona product} of 
$G_1$ and $G_2$, denoted by $G_1 \star G_2$, is the graph obtained by
taking one copy of $G_1$ and $|V(G_1)|$ copies of $G_2$ and  
joining the neighbours of the $i$-th vertex of $G_1$ to every vertex in the $i$-th 
copy of $G_2$~\cite{Gopalapillai}. 
Thus, this graph has $ |V(G_1)| \cdot (|V(G_2)| +1 ) $ vertices.
Figure~\ref{C4starC3} shows $C_4\star K_3$, where $C_n$ is  the cycle of order~$n$ 
and $K_n$ is the complete graph of order $n$. For more results on the neighbourhood corona product of two graphs, we refer the reader to~\cite{Barik,Nima2,linear}.

\begin{figure}
\begin{center}
\psscalebox{0.4 0.4}
{
\begin{pspicture}(0,-8.0)(9.594231,1.994231)
\psdots[linecolor=black, dotsize=0.4](2.9971154,-1.0028845)
\psdots[linecolor=black, dotsize=0.4](6.9971156,-1.0028845)
\psdots[linecolor=black, dotsize=0.4](2.9971154,-5.0028844)
\psdots[linecolor=black, dotsize=0.4](6.9971156,-5.0028844)
\psdots[linecolor=black, dotsize=0.4](7.7971153,1.7971154)
\psdots[linecolor=black, dotsize=0.4](9.397116,0.19711548)
\psdots[linecolor=black, dotsize=0.4](9.397116,1.7971154)
\psline[linecolor=black, linewidth=0.08](2.9971154,-1.0028845)(6.9971156,-1.0028845)(6.9971156,-5.0028844)(2.9971154,-5.0028844)(2.9971154,-1.0028845)(2.9971154,-1.0028845)
\psline[linecolor=black, linewidth=0.08](7.7971153,1.7971154)(9.397116,1.7971154)(9.397116,0.19711548)(7.7971153,1.7971154)(7.7971153,1.7971154)
\psline[linecolor=black, linewidth=0.04](2.9971154,-1.0028845)(9.397116,1.7971154)(6.9971156,-5.0028844)(6.9971156,-5.0028844)
\psline[linecolor=black, linewidth=0.04](2.9971154,-1.0028845)(7.7971153,1.7971154)(6.9971156,-5.0028844)(6.9971156,-5.0028844)
\psline[linecolor=black, linewidth=0.04](2.9971154,-1.0028845)(9.397116,0.19711548)(6.9971156,-5.0028844)(6.9971156,-5.0028844)
\psdots[linecolor=black, dotsize=0.4](9.397116,-6.2028847)
\psdots[linecolor=black, dotsize=0.4](9.397116,-7.8028846)
\psdots[linecolor=black, dotsize=0.4](7.7971153,-7.8028846)
\psdots[linecolor=black, dotsize=0.4](0.59711546,-6.2028847)
\psdots[linecolor=black, dotsize=0.4](0.59711546,-7.8028846)
\psdots[linecolor=black, dotsize=0.4](2.1971154,-7.8028846)
\psdots[linecolor=black, dotsize=0.4](0.19711548,0.19711548)
\psdots[linecolor=black, dotsize=0.4](0.19711548,1.7971154)
\psdots[linecolor=black, dotsize=0.4](1.7971154,1.7971154)
\psline[linecolor=black, linewidth=0.08](6.9971156,-1.0028845)(6.9971156,-1.0028845)
\psline[linecolor=black, linewidth=0.04](6.9971156,-1.0028845)(1.7971154,1.7971154)(2.9971154,-5.0028844)(0.19711548,0.19711548)(6.9971156,-1.0028845)(0.19711548,1.7971154)(2.9971154,-5.0028844)(2.9971154,-5.0028844)
\psline[linecolor=black, linewidth=0.04](2.9971154,-1.0028845)(0.59711546,-6.2028847)(6.9971156,-5.0028844)(2.1971154,-7.8028846)(2.9971154,-1.0028845)(0.59711546,-7.8028846)(6.9971156,-5.0028844)(2.9971154,-5.0028844)(7.7971153,-7.8028846)(6.9971156,-1.0028845)(9.397116,-6.2028847)(2.9971154,-5.0028844)(9.397116,-7.8028846)(6.9971156,-1.0028845)(6.9971156,-1.0028845)
\psline[linecolor=black, linewidth=0.08](0.19711548,1.7971154)(1.7971154,1.7971154)(0.19711548,0.19711548)(0.19711548,1.7971154)(0.19711548,1.7971154)
\psline[linecolor=black, linewidth=0.08](0.59711546,-6.2028847)(2.1971154,-7.8028846)(0.59711546,-7.8028846)(0.59711546,-6.2028847)(0.59711546,-6.2028847)
\psline[linecolor=black, linewidth=0.08](9.397116,-6.2028847)(9.397116,-7.8028846)(7.7971153,-7.8028846)(9.397116,-6.2028847)(9.397116,-6.2028847)
\end{pspicture}
}
\end{center}
\caption{$C_4 \star K_3$.} \label{C4starC3}
\end{figure}
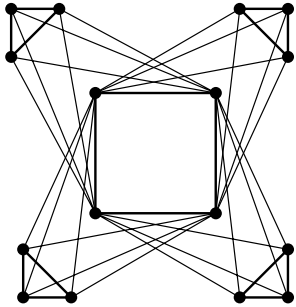

%\begin{figure}
%\centering
%\includegraphics[width=0.7\textwidth]{c4stark3.png}
%\caption{$C_4 \star K_3$.} 
%\label{C4starC3}
%\end{figure}

%If $r\in \N \cup \{0\}$ and $r\leq min 
%\{\omega(G_1),\omega(G_2)\}$, where $\omega(G)$ is the clique number of $G$, 
%then we say that $G\in G_1\cup_{K_r}G_2$ if $G$ can be formed from $G_1$ and 
%$G_2$ by identifying an $r$-clique from both graphs. See~\cite{Nima2} for some 
%previous results on $r$-gluing.  

Let $G_1$ and $G_2$ be two graphs and $r\in \N_0 := \N\cup \{0\}$ with $r\leq 
\min \{\omega(G_1),\omega(G_2)\}$, where $\omega(G)$ is the clique number of 
$G$.  Choose a clique $K_r$ from each $G_i$, $i=1,2$, and form a new graph $G$ 
from the union of $G_1$ and $G_2$ by identifying the two chosen $r$-cliques in 
an arbitrary manner. The graph $G$ is called \emph{$r$-gluing} of $G_1$ and $G_2$ and 
denoted by $G_1\cup_{K_r}G_2$. If $r=0$, then $G_1\cup_{K_0}G_2$ is just its
disjoint union. $G_1\cup_{K_i}G_2$ for $i=1,2$, is called vertex and edge 
gluing, respectively. Notice that there are sometimes several ways to $r$-glue 
two graphs together (see Figure \ref{GcupHmorethanone}). 
We refer the reader for some results on the $r$-gluing of two 
graphs to~\cite{Nima2}.

\begin{figure}
\begin{center}
\psscalebox{0.6 0.6}
{
\begin{pspicture}(0,-4.9014425)(18.394232,2.895673)
\psdots[linecolor=black, dotsize=0.4](0.19711533,-0.10144226)
\psdots[linecolor=black, dotsize=0.4](1.7971153,-0.10144226)
\psdots[linecolor=black, dotsize=0.4](1.7971153,-1.7014422)
\psdots[linecolor=black, dotsize=0.4](0.19711533,-1.7014422)
\psline[linecolor=black, linewidth=0.08](0.19711533,-0.10144226)(1.7971153,-0.10144226)(1.7971153,-1.7014422)(0.19711533,-1.7014422)(0.19711533,-0.10144226)(0.19711533,-0.10144226)
\psline[linecolor=black, linewidth=0.08](1.7971153,-0.10144226)(0.19711533,-1.7014422)(0.19711533,-1.7014422)
\psline[linecolor=black, linewidth=0.08](0.19711533,-0.10144226)(1.7971153,-1.7014422)(1.7971153,-1.7014422)
\psdots[linecolor=black, dotsize=0.4](1.7971153,1.4985577)
\psdots[linecolor=black, dotsize=0.4](0.19711533,1.4985577)
\psdots[linecolor=black, dotsize=0.4](1.7971153,2.6985579)
\psdots[linecolor=black, dotsize=0.4](0.19711533,2.6985579)
\psdots[linecolor=black, dotsize=0.4](1.7971153,-3.3014421)
\psdots[linecolor=black, dotsize=0.4](0.19711533,-3.3014421)
\psline[linecolor=black, linewidth=0.08](1.7971153,-0.10144226)(1.7971153,2.6985579)(0.19711533,2.6985579)(0.19711533,-0.10144226)(0.19711533,-0.10144226)
\psline[linecolor=black, linewidth=0.08](1.7971153,-1.7014422)(1.7971153,-3.3014421)(0.19711533,-3.3014421)(0.19711533,-1.7014422)(0.19711533,-1.7014422)
\psdots[linecolor=black, dotsize=0.4](4.1971154,-1.7014422)
\psdots[linecolor=black, dotsize=0.4](5.7971153,-1.7014422)
\psdots[linecolor=black, dotsize=0.4](4.997115,-0.10144226)
\psline[linecolor=black, linewidth=0.08](4.1971154,-1.7014422)(4.997115,-0.10144226)(5.7971153,-1.7014422)(4.1971154,-1.7014422)(4.1971154,-1.7014422)
\psline[linecolor=black, linewidth=0.08](4.997115,-0.10144226)(4.1971154,1.4985577)(4.1971154,1.4985577)
\psline[linecolor=black, linewidth=0.08](4.997115,-0.10144226)(5.7971153,1.4985577)(5.7971153,1.4985577)
\psdots[linecolor=black, dotsize=0.4](4.1971154,1.4985577)
\psdots[linecolor=black, dotsize=0.4](5.7971153,1.4985577)
\psdots[linecolor=black, dotsize=0.4](10.197115,-0.10144226)
\psdots[linecolor=black, dotsize=0.4](11.797115,-0.10144226)
\psdots[linecolor=black, dotsize=0.4](11.797115,-1.7014422)
\psdots[linecolor=black, dotsize=0.4](10.197115,-1.7014422)
\psline[linecolor=black, linewidth=0.08](10.197115,-0.10144226)(11.797115,-0.10144226)(11.797115,-1.7014422)(10.197115,-1.7014422)(10.197115,-0.10144226)(10.197115,-0.10144226)
\psline[linecolor=black, linewidth=0.08](11.797115,-0.10144226)(10.197115,-1.7014422)(10.197115,-1.7014422)
\psline[linecolor=black, linewidth=0.08](10.197115,-0.10144226)(11.797115,-1.7014422)(11.797115,-1.7014422)
\psdots[linecolor=black, dotsize=0.4](11.797115,1.4985577)
\psdots[linecolor=black, dotsize=0.4](10.197115,1.4985577)
\psdots[linecolor=black, dotsize=0.4](11.797115,2.6985579)
\psdots[linecolor=black, dotsize=0.4](10.197115,2.6985579)
\psdots[linecolor=black, dotsize=0.4](11.797115,-3.3014421)
\psdots[linecolor=black, dotsize=0.4](10.197115,-3.3014421)
\psline[linecolor=black, linewidth=0.08](11.797115,-0.10144226)(11.797115,2.6985579)(10.197115,2.6985579)(10.197115,-0.10144226)(10.197115,-0.10144226)
\psline[linecolor=black, linewidth=0.08](11.797115,-1.7014422)(11.797115,-3.3014421)(10.197115,-3.3014421)(10.197115,-1.7014422)(10.197115,-1.7014422)
\psdots[linecolor=black, dotsize=0.4](12.997115,0.69855773)
\psdots[linecolor=black, dotsize=0.4](12.997115,-0.9014423)
\psline[linecolor=black, linewidth=0.08](11.797115,-0.10144226)(12.997115,0.69855773)(12.997115,0.69855773)
\psline[linecolor=black, linewidth=0.08](11.797115,-0.10144226)(12.997115,-0.9014423)(12.997115,-0.9014423)
\psdots[linecolor=black, dotsize=0.4](15.397116,-0.10144226)
\psdots[linecolor=black, dotsize=0.4](16.997116,-0.10144226)
\psdots[linecolor=black, dotsize=0.4](16.997116,-1.7014422)
\psdots[linecolor=black, dotsize=0.4](15.397116,-1.7014422)
\psline[linecolor=black, linewidth=0.08](15.397116,-0.10144226)(16.997116,-0.10144226)(16.997116,-1.7014422)(15.397116,-1.7014422)(15.397116,-0.10144226)(15.397116,-0.10144226)
\psline[linecolor=black, linewidth=0.08](16.997116,-0.10144226)(15.397116,-1.7014422)(15.397116,-1.7014422)
\psline[linecolor=black, linewidth=0.08](15.397116,-0.10144226)(16.997116,-1.7014422)(16.997116,-1.7014422)
\psdots[linecolor=black, dotsize=0.4](16.997116,1.4985577)
\psdots[linecolor=black, dotsize=0.4](15.397116,1.4985577)
\psdots[linecolor=black, dotsize=0.4](16.997116,2.6985579)
\psdots[linecolor=black, dotsize=0.4](15.397116,2.6985579)
\psdots[linecolor=black, dotsize=0.4](16.997116,-3.3014421)
\psdots[linecolor=black, dotsize=0.4](15.397116,-3.3014421)
\psline[linecolor=black, linewidth=0.08](16.997116,-0.10144226)(16.997116,2.6985579)(15.397116,2.6985579)(15.397116,-0.10144226)(15.397116,-0.10144226)
\psline[linecolor=black, linewidth=0.08](16.997116,-1.7014422)(16.997116,-3.3014421)(15.397116,-3.3014421)(15.397116,-1.7014422)(15.397116,-1.7014422)
\psdots[linecolor=black, dotsize=0.4](18.197115,-0.9014423)
\psdots[linecolor=black, dotsize=0.4](18.197115,-2.5014422)
\psline[linecolor=black, linewidth=0.08](16.997116,-1.7014422)(18.197115,-0.9014423)(18.197115,-0.9014423)
\psline[linecolor=black, linewidth=0.08](16.997116,-1.7014422)(18.197115,-2.5014422)(18.197115,-2.5014422)
\rput[bl](0.7971153,-4.801442){$G$}
\rput[bl](4.8171153,-4.821442){$H$}
\rput[bl](10.317116,-4.881442){$G\cup_{K_3} H$}
\rput[bl](15.457115,-4.901442){$G\cup_{K_3} H$}
\end{pspicture}
}
\end{center}
\caption{Graphs $G$, $H$ and all non-isomorphic graphs $G\cup_{K_3}H$, respectively.} \label{GcupHmorethanone}
\end{figure}
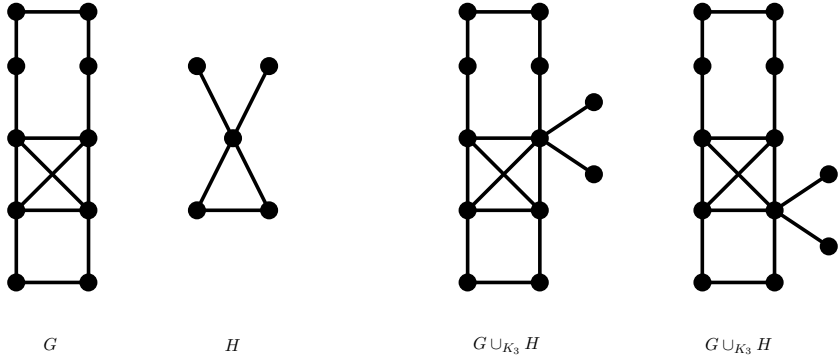

%\begin{figure}
%\centering
%\begin{subfigure}{0.24\linewidth}
%\includegraphics[width = 0.9\textwidth]{gunion.png}
%\caption{$G$.}
%\end{subfigure}
%\hfill
%\begin{subfigure}{0.24\linewidth}
%\includegraphics[width = 0.9\textwidth]{hunion.png}
%\caption{$H$.}
%\end{subfigure}
%\hfill
%\begin{subfigure}{0.48\linewidth}
%\includegraphics[width = 0.9\textwidth]{ghunion.png}
%\caption{Two versions of $ G \cup_{K_3} H $.}
%\end{subfigure}
%\caption{Example, where $r$-gluing of two graphs is not unique.}
%\label{GcupHmorethanone}
%\end{figure}

Given graphs $G_1 = (V_1,E_1)$ and $G_2 = (V_2, E_2)$
with disjoint vertex sets, an edge $x_1y_1\in E_1$, and an edge $x_2y_2\in E_2$, the
\emph{Haj\'{o}s sum} $G_3 = G_1(x_1y_1) +_H G_2(x_2y_2)$ is the graph obtained as follows:
begin with $G_3 = (V_1 \cup V_2, E_1 \cup E_2)$; then in $G_3$ delete the edges $x_1y_1$ and
$x_2y_2$, identify the vertices $x_1$ and $x_2$ as $v_H(x_1x_2)$, and add the edge $y_1y_2$~\cite{HAJOSSUM}.
%Now define $G_3$ as the current $G_3'$~\cite{HAJOSSUM}. 
Figure~\ref{HaJ-K4C4} shows the Haj\'{o}s sum  of $K_4$ and $C_4$ 
with respect to $x_1y_1$ and $x_2y_2$. 

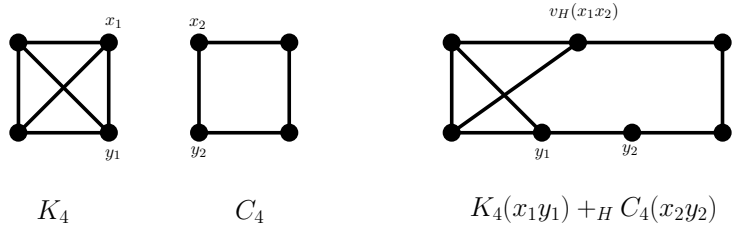
\begin{figure}
\begin{center}
\psscalebox{0.6 0.6}
{
\begin{pspicture}(0,-7.555)(15.994231,-2.665)
\rput[bl](2.1171153,-3.235){$x_1$}
\rput[bl](3.9771154,-3.255){$x_2$}
\rput[bl](2.1171153,-6.155){$y_1$}
\rput[bl](4.017115,-6.095){$y_2$}
\rput[bl](11.6371155,-6.135){$y_1$}
\rput[bl](13.577115,-6.055){$y_2$}
\rput[bl](11.957115,-3.055){$v_H(x_1x_2)$}
\psline[linecolor=black, linewidth=0.08](2.1971154,-3.555)(0.19711533,-3.555)(0.19711533,-5.555)(2.1971154,-5.555)(2.1971154,-3.555)(0.19711533,-5.555)(0.19711533,-5.555)
\psline[linecolor=black, linewidth=0.08](0.19711533,-3.555)(2.1971154,-5.555)(2.1971154,-5.555)
\psdots[linecolor=black, dotsize=0.4](0.19711533,-3.555)
\psdots[linecolor=black, dotsize=0.4](2.1971154,-3.555)
\psdots[linecolor=black, dotsize=0.4](0.19711533,-5.555)
\psdots[linecolor=black, dotsize=0.4](2.1971154,-5.555)
\psline[linecolor=black, linewidth=0.08](4.1971154,-5.555)(4.1971154,-3.555)(6.1971154,-3.555)(6.1971154,-5.555)(4.1971154,-5.555)(4.1971154,-5.555)
\psdots[linecolor=black, dotsize=0.4](4.1971154,-3.555)
\psdots[linecolor=black, dotsize=0.4](6.1971154,-3.555)
\psdots[linecolor=black, dotsize=0.4](6.1971154,-5.555)
\psdots[linecolor=black, dotsize=0.4](4.1971154,-5.555)
\psline[linecolor=black, linewidth=0.08](9.797115,-3.555)(11.797115,-5.555)(11.797115,-5.555)
\psdots[linecolor=black, dotsize=0.4](9.797115,-3.555)
\psdots[linecolor=black, dotsize=0.4](12.5971155,-3.555)
\psdots[linecolor=black, dotsize=0.4](9.797115,-5.555)
\psdots[linecolor=black, dotsize=0.4](11.797115,-5.555)
\psdots[linecolor=black, dotsize=0.4](15.797115,-3.555)
\psdots[linecolor=black, dotsize=0.4](15.797115,-5.555)
\psdots[linecolor=black, dotsize=0.4](13.797115,-5.555)
\psline[linecolor=black, linewidth=0.08](11.797115,-5.555)(13.797115,-5.555)(15.797115,-5.555)(15.797115,-3.555)(15.797115,-3.555)
\psline[linecolor=black, linewidth=0.08](15.797115,-3.555)(9.797115,-3.555)(9.797115,-5.555)(11.797115,-5.555)(11.797115,-5.555)
\psline[linecolor=black, linewidth=0.08](9.797115,-5.555)(12.5971155,-3.555)(12.5971155,-3.555)
\rput[bl](0.59711534,-7.555){\LARGE{$K_4$}}
\rput[bl](4.997115,-7.555){\LARGE{$C_4$}}
\rput[bl](10.197115,-7.555){\LARGE{$K_4(x_1y_1)+_H C_4(x_2y_2)$}}
\end{pspicture}
}
\end{center}
\caption{Haj\'{o}s construction of $K_4$ and $C_4$.} \label{HaJ-K4C4}
\end{figure}

%\begin{figure}[H]
%\centering
%%\begin{subfigure}{0.32\linewidth}
%%\includegraphics[width=0.9\textwidth]{k4hajos.png}
%%\subcaptionbox[short]{$K_4$.}
%%\end{subfigure}
%%\hfill
%\begin{subfigure}{0.32\linewidth}
%\includegraphics[width=0.9\textwidth]{c4hajos.png}
%\subcaptionbox{$C_4$.}
%\end{subfigure}
%\hfill
%\begin{subfigure}{0.32\linewidth}
%\includegraphics[width=0.9\textwidth]{k4c4hajos.png}
%\subcaptionbox[short]{$ K_4(x_1 y_1) +_H C_4(x_2y_2)$.}
%\end{subfigure}
%\caption{Haj\'{o}s construction of $K_4$ and $C_4$.} 
%\label{HaJ-K4C4}
%\end{figure}

\medskip
\paragraph{Previous work.} Super dominating sets have been studied in multiple 
papers since the inception of the concept in 2015~\cite{Lemans}. 
In particular, the following 
tight lower and upper bounds are known for the super domination number.
	\begin{theorem}\cite{Lemans}\label{thm-1}
 Let $G$ be a graph of order $n$ without isolated vertices. Then,
 $$1\leq \gamma(G) \leq \frac{n}{2} \leq \gamma_{sp}(G) \leq n-1.$$
	\end{theorem}

Besides general bounds, the super domination number is known exactly for 
many graph classes, some stated in  the following theorem.

\begin{theorem}\cite{Lemans}\label{thm-2}
Let $n \in \N$.
\begin{itemize}
\item[(a)]
For the path graph $P_n$ it holds that $\gamma_{sp}(P_n)=\lceil \frac{n}{2} \rceil$.
\item[(b)]
For the cycle graph $C_n$ it holds that
\begin{displaymath}
\gamma_{sp}(C_n)= \left\{ \begin{array}{ll}
\lceil\frac{n+1}{2}\rceil & \textrm{if $n \equiv 2 \pmod 4$, }\\
\\
\lceil\frac{n}{2}\rceil & \textrm{otherwise.}
\end{array} \right.
\end{displaymath}
\item[(c)]
For the complete graph $K_n$, where $n\geq 2$, it holds that  $\gamma_{sp}(K_n)=n-1$.
\item[(d)]
For the complete bipartite graph $K_{n,m}$, where $\min\{n,m\}\geq 2$,
it holds that
$\gamma_{sp}(K_{n,m})=n+m-2$. 
\item[(e)]
For the star graph $K_{1,n}$ it holds that $\gamma_{sp}(K_{1,n})=n$.
\end{itemize}
\end{theorem}

\begin{theorem}\cite{Nima}\label{Firend-thm}
 For the friendship graph $F_n$ it holds that $\gamma_{sp}(F_n)=n+1$.
\end{theorem}

Later we will refer to the following known 
results from~\cite{Nima,Nima1}.

\begin{proposition}\cite{Nima1}\label{pro-disconnect}
Let $G$ be a disconnected graph with components $G_1$ and $G_2$. Then $$\gamma 
_{sp}(G)=\gamma _{sp}(G_1)+\gamma _{sp}(G_2).$$
\end{proposition}

\begin{theorem}\cite{Nima1}\label{chain}
	Let $G_1,G_2, \ldots, G_n$ be a finite sequence of pairwise disjoint connected graphs and let
$x_i,y_i \in V(G_i)$. Let $C(G_1,G_2,\dots,G_n)$ be the chain of graphs 
$\{G_i\}_{i=1}^n$ with respect to the vertices $\{x_i, y_i\}_{i=1}^k$ 
obtained by identifying the vertex $y_i$ with the vertex $x_{i+1}$ for 
$i=1,2,\ldots,n-1$ (see Figure~\ref{chain-n}). Then, for $n=2$, we have:
$$\gamma _{sp}(G_1)+\gamma _{sp}(G_2) -1\leq \gamma _{sp}(C(G_1,G_2)) \leq \gamma _{sp}(G_1)+\gamma _{sp}(G_2).$$
Furthermore, these bounds are tight.
\end{theorem}

\begin{figure}
\begin{center}
\psscalebox{0.75 0.75}
{
\begin{pspicture}(0,-3.9483333)(12.236668,-2.8316667)
\psellipse[linecolor=black, linewidth=0.04, dimen=outer](1.2533334,-3.4416668)(1.0,0.4)
\psellipse[linecolor=black, linewidth=0.04, dimen=outer](3.2533333,-3.4416668)(1.0,0.4)
\psellipse[linecolor=black, linewidth=0.04, dimen=outer](5.2533336,-3.4416668)(1.0,0.4)
\psellipse[linecolor=black, linewidth=0.04, dimen=outer](8.853333,-3.4416668)(1.0,0.4)
\psellipse[linecolor=black, linewidth=0.04, dimen=outer](10.853333,-3.4416668)(1.0,0.4)
\psdots[linecolor=black, fillstyle=solid, dotstyle=o, dotsize=0.3, fillcolor=white](2.2533333,-3.4416666)
				\psdots[linecolor=black, fillstyle=solid, dotstyle=o, dotsize=0.3, fillcolor=white](0.25333345,-3.4416666)
				\psdots[linecolor=black, fillstyle=solid, dotstyle=o, dotsize=0.3, fillcolor=white](2.2533333,-3.4416666)
				\psdots[linecolor=black, fillstyle=solid, dotstyle=o, dotsize=0.3, fillcolor=white](4.2533336,-3.4416666)
				\psdots[linecolor=black, fillstyle=solid, dotstyle=o, dotsize=0.3, fillcolor=white](4.2533336,-3.4416666)
				\psdots[linecolor=black, fillstyle=solid, dotstyle=o, dotsize=0.3, fillcolor=white](9.853333,-3.4416666)
				\psdots[linecolor=black, fillstyle=solid, dotstyle=o, dotsize=0.3, fillcolor=white](9.853333,-3.4416666)
				\psdots[linecolor=black, fillstyle=solid, dotstyle=o, dotsize=0.3, fillcolor=white](11.853333,-3.4416666)
				\rput[bl](0.0,-3.135){$x_1$}
				\rput[bl](2.0400002,-3.2016668){$x_2$}
				\rput[bl](3.9866667,-3.1216667){$x_3$}
				\rput[bl](2.1733334,-3.9483335){$y_1$}
				\rput[bl](4.12,-3.9483335){$y_2$}
				\rput[bl](6.1733336,-3.8816667){$y_3$}
				\rput[bl](0.9600001,-3.6283333){$G_1$}
				\rput[bl](3.0,-3.5883334){$G_2$}
				\rput[bl](5.04,-3.5616667){$G_3$}
				\psdots[linecolor=black, fillstyle=solid, dotstyle=o, dotsize=0.3, fillcolor=white](6.2533336,-3.4416666)
				\psdots[linecolor=black, fillstyle=solid, dotstyle=o, dotsize=0.3, fillcolor=white](7.8533335,-3.4416666)
				\psdots[linecolor=black, dotsize=0.1](6.6533337,-3.4416666)
				\psdots[linecolor=black, dotsize=0.1](7.0533333,-3.4416666)
				\psdots[linecolor=black, dotsize=0.1](7.4533334,-3.4416666)
				\rput[bl](9.6,-3.0816667){$x_n$}
				\rput[bl](11.826667,-3.8683333){$y_n$}
				\rput[bl](9.586667,-3.9483335){$y_{n-1}$}
				\rput[bl](8.533334,-3.6016667){$G_{n-1}$}
				\rput[bl](7.4,-3.1616666){$x_{n-1}$}
				\rput[bl](10.613334,-3.575){$G_n$}
				\end{pspicture}
			}
\end{center}
\caption{Chain of $n$ graphs $G_1,G_2, \ldots, G_n$.} \label{chain-n}
\end{figure}

%\begin{figure}
%\includegraphics[width=0.95\textwidth, height = 3cm]{chain-n.png}
%\caption{Chain of $n$ graphs $G_1,G_2, \ldots, G_n$.} 
%\label{chain-n}
%\end{figure}
Besides being studied in exact graph classes, super domination has also been 
studied for different graph products. In particular, Dettlaff et al.~\cite{Dett} 
have studied the super domination number of lexicographic products and joins 
and also shown that determining the super domination number of a graph is  
$\mathcal{NP}$-hard. Klein et al.~\cite{Kle} 
have studied Cartesian products and (usual) corona products.

\paragraph{Our results.} 
In this paper, we continue the study of the super domination number of a graph, 
started in~\cite{Nima1,Nima2}.
First in Section~\ref{Sec:NeighbourCor}, we present a key lemma which will be used throughout this paper. In Section~\ref{Sec:NeighbourCor}, we find the exact value of the super 
domination number of the neighbourhood corona product of two graphs. In 
Section~\ref{Sec:rGlue}, we present tight lower and upper bounds on the 
$r$-gluing of two graphs and provide infinite families of graphs attaining these 
bounds. We study the super domination number of the  Haj\'{o}s sum of two graphs 
and find tight upper and lower bounds for it, together with infinite families 
of examples attaining the bounds, in Section~\ref{Sec:Hajos}. Finally,  in 
Section~\ref{Sec:SPsetNumber}, we count exactly the number of minimum size 
super dominating sets of some graph classes.

\section{Key lemma}\label{Sec:keylemma}

We introduce a technical key lemma for 
analysing super dominating sets. It will be needed
in most of the following results.

\begin{lemma}
\label{Lem:Partition}
Let $S$ be a super dominating set in a graph $G$. 
\begin{itemize}
\item[(a)]
Then there is a super dominating set $S'$ 
with same cardinality with $ \overline{S} \subseteq S' $
and $\overline{S'}\subseteq S $.

Furthermore, there is a bijective function
$ f: \overline{S'} \to \overline{S} $ so that $ f(a)=b$ holds
if and only if 
$a$ super dominates $b$ 
for the super dominating set $ S $ and
$b$ super dominates $a$ 
for the super dominating set $ S' $.
\item[(b)]
Let $ D =  S \, \cap \, S' $.  
Then $V(G)$ can be partitioned as 
$ V(G) = \overline{S} \, \dot{\cup} \, \overline{S'} 
\, \dot{\cup} \, D $, where it holds that
%$ D \subseteq S $ and $ D \subseteq S' $ 
$ S = \overline{S'} \, \dot{\cup} \, D $,  
$ S'=  \overline{S} \, \dot{\cup} \, D $.  

\item[(c)]
Let $S$ have cardinality $|V(G)|/2$. Then 
$\overline{S}$ is a super dominating set with the same cardinality
and it holds that $ \overline{S} = S' $ and $ \overline{S'} = S $. 

Furthermore, each vertex in $S$ super dominates exactly one vertex in $\overline{S}$
and vice versa, i.e.,
the function $f: S \to S' $ from (a) is uniquely determined. 
\end{itemize}
\end{lemma}

\pagebreak[3]

\begin{proof}
\begin{itemize}
\item[(a)]
Let $S\subseteq V(G)$ be a super dominating set in $G$. 
Then for each $b\in \overline{S}$ there exists a 
vertex $a\in S$ such that $N(a)\cap \overline{S}=\{b\}$. 
We construct the new super dominating set $ S'$ by replacing each
$ a \in S $ by the corresponding $ b \in \overline{S} $. 
Denote $A \subseteq S $ as the set of all vertices $ a $
removed from $S$ during this process. 
Let $ f: A \to \overline{S} $ be the corresponding function 
with $ f(a) = b $ for each $ a \in A $ and $ b \in \overline{S} $ super dominated by $a$.
By construction, $A = \overline{S'} $, $ |S| = |S'| $ and 
$|\overline{S}| = |\overline{S'}| $ hold.
As each vertex $ a \in \overline{S'}$ can super dominate only one vertex
$ b \in \overline{S}$, $f$ is well defined.
By construction, for each $b$ we have only one $a$, i.e., 
$f$ is also injective. Because of 
$ |\overline{S'}| = |\overline{S}| $, it is also bijective.

\noindent
{\bf Claim:} 
$S'$ is a super dominating set in $G$. 

{\bf Proof (Claim):} 
Let $ f(a)=b$ with $a \in \overline{S'} $, $ b \in \overline{S}$. 
Assume that $ b $ does not super dominate $a $.  
Then $ b $ is also adjacent to another $a' \in \overline{S'}$.
Thus, there is a $b' \in \overline{S}$ with $ f(a') = b' $.
This would mean that $a'$ super dominates $b'$, but it is also
adjacent to $b$, which is a contradiction.
Thus, $b$ super dominates $a$, and the claim follows.

%and that $ E = \overline{S'}$. 
%We know that for each $u'\in E$, 
%there exists a vertex $u\in \overline{S}$ such that $N(u') \, \cap \,
%\overline{S}=\{u\}$. Consider now $N(u)$. If we have $v\in N(u)\cap 
%E$, then, by definition of $E$, we have $|N(v)\cap \overline{S}|=1$. Thus, $N(v)\cap 
%\overline{S}=\{u\}$. 
%However, now $v=u'$ since we selected exactly one vertex of $E$ for each 
%vertex in $\overline{S}$. Therefore, $N(u)\cap E=\{u'\}$. Hence, $S'$ is a 
%super dominating set by definition.
\item[(b)] 
This follows easily from (a).
\item[(c)] 
Here $D = S \cap S' = \emptyset $ holds.
%Let $ f(a)=b$ with $a \in S $, $ b \in S'$, and $a$ super dominates $b$. 
%Assume that also $ a' \in S $ super dominates $b \in S' $.  
%Then there is a $b' \in S' \setminus \{b\} $ with $ f(a') = b' $.
%This would mean that $a'$ super dominates $b'$ 
%but it is also adjacent to $b$, which is a contradiction.
From (b) it follows that $S=\overline{S'}$ and $S'=\overline{S}$.
Since each vertex can super dominate at most one other vertex, each
vertex in $S$ (and each vertex in $S'$) super dominates exactly one vertex. 
Thus, $f$ is uniquely determined.
\qed
\end{itemize}
\end{proof}

\section{Super domination number of neighbourhood corona product of two graphs}\label{Sec:NeighbourCor}

In this section, we study the super domination number of the neighbourhood 
corona product of two graphs. 
Let $G$ and $H$ be two graphs of orders $n,m\in \N$, respectively. 
It is clear that 
\begin{eqnarray}
\label{eqmn}
\gamma_{sp}(G \star H) & \leq & \gamma_{sp}(G)+nm,
\end{eqnarray}
since if we consider  all vertices of all copies of $H$ in our super dominating 
set, then we only need to find a super dominating set for $G$ and the claim 
follows by the definition of a super dominating set. 

%Before we find a better upper bound for the super domination number of 
%neighbourhood corona product of two graphs, we need the following lemma.

Later (in Corollary~\ref{corneighbour})
we will show that, under some conditions, it holds that
$\gamma_{sp}(G \star H) = n(\gamma_{sp}(H)+1)$.
The following proposition shows that this is an improvement
over the trivial upper bound from~\eqref{eqmn} for $m\geq2$.

\begin{proposition}
Let $G$ and $H$ be two connected graphs of orders $n$ and $m\neq 1$, respectively. Then
$$n(\gamma_{sp}(H)+1)< \gamma_{sp}(G)+nm.$$
\end{proposition}

\begin{proof}
By Theorem~\ref{thm-1}, we know that $\gamma_{sp}(H) \leq m-1$. So we have 
$n (\gamma_{sp}(H) +1) = n \gamma_{sp}(H) +n  \leq nm$. 
%Hence, we have $ \gamma_{sp}(H) +n \leq nm$. 
Hence, we have $n(\gamma_{sp}(H) +1) < nm +\gamma_{sp}(G)$ and therefore, we have the result.
\qed
\end{proof}

Next we present a tight upper bound for $\gamma_{sp}(G \star H)$.

\begin{theorem}\label{upperstar}
Let $G$ and $H$ be two connected graphs of orders $n,m \in \N$. Then	
$$\gamma_{sp}(G \star H)\leq n(\gamma_{sp}(H)+1).$$
\end{theorem}

\begin{proof}
%For a given $ w \in V(G) $, consider the copy $H_w$ of $H$, and 
%let $S_w$ be a super dominating set for this $H_w$. To create a 
%super dominating set for $G \star H$, we put all vertices of $G$ in a set, 
%say $S$. Next, for each $H_w$, we consider $S_w$ and place also these 
%vertices in $S$. 
%
%Let $u\in \overline{S}$. Then $ u \in V(H_w) $  for some copy $H_w$. 
%Since there exists a vertex $v\in S_w $  such that 
%$N_{H_w}(v)\cap \overline{S_w}=\{u\}$ 
%and $N_G(v)\subseteq S$, $N_{H_x}(v) = \emptyset $ for all $ x \in V(G) 
%\setminus \{w\} $, it follows that
%$N_{G \star H}(v)\cap \overline{S}=\{u\}$ for this $v$.
%Thus, the set $S$ is a super dominating set for $G \star H$ with 
Let $S_H$ be a super dominating set for $H$. We create a set $S$ for $G \star H$ by 
placing all vertices of $G$ in $S$. Next, for each copy of $H$, we place all
vertices corresponding to $S_H$ in $S$. In the following, we show 
that $S$ is a super dominating set for $G\star H$.  

Let $u\in \overline{S}$. Then $ u \in V(H') $  for some copy of $H$. Let $ v\in 
V(H')$ super dominate $u$ in $H'$. If $v$ does not super dominate $u$ in 
$G\star H$, then $v$ has another neighbour in $\overline{S}\cap (V(G\star   
H)\setminus V(H'))$. However, this is not possible, since $V(G)\subseteq S$ and 
there are no edges between different copies of $H$. Thus, $S$  
is a super dominating set for $G \star H$ with
cardinality $n \gamma_{sp}(H)+n = n( \gamma_{sp}(H)+1)$ and the assertion follows.
\qed
\end{proof}

In the following theorem, we show that the upper bound in Theorem~\ref{upperstar} 
is actually the super domination number of $G \star H$, when $G$ and 
$H\neq K_1$ are connected graphs and $\gamma_{sp}(H)<m-1$ or it holds that $ H =K_m $.

\begin{theorem}\label{notbiggerupperstar}
Let $G$ and $H$ be two connected graphs of orders $n$ and $m\not= 1$, respectively, 
where it holds that $\gamma_{sp}(H)<m-1$ or it holds that $ H =K_m $.
Then	$$ n(\gamma_{sp}(H)+1)\leq \gamma_{sp}(G \star H). $$
\end{theorem}

\begin{proof}
Let $G=(V(G),E(G))$ and $H=(V(H),E(H))$. 
Let $S$ be a super dominating set for $G \star H$.
Let $ H_w$ be the copy of $ H $ corresponding to $ w \in V(G) $.
Then it holds that $V (G \star H) = V(G) \cup \bigcup_{w \in V(G)} V (H_w) $. 

\medskip

\noindent
{\bf Claim 1:} In each copy of $H$, the set $S$ has at least
$\gamma_{sp}(H)$ vertices. 

\noindent{\bf Proof (Claim 1):} 
For a given $ w \in V(G) $, consider the copy $H_w$ of $H$. 
Assume that $r := |S \cap V(H_w)| < \gamma_{sp}(H)$. 
As $H$ is connected, it follows by Theorem~\ref{thm-1} 
that $r < \gamma_{sp}(H)\leq m-1$. Thus, this copy
of $H$ has at least two vertices which are contained in $ \overline{S} $.

As $ S \cap V(H_w) $ is \textit{not} a super dominating set of $H_w$,
there must exist a vertex $u\in \overline{S} \cap V(H_w)$, 
such that there does \textit{not} exist a vertex $v \in S\cap V(H_w) $
for which $N_{H_w}(v) \cap \overline{S}=\{u\}$.
On the other hand, for this $u\in \overline{S} \cap V(H_w)$, 
there exists a vertex $v \in (V(G \star H) \setminus V(H_w)) \cap S $ 
such that $N_{G \star H}(v) \cap \overline{S}=\{u\}$. 

As $v \in V(G \star H) \setminus V(H_w) $ holds, 
we have two possibilities. Either 
$v\in V(G)$ or $v\in V(H_x)$ for some $ x \in V(G) \setminus \{w\} $.  

Firstly, $v\in V(G)$ cannot hold, since then all vertices in $V(H_w) \setminus \{u\}$ 
would lie in~$S$, because $v$ is adjacent to all of them. Thus, $ r=m-1$ holds, 
a contradiction.

Secondly, $v\in V(H_x)$ cannot hold, as there are 
no adjacent vertices between different copies of $H$.
Thus, we have a contradiction again. 

Claim 1 follows.

\medskip

In the following, for a given $ w \in V(G) $,
we define a block of vertices $ B_w(G \star H ) = \{ w \} \cup V(H_w) $
(or shortly $ B_w $). 
The blocks $B_w$ clearly partition the vertex set $V(G\star H)$.
%Clearly, it holds that $ V (G \star H ) = \bigcup_{w \in V(G)} B_w( G \star H) $.

For the given super dominating set $ S $, we define a block as 
\textit{over-satisfied},
if it has more than $ \gamma_{sp}(H) + 1 $ vertices (note that for $m=1$, i.e., 
$H=K_1$, a block can never be over-satisfied), 
as \textit{satisfied}, if it has exactly $ \gamma_{sp}(H) + 1 $ vertices, 
and as \textit{under-satisfied}, if it has 
less than $ \gamma_{sp}(H) + 1 $ vertices.
Note that by Claim 1, an under-satisfied block $B_w $ 
has always exactly $ \gamma_{sp}(H) $ vertices in $ S $
and that $ w $ lies in $ \overline{S} $.

\medskip

\noindent {\bf Claim 2:} 
%Let $ w \in V(G) $ and let $ B_w $ be under-satisfied. Let $ x \in V(G) 
%\setminus \{w\} $ be chosen such that $ \{w,x\} \in E(G)$
%and one vertex in $ B_x $ super dominates $ w \in V(G) $.
%Then $B_x$ is an over-satisfied block. 
Let $ w \in V(G) $ and let $ B_w $ be under-satisfied. Let $ x \in V(G) 
\setminus \{w\} $ be chosen so that
a vertex in $ B_x $ super dominates $ w \in V(G) $ (and $ \{w,x\} \in E(G)$).
Then $B_x$ is an over-satisfied block.

\noindent{\bf Proof (Claim 2):} 
As we have mentioned above, in an under-satisfied block $B_w$ there are 
exactly $\gamma_{sp}(H)$ vertices in $S$ and $ w\in \overline{S}$. Thus, at 
least one vertex $ y \in V(H_w) $ lies in $ \overline{S} $.

%By Claim 1, $ S $ has at least $ \gamma_{sp}(H) $ vertices in $ H_w$, and as 
%$B_w$ is under-satisfied, it has even exactly $ \gamma_{sp}(H) $ vertices in $ 
%H_w$.  It follows that $ w \in \overline{S} $
%and that at least one vertex $ y \in V(H_w) $ lies in $ \overline{S} $.

First, assume that $ x \in \overline{S} $. 
On the one hand, each vertex $ z \in V(H_w) $ 
adjacent to $ y $ is also adjacent to $ x $, 
i.e., $z$ cannot super dominate $y$.  
On the other hand, each vertex $ z \notin V(H_w) $ adjacent to $ y$
is also adjacent to $w$,
i.e., $z$ cannot super dominate $y$ either.  
In summary, $y$ cannot be super dominated by another 
vertex of $S$.
Thus, by contradiction $ x \in S $  follows.

%At least two vertices $ y,z \in V(H_w) $ lie in $ \overline{S} $.

%First, assume that $ x \in \overline{S} $. 
%On the one hand, each vertex $ a \in V(H_w) $ 
%adjacent to $ y $ or $z$ is also adjacent to $ x $, 
%i.e., $a$ cannot super dominate $y$ or $z$.  
%On the other hand, each vertex $ a \notin V(H_w) $ adjacent to $ y$
%or $z$ is also adjacent to the other vertex of those two,
%i.e., $a$ cannot super dominate $y$ or $z$ either.  
%In summary, neither $y$ nor $z$ can be super dominated by another 
%vertex of $S$.
%Thus, by contradiction $ x \in S $  follows.

We continue by dividing the proof into two cases.

\begin{description}
\item[{\bf Case 1:}] $ \gamma_{sp}(H) <m-1 $. 

Assume that there are $ b,c \in V(H_x) \cap \overline{S} $. 
Each vertex $ z \in V(H_x) $, 
adjacent to $ b $ or $c$, is also adjacent to $ w $, 
and if $z\not\in V(H_x) $ is adjacent to $b$, then it is also adjacent to $c$ and vice versa.
So neither $b$ nor $c$ can be super dominated by another vertex.
Thus, there are at least $ \gamma_{sp}(H)+1 $
vertices in $ V(H_x) \cap S $, 
and $ B_x $ is over-satisfied.

\item[{\bf Case 2:}] $ H =K_m $. 

Assume that there is a $ b \in V(H_x) \cap \overline{S} $. 
By the choice of $ x \in V(G) $, one vertex of $B_x$ super dominates 
$w$. This cannot be $x$, as $x$ is also adjacent to $y$.
On the other hand, as $ H=K_m$, each other vertex in $B_x$
super dominating $w$ is also adjacent to $b$, a contradiction.
Thus, there does not exist such $b$, and $B_x$ is over-satisfied. 

\end{description}

Claim 2 follows.

\medskip

\noindent
{\bf Claim 3:} 
No two under-satisfied blocks $ B_w $ and $ B_{w'} $ 
can be assigned to the same over-satisfied block $ B_x $.

\noindent{\bf Proof (Claim 3):} 
Assume that this does not hold.
If there is a $ d \in B_x $ super dominating $ w \in V(G) $,
then it super dominates also $w' \in V(G) $ and vice versa. 
On the other hand,
no vertex can super dominate two vertices. Thus, we have a 
contradiction, and Claim 3 follows.

\medskip

By Claim $1$, for each under-satisfied block $B_w$, where $ w \in V(G) $,
we have exactly $ \gamma_{sp}(H) $ vertices in $ S $. By Claims 2 and 3,  we
have a corresponding over-satisfied block with at least $ \gamma_{sp}(H)+2 $ vertices in $S$.

In total, we have at least $ \gamma_{sp}(H) +1 $ vertices in $S$ for each block.
Thus, each super dominating set $ $ has cardinality of at least
$n(\gamma_{sp}(H)+1)$ in $G \star H$. 
\qed
\end{proof}

By using Theorems~\ref{upperstar} and~\ref{notbiggerupperstar}, we have the 
following result which gives us the exact value of the super domination number 
of $G \star H$.

\begin{corollary}\label{corneighbour}
Let $G$ and $H$ be two connected graphs of orders $n$ and $m\not= 1$, respectively,
with $\gamma_{sp}(H) < m-1$ or $H=K_m$. Then	
$$\gamma_{sp}(G \star H)= n(\gamma_{sp}(H)+1).$$
\end{corollary}

In the following example, we show that ``$\gamma_{sp}(H)< |V(H)|-1$ or 
$H=K_m$'' is a 
necessary condition for Theorem \ref{notbiggerupperstar} and Corollary 
\ref{corneighbour}.

\begin{example}	
Consider the graph $G$ from Figure~\ref{P3starG}. One can easily check 
that $\gamma_{sp}(G)=3$ and the set of black vertices in $P_3\star G$ is a super 
dominating set for $G$. We have $\gamma_{sp}(P_3 \star G) = 11 < 12 = 
3 \cdot (\gamma_{sp}(G)+1).$
\end{example}

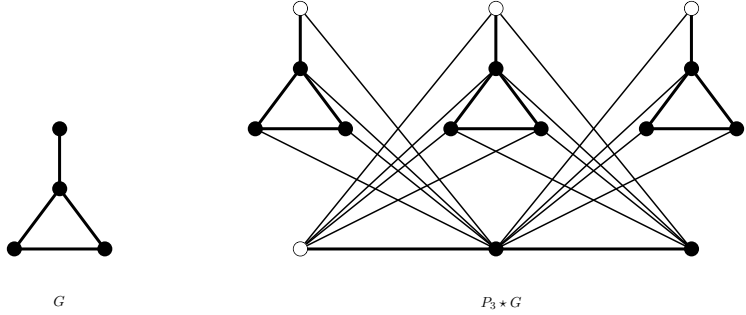
\begin{figure}
\begin{center}
\psscalebox{0.5 0.5}
{
\begin{pspicture}(0,-5.0993056)(19.59423,3.1020834)
\psdots[linecolor=black, dotsize=0.4](7.7971153,1.3006946)
\psdots[linecolor=black, dotsize=0.4](6.5971155,-0.2993054)
\psdots[linecolor=black, dotsize=0.4](8.997115,-0.2993054)
\psline[linecolor=black, linewidth=0.08](7.7971153,1.3006946)(6.5971155,-0.2993054)(8.997115,-0.2993054)(7.7971153,1.3006946)(7.7971153,2.9006946)(7.7971153,2.9006946)
\psdots[linecolor=black, dotsize=0.4](12.997115,-3.4993055)
\psdots[linecolor=black, dotsize=0.4](18.197115,-3.4993055)
\psline[linecolor=black, linewidth=0.08](7.7971153,-3.4993055)(18.197115,-3.4993055)(18.197115,-3.4993055)
\psdots[linecolor=black, dotsize=0.4](1.3971153,-1.8993055)
\psdots[linecolor=black, dotsize=0.4](0.19711533,-3.4993055)
\psdots[linecolor=black, dotsize=0.4](2.5971153,-3.4993055)
\psline[linecolor=black, linewidth=0.08](1.3971153,-1.8993055)(0.19711533,-3.4993055)(2.5971153,-3.4993055)(1.3971153,-1.8993055)(1.3971153,-0.2993054)(1.3971153,-0.2993054)
\psdots[linecolor=black, dotsize=0.4](1.3971153,-0.2993054)
\psdots[linecolor=black, dotsize=0.4](12.997115,1.3006946)
\psdots[linecolor=black, dotsize=0.4](11.797115,-0.2993054)
\psdots[linecolor=black, dotsize=0.4](14.197115,-0.2993054)
\psline[linecolor=black, linewidth=0.08](12.997115,1.3006946)(11.797115,-0.2993054)(14.197115,-0.2993054)(12.997115,1.3006946)(12.997115,2.9006946)(12.997115,2.9006946)
\psdots[linecolor=black, dotsize=0.4](18.197115,1.3006946)
\psdots[linecolor=black, dotsize=0.4](16.997116,-0.2993054)
\psdots[linecolor=black, dotsize=0.4](19.397116,-0.2993054)
\psline[linecolor=black, linewidth=0.08](18.197115,1.3006946)(16.997116,-0.2993054)(19.397116,-0.2993054)(18.197115,1.3006946)(18.197115,2.9006946)(18.197115,2.9006946)
\psline[linecolor=black, linewidth=0.04](7.7971153,-3.4993055)(11.797115,-0.2993054)(11.797115,-0.2993054)
\psline[linecolor=black, linewidth=0.04](7.7971153,-3.4993055)(14.197115,-0.2993054)(14.197115,-0.2993054)
\psline[linecolor=black, linewidth=0.04](7.7971153,-3.4993055)(12.997115,1.3006946)
\psline[linecolor=black, linewidth=0.04](7.7971153,-3.4993055)(12.997115,2.9006946)(18.197115,-3.4993055)(12.997115,1.3006946)(12.997115,1.3006946)
\psline[linecolor=black, linewidth=0.04](14.197115,-0.2993054)(18.197115,-3.4993055)(18.197115,-3.4993055)
\psline[linecolor=black, linewidth=0.04](11.797115,-0.2993054)(18.197115,-3.4993055)(18.197115,-3.4993055)
\psline[linecolor=black, linewidth=0.04](6.5971155,-0.2993054)(12.997115,-3.4993055)(12.997115,-3.4993055)
\psline[linecolor=black, linewidth=0.04](8.997115,-0.2993054)(12.997115,-3.4993055)(12.997115,-3.4993055)
\psline[linecolor=black, linewidth=0.04](7.7971153,1.3006946)(12.997115,-3.4993055)(12.997115,-3.4993055)
\psline[linecolor=black, linewidth=0.04](12.997115,-3.4993055)(7.7971153,2.9006946)(7.7971153,2.9006946)
\psline[linecolor=black, linewidth=0.04](12.997115,-3.4993055)(16.997116,-0.2993054)(16.997116,-0.2993054)
\psline[linecolor=black, linewidth=0.04](12.997115,-3.4993055)(18.197115,1.3006946)(18.197115,1.3006946)
\psline[linecolor=black, linewidth=0.04](12.997115,-3.4993055)(18.197115,2.9006946)(18.197115,2.9006946)
\psline[linecolor=black, linewidth=0.04](12.997115,-3.4993055)(19.397116,-0.2993054)(19.397116,-0.2993054)
\psdots[linecolor=black, dotstyle=o, dotsize=0.4, fillcolor=white](18.197115,2.9006946)
\psdots[linecolor=black, dotstyle=o, dotsize=0.4, fillcolor=white](12.997115,2.9006946)
\psdots[linecolor=black, dotstyle=o, dotsize=0.4, fillcolor=white](7.7971153,2.9006946)
\psdots[linecolor=black, dotstyle=o, dotsize=0.4, fillcolor=white](7.7971153,-3.4993055)
\rput[bl](1.2371154,-5.019305){$G$}
\rput[bl](12.5971155,-5.0993056){$P_3\star G$}
\end{pspicture}
}
\end{center}
\caption{Graphs $G$ and $P_3\star G$, respectively.} \label{P3starG}
\end{figure}

%\begin{figure}
%\centering
%\begin{subfigure}{0.5\linewidth}
%\includegraphics[width=0.95\textwidth, height = 3cm]{gstar.png}
%\caption{$G$.}
%\end{subfigure}
%\begin{subfigure}{0.5\linewidth}
%\includegraphics[width=0.95\textwidth, height = 3cm]{p3starg.png}
%\caption{$P_3 \star G$.}
%\end{subfigure}
%\caption{Counterexample to Corollary~\ref{corneighbour}, if $ \gamma_{sp}(H)=m-1$.} 
%\label{P3starG}
%\end{figure}

Interestingly, the value of Corollary~\ref{corneighbour} is equal to the 
value attained for the (usual) corona product of two graphs in~\cite[Theorem 10]{Kle}. 

We end this section by determining the super domination number of the neighbourhood 
corona product of some specific graphs.
These results follow directly from Theorems~\ref{thm-2} (a)-(d) and~\ref{Firend-thm} and 
Corollary~\ref{corneighbour}.

\begin{example} Let $n,m\in\N$. 
\begin{itemize}
\item[(a)]	$\gamma_{sp}( C_{n}\star P_{2m} )=n(m+1)$.
\item[(b)]	$\gamma_{sp}( P_{2n}\star C_{4m} )=4nm+2n$.
\item[(c)]	$\gamma_{sp}( C_{2n} \star K_m  )=2nm$ for $ m \geq 2 $.
\item[(d)]	$\gamma_{sp}( C_n \star K_{2,3}  )=4n$.
\item[(e)]	$\gamma_{sp}( P_n \star F_n )=n^2+2n$.
\end{itemize}
\end{example}

\section{Super domination number of $\mathbf{r}$-gluing of two graphs}	\label{Sec:rGlue}

In this section, we give exact upper and lower bounds for the super domination 
number of $r$-gluing of two graphs.

Since for every two graphs $G_1$ and $G_2$, $G_1\cup_{K_0}G_2$ is 
their disjoint union, by Proposition~\ref{pro-disconnect}, we have the following 
result: 
$$\gamma _{sp}(G_1)+\gamma _{sp}(G_2) -0\leq  
\gamma _{sp}(G_1\cup_{K_0}G_2)
\leq \gamma _{sp}(G_1)+\gamma _{sp}(G_2).$$
Also  $G_1\cup_{K_1}G_2$ is same as the chain 
of two graphs  and by  Theorem~\ref{chain}, we have the following result:
$$\gamma _{sp}(G_1)+\gamma _{sp}(G_2) -1\leq  
\gamma _{sp}(G_1\cup_{K_1}G_2)
\leq \gamma _{sp}(G_1)+\gamma _{sp}(G_2).$$
In Theorem~\ref{G1G2Kr-thm}, we consider the $r$-gluing of two graphs,
and we generalize this result for each $r$-clique. 

\medskip

\begin{theorem}\label{G1G2Kr-thm}
Let $G_1=(V_1,E_1)$ and $G_2=(V_2,E_2)$ be two graphs with clique number at 
least $r \in \N$. Then, 
$$\gamma_{sp}(G_1)+\gamma_{sp}(G_2)-r\leq \gamma_{sp}(G_1\cup_{K_r} G_2)\leq \gamma_{sp}(G_1)+\gamma_{sp}(G_2).$$
\end{theorem}

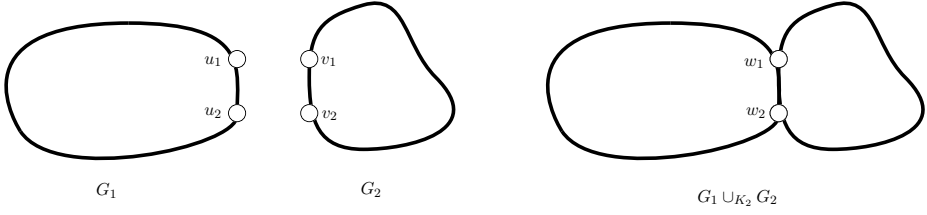
\begin{figure}
\begin{center}
\psscalebox{0.6 0.6}
{
\begin{pspicture}(0,-5.784467)(20.414343,-1.2334015)
\psbezier[linecolor=black, linewidth=0.08](5.1660495,-3.78161)(5.1660495,-4.4673243)(1.1660494,-5.324467)(0.36604938,-4.124467163085938)
\psbezier[linecolor=black, linewidth=0.08](0.36604938,-4.1244674)(0.36604938,-4.1244674)(-1.2339506,-1.7244672)(2.7660494,-1.7244671630859374)
\psbezier[linecolor=black, linewidth=0.08](2.7660494,-1.7244672)(5.5089064,-1.7244672)(5.1660495,-2.7530386)(5.1660495,-3.781610020228795)
\psbezier[linecolor=black, linewidth=0.08](9.56605,-2.924467)(10.766049,-4.1244674)(8.966049,-4.524467)(7.966049,-4.524467163085937)(6.966049,-4.524467)(6.7660494,-3.924467)(6.7660494,-2.924467)(6.7660494,-1.9244672)(7.0045257,-1.5991883)(7.966049,-1.3244672)(8.927573,-1.049746)(8.858943,-2.2173605)(9.56605,-2.924467)
\psdots[linecolor=black, dotstyle=o, dotsize=0.4, fillcolor=white](5.1660495,-2.5244672)
\psdots[linecolor=black, dotstyle=o, dotsize=0.4, fillcolor=white](5.1660495,-3.7244673)
\psdots[linecolor=black, dotstyle=o, dotsize=0.4, fillcolor=white](6.7660494,-2.5244672)
\psdots[linecolor=black, dotstyle=o, dotsize=0.4, fillcolor=white](6.7660494,-3.7244673)
\psbezier[linecolor=black, linewidth=0.08](19.96605,-2.924467)(21.16605,-4.1244674)(19.366049,-4.524467)(18.366049,-4.524467163085937)(17.366049,-4.524467)(17.16605,-3.924467)(17.16605,-2.924467)(17.16605,-1.9244672)(17.404526,-1.5991883)(18.366049,-1.3244672)(19.327574,-1.049746)(19.258942,-2.2173605)(19.96605,-2.924467)
\psdots[linecolor=black, dotstyle=o, dotsize=0.4, fillcolor=white](17.16605,-2.5244672)
\psdots[linecolor=black, dotstyle=o, dotsize=0.4, fillcolor=white](17.16605,-3.7244673)
\psdots[linecolor=black, dotstyle=o, dotsize=0.4, fillcolor=white](17.16605,-2.5244672)
\psdots[linecolor=black, dotstyle=o, dotsize=0.4, fillcolor=white](17.16605,-3.7244673)
\rput[bl](4.446049,-2.684467){$u_1$}
\rput[bl](4.446049,-3.8244672){$u_2$}
\rput[bl](7.0260496,-2.684467){$v_1$}
\rput[bl](7.0460496,-3.8844671){$v_2$}
\rput[bl](16.40605,-2.7244673){$w_1$}
\rput[bl](16.446049,-3.8244672){$w_2$}
\rput[bl](2.0460494,-5.6044674){$G_1$}
\rput[bl](7.9060493,-5.584467){$G_2$}
\rput[bl](15.36605,-5.784467){$G_1\cup_{K_2}G_2$}
\psbezier[linecolor=black, linewidth=0.08](17.16605,-3.78161)(17.16605,-4.4673243)(13.166049,-5.324467)(12.36605,-4.124467163085938)
\psbezier[linecolor=black, linewidth=0.08](12.36605,-4.1244674)(12.36605,-4.1244674)(10.766049,-1.7244672)(14.766049,-1.7244671630859374)
\psbezier[linecolor=black, linewidth=0.08](14.766049,-1.7244672)(17.508907,-1.7244672)(17.16605,-2.7530386)(17.16605,-3.781610020228795)
\psdots[linecolor=black, dotstyle=o, dotsize=0.4, fillcolor=white](17.16605,-2.5244672)
\psdots[linecolor=black, dotstyle=o, dotsize=0.4, fillcolor=white](17.16605,-3.7244673)
\end{pspicture}
}
\end{center}
\caption{Graphs $G_1$, $G_2$ and $G_1\cup_{K_2}G_2$, respectively.} \label{G1G2K2}
\end{figure}

%\begin{figure}
%\begin{subfigure}{0.33\linewidth}
%\includegraphics[width=0.95\textwidth, height = 3cm]{g1union.png}
%\caption{$G_1$.}
%\end{subfigure}
%\begin{subfigure}{0.33\linewidth}
%\includegraphics[width=0.95\textwidth, height = 3cm]{g2union.png}
%\caption{$G_2$.}
%\end{subfigure}
%\begin{subfigure}{0.33\linewidth}
%\includegraphics[width=0.95\textwidth, height = 3cm]{g1uniong2.png}
%\caption{$G_1\cup_{K_2}G_2$.}
%\end{subfigure}
%\caption{Example for $2$-gluing of two graphs $G_1$ and $G_2$.}
%\label{G1G2K2}
%\end{figure}

\begin{proof}
Let $S_1$ and $S_2$ be two minimum size super dominating sets 
for  $G_1$ and $G_2$, respectively. Let the vertex sets $V'=\{v'_i\mid 
1\leq i\leq r\}$ and $V''=\{v''_i\mid 1\leq i\leq r\}$ form $r$-cliques in $G_1$ 
and $G_2$, respectively. Furthermore, 
let us create a vertex $w_i$ by identifying 
the vertices $v'_i$ and $v''_i$ for each $i\in \{1,2,\dots, r\}$ (see Figure~\ref{G1G2K2} for 
the case $r=2$) and denote $W=\{w_i\mid 1\leq i\leq r\}$. Denote 
by $G$ the graph which we obtain in this way. When $r=1$, the claim follows from 
Theorem~\ref{chain}. Hence, we assume from now on that $r\geq2$. 
We divide the proof into two parts, namely the lower and the upper bound.

\noindent
{\bf Lower bound} $\mathbf{\gamma_{sp}(G_1)+\gamma_{sp}(G_2)-r\leq 
\gamma_{sp}(G_1\cup_{K_r} G_2)}$.
\medskip
%\leq \gamma_{sp}(G_1)+\gamma_{sp}(G_2).$$

%Let us then consider  the lower bound for $\gamma _{sp}(G_1\cup_{K_r}G_2)$. 

\noindent We use Lemma~\ref{Lem:Partition} to show that we can assume two
cases for the minimum size super dominating set $S$:

\begin{itemize}
\item $ W \subseteq S $,
\item $ |W \cap S| = r-1$ and one vertex of $W$, say $w_1$, super dominates the 
missing vertex, say $w_r$.
\end{itemize}
To show this, assume that $ |S \cap W| \leq r-1$ and, without loss of generality, 
$ w_r \notin S $. By Lemma~\ref{Lem:Partition}, 
there is another minimum size super dominating set $S'$, the corresponding 
set $D = S \cap S' $ and
a bijective function $ f: \overline{S'} \to \overline{S} $ with the mentioned
characteristics.
If no vertex of $W \cap S $ super dominates another vertex,
then $ W \cap S \subseteq D $ follows, and we can replace $ S $
by $ S' $, where $ W \subseteq S' $ holds.
On the other hand, if 
$ W \nsubseteq S $ and
one vertex of $ W \cap S $ super dominates another vertex, 
then $ |W \cap S| = r-1$ and, 
%one vertex 
without loss of generality,
$w_1$ super dominates the missing vertex $w_r$. 
This finishes the proof of this assumption.
We split our considerations into three subcases.

Here we split our considerations into two subcases.
\begin{itemize}
\item[(i)] $W\subseteq {S}$. 

Let
$$S_1=\left( S\cap V_1 \right) \cup V'$$
and
$$S_2=\left( S\cap V_2 \right) \cup V''.$$

As $ \overline{S_i}\subseteq \overline{S}$, if $v\in\overline{S_i}$ is super 
dominated by some vertex $u$ in $S$, then it is super dominated by a vertex corresponding to 
$u$ in $S_i$.
%As any vertex in $\overline{S_i}$, for $i\in\{1,2\}$, is super 
%dominated by the corresponding vertex,
Thus, $S_1$ is a super dominating set for $G_1$, and $S_2$ is a super dominating set 
for $G_2$. 

As we replace $r$ vertices from $S$ 
by $2r$ other vertices to reach $S_1 \cup S_2 $, 
it follows that $|S_1|+|S_2| -r\leq \gamma_{sp}(G_1\cup_{K_r} G_2)$.

\item[(ii)] $ | W \cap S | = r-1$.

Let
$$S_1=(S\cap V_1)\cup (V' \setminus \{v'_r\})$$ 
and
$$S_2=(S\cap V_2)\cup (V''\setminus \{v''_r\}).$$

The vertex $v'_r$ is super dominated by $v'_1$, and $v''_r$ is super dominated by 
$v''_1$. 
As $ \overline{S_i}\setminus\{v'_r,v''_r\}\subseteq \overline{S}$, if any other vertex 
$v\in\overline{S_i}$ is super 
dominated by some vertex $u$ in $S$, 
then it is super dominated by a vertex corresponding to 
$u$ in $S_i$.
%As any other vertex in $\overline{S_i}$, for $i\in\{1,2\}$, is super 
%dominated by the corresponding vertex, $S_1$ is a super dominating set 
%for $G_1$, and $S_2$ is a super dominating set for $G_2$. 

As we replace $r-1$ vertices from $S$ 
by $2r-2$ other vertices to reach $S_1 \cup S_2 $, 
it follows that $|S_1|+|S_2| -r\leq \gamma_{sp}(G_1\cup_{K_r} G_2)$.

%$W' \setminus \{w_r\}\subseteq {S}$ and $w_r \in \overline{S}$. Recall that we can with 
%Lemma~\ref{Lem:Partition} either swap the vertices in $S$ so that there are no 
%vertices left in $\overline{S}\cap W'$ (this is considered in case (ii)) or we 
%have (using notation of Lemma~\ref{Lem:Partition}) $|D_1\cap W'|=1$ and 
%$|D_2\cap W'|=1$. Moreover, if the latter is the case, then we may assume that 
%$D_1\cap W'=\{w_1\}$ and $D_2\cap W'=\{w_r\}$. 
%Let us now consider 
%Hence
%$$\gamma _{sp}(G_1)+\gamma _{sp}(G_2) \leq \gamma _{sp}(G_1\cup_{K_r}G_2)+r .$$
\end{itemize}
\medskip

\medskip

\noindent
{\bf Upper bound} $\mathbf{\gamma_{sp}(G_1\cup_{K_r} G_2)\leq 
\gamma_{sp}(G_1)+\gamma_{sp}(G_2)}$.

\medskip

\noindent
As in the proof of the lower bound, 
by Lemma~\ref{Lem:Partition} we can assume two
cases for the minimum size super dominating set $S_1$:
\begin{itemize}
\item $ V' \subseteq S_1 $,
\item $ |V' \cap S_1| = r-1$ and one vertex of $V'$, say $v'_1$, super dominates the 
missing vertex, say $v'_r$.
\end{itemize}

(By symmetry, this holds analogously for $S_2$ and $V''$.)

\begin{itemize}
\item[(i)] $V'\subseteq S_1$ and $V''\subseteq S_2$. 

If for $i\in\{1,2,\dots,r\}$ we have $|N(v'_i)\cap \overline{S_1}|=1$, then we denote 
$\{x_i\}=N(v'_i)\cap \overline{S_1}$. If $|N(v'_i)\cap \overline{S_1}|\neq1$ 
but $|N(v''_i)\cap \overline{S_2}|=1$, then  we denote $\{x_i\}=N(v''_i)\cap 
\overline{S_2}$. Denote $X=\{x_i\mid 1\leq i\leq r, \; \; x_i\ \mbox{exists}\}$. 
(Notice that the vertices $x_i$ are included in $X$ only if they exist.) 

Let $$S=(S_1\cup S_2\cup W\cup X) \setminus (V'\cup V'').$$

%Clearly, all vertices from $ S \setminus W $ super dominate the corresponding
%vertices as $S_1 $ or $ S_2 $.
Clearly, if a vertex $u\in\overline{S}$ was super dominated by a vertex $v$ in $S_1\setminus V'$ or 
in $S_2\setminus V''$, then it is now super dominated by $v\in S\setminus W$.
%The vertices from $X$ are not needed to super dominate any vertices.
Let us then assume that the 
vertex $u\in\overline{S}$ was super dominated by a vertex $v$ in $S_1\cap V'$ or in $S_2\cap V''$. Since $u\not\in X$, we have $v\in S_2\cap V''$. 
We may assume that $v=v''_h$ for some $h$ with $ 1 \leq h \leq r $. 
Thus, $x_h\in X \subseteq S $ and $N_G(w_h)\cap\overline{S}=\{u\}$.  
Thus, $S$ is a super dominating set of $G$. 

%Assume that there exists a vertex $v_1\in \overline{S}$ 
%such that there does not exist any $v_2\in S$ with $N(v_2)\cap 
%\overline{S}=\{v_1\}$. Then $v_1\in \overline{S_i}$ for $i\in\{1,2\}$ and there 
%exists a $v_2\in S_i$ such that $N(v_2)\cap \overline{S_i}=\{v_1\}$. Thus, 
%$w \in W $ exists with $N(w)\cap V_j\neq\emptyset$ for each $j\in \{1,2\}$. However, since 
%$X\subseteq S$, we have $N(w)\cap \overline{S}=\{v_1\}$, this is a contradiction. 

As we replace $2r$ vertices from $S_1 \cup S_2 $ 
by at most $2r$ other vertices to reach $S$, it follows that $|S|\leq\gamma_{sp}(G_1)+\gamma_{sp}(G_2)$, 
as claimed.

\item[(ii)] Either $|S_1 \cap V'| = r-1 $ or $ |S_2 \cap V''| = r-1 $.
%$V'\subseteq S_1$, $V''\setminus \{v''_r\}\subseteq S_2$, $v''_r\in 

%\overline{S_2}$ or $V''\subseteq S_2$, 

%$V' \setminus \{v'_r\}\subseteq S_1$, $v'_r\in \overline{S_1}$.  

Notice that the two cases are essentially identical and we may 
assume, without loss of generality, that 
$ |S_2 \cap V''| = r-1 $.
%$V'\subseteq S_1$, $V'' \setminus \{v'_r\}
%\subseteq S_2$ and $v''_r\in \overline{S_2}$. 

As we have mentioned above, we can assume that $v''_1 $ super dominates $v''_r$ in $S_2$.
Let us have $x_i\in N(v'_i)\cap \overline{S_1}$ for 
each $2 \leq i \leq r $  (if such vertex exists). 
Define $ X = \{x_i\mid 2\leq i\leq r, \; \; x_i\  \mbox{exists} \} $.

Let $$S=(S_1\cup S_2\cup W  \cup X)
\setminus (V'\cup V'').$$ 

%Again, all vertices from $ S \setminus W $ super dominate the corresponding
%vertices and no vertices from $X$ are needed to super dominate any vertices.
Again, if a vertex $u\in \overline{S}$ was super dominated by a vertex $v$ in $S_1\setminus V'$ or in 
$S_2\setminus V''$, then it is now super dominated by $v\in S\setminus W$. If a 
vertex $u\in\overline{S}$ was previously super dominated by a vertex $v$ in $V'$ 
or in $V''$, then $v=v'_1$ since $u\not\in W\cup X$ and vertices in $V''$ could 
super dominate only the vertex $v''_r$. 
Since $v''_1$ super dominates $v''_r$, the 
vertex $w_1$ super dominates $u$. Thus, $S$ is a super dominating set of $G$. 

%$S$ is a super dominating set of $G$ because of the following reasoning.
%Assume that there exists a vertex $v_1\in \overline{S}$ such that there 
%does not exist any $v_2\in S$ with $N(v_2)\cap \overline{S}=\{v_1\}$. Then $v_1$ must 
%be dominated by $v'_i$ in $S_1$ for some $i\in\{1,2,\dots,r\}$. However, if $i\geq2$, 
%then $v_2=x\in S$, a contradiction, and if $i=1$, then $N(w_1)\cap 
%\overline{S}=\{v_2\}$ since $N(w_1)\cap \overline{S}\cap V_2=\emptyset$. 

As we replace $2r-1$ vertices from $S_1 \cup S_2 $ 
by at most $2r-1$ other 
vertices to reach $S$, it follows that $|S|\leq\gamma_{sp}(G_1)+\gamma_{sp}(G_2)$. 

%Since $v'_r\not\in S_2$, we cannot use Lemma~\ref{Lem:Partition} to have 
%$V'\subseteq S_2$. Hence, using the notation of Lemma~\ref{Lem:Partition}, we 
%have $v_i\in D_1$ for some $i\in\{1,2,\dots,r-1\}$ and $v_r\in D_2$. We may 
%%assume that $i=1$. Thus, $N(v_1)\cap \overline{S_2}=\{v_r\}$. 
%Moreover, let us %have $u'_i\in N(u_i)\cap \overline{S_1}$ for 
%each $i\in\{2,\dots,r\}$  (if such %vertex exists). W
%is super dominating in $G$. Since each vertex in 
%Moreover, $S$ clearly has the claimed cardinality.

\item[(iii)] $|S_1 \cap V'| = r-1 $ and $ |S_2 \cap V''| = r-1 $.

%\item[(iii)] $V' \setminus \{v'_r\}\subseteq S_1$, $v'_r\in \overline{S_1}$, 
%$V'' \setminus \{v''_h\}\subseteq S_2$ and $v''_h\in \overline{S_2}$. 

We can assume, without loss of generality, that 
$v'_1 $ super dominates $v'_r$ in $S_1$ and
$v''_h $ super dominates $v''_p$ in $S_2$, where $ 1 \leq h < p \leq r $.

%Since $u_r\not\in S_1$ 
%and $v_h\not\in S_2$, we cannot use Lemma~\ref{Lem:Partition} to have  
%$U'\subseteq S_1$ or $V'\subseteq S_2$. If we reconsider the partition in 
%Lemma~\ref{Lem:Partition}, then we notice that, using the notation of the lemma, 
%we have $|D_1\cap U'|=|D_2\cap U'|=1$ and the same result holds for $V'$. 
%Indeed, we have at least one vertex in both of these intersections and we cannot 
%have more than one vertex since each vertex in $D_1$ is adjacent to exactly one 
%vertex in $D_2$ and vice versa. Hence, $S_1 \setminus U'$ is super dominating in 
%$G_1 \setminus U'$ 
%and $S_2 \setminus V'$ is super dominating in $G_2 \setminus V'$.  
%Thus, as in the previous case, we may assume that $N(u_1)\cap 
%\overline{S_1}=\{u_r\}$. Moreover, we may assume that %$h\neq 1$ (if $h=1$, 
%then we can use Lemma~\ref{Lem:Partition} to swap the vertex not in the super 
%dominating set away from $v_1$) and 
%$N(v_j)\cap \overline{S_2}=\{v_h\}$ for some $j\neq h$ such that $j\in \{1,2,\dots,r\}$. 
Let $$S=(S_1\cup S_2\cup W) \setminus (V'\cup V'').$$ 

%Again, all vertices from $ S \setminus W $ super dominate the corresponding vertices. 

In this case, no vertex in $V'$ or in $V''$ super dominates any vertex in $ \overline{S_1} 
\setminus V'$ or $ \overline{S_2} \setminus V'' $, 
respectively. Furthermore, $\overline{S} \subseteq \overline{S_1}\cup\overline{S_2}$. Thus, if $u\in \overline{S}$, then $u\in \overline{S_1}\cup\overline{S_2}$ and $u$ is super dominated by a vertex in $S_1$ or in $S_2$. That same vertex is in $S$ and super dominates $u$ in $S$. Thus, $S$ is a super dominating set of $G$. 

As we replace $2r-2$ vertices from $S_1 \cup S_2 $ 
by $r \geq 2 $ other vertices to reach $S$, it follows that $|S|\leq\gamma_{sp}(G_1)+\gamma_{sp}(G_2)$. 

%Since each  vertex in $\overline{S}$ is super dominated by the same vertices as 
%previously (and possibly also by some new vertices in $W'$), $S$ is super 
%dominating in $G$ and its cardinality clearly satisfies the upper bound.
\end{itemize}

\pagebreak[3]

This finishes the proof. 
\qed
\end{proof}

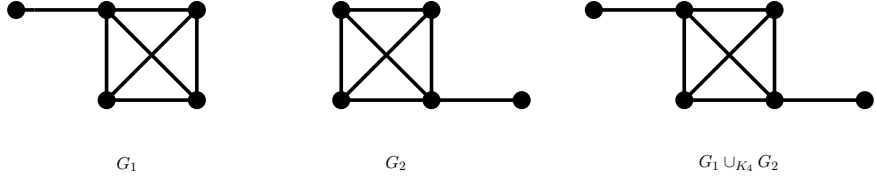
\begin{figure}
\begin{center}
\psscalebox{0.6 0.6}
{
\begin{pspicture}(0,-10.891442)(19.194231,-7.114327)
\rput[bl](2.4171154,-10.891442){$G_1$}
\rput[bl](8.377115,-10.871442){$G_2$}
\psline[linecolor=black, linewidth=0.08](2.1971154,-7.3114424)(4.1971154,-7.3114424)
\psline[linecolor=black, linewidth=0.08](2.1971154,-7.3114424)(4.1971154,-9.311442)(2.1971154,-9.311442)(2.1971154,-7.3114424)(2.1971154,-7.3114424)
\psline[linecolor=black, linewidth=0.08](2.1971154,-9.311442)(4.1971154,-7.3114424)(4.1971154,-9.311442)(4.1971154,-9.311442)
\psline[linecolor=black, linewidth=0.08](2.1971154,-7.3114424)(0.5971154,-7.3114424)(0.5971154,-7.3114424)
\psdots[linecolor=black, dotsize=0.4](0.1971154,-7.3114424)
\psdots[linecolor=black, dotsize=0.4](2.1971154,-7.3114424)
\psdots[linecolor=black, dotsize=0.4](4.1971154,-7.3114424)
\psdots[linecolor=black, dotsize=0.4](4.1971154,-9.311442)
\psdots[linecolor=black, dotsize=0.4](2.1971154,-9.311442)
\psline[linecolor=black, linewidth=0.08](0.1971154,-7.3114424)(0.5971154,-7.3114424)(0.5971154,-7.3114424)
\psline[linecolor=black, linewidth=0.08](7.397115,-7.3114424)(9.397116,-7.3114424)
\psline[linecolor=black, linewidth=0.08](7.397115,-7.3114424)(9.397116,-9.311442)(7.397115,-9.311442)(7.397115,-7.3114424)(7.397115,-7.3114424)
\psline[linecolor=black, linewidth=0.08](7.397115,-9.311442)(9.397116,-7.3114424)(9.397116,-9.311442)(9.397116,-9.311442)
\psdots[linecolor=black, dotsize=0.4](7.397115,-7.3114424)
\psdots[linecolor=black, dotsize=0.4](9.397116,-7.3114424)
\psdots[linecolor=black, dotsize=0.4](9.397116,-9.311442)
\psdots[linecolor=black, dotsize=0.4](7.397115,-9.311442)
\psline[linecolor=black, linewidth=0.08](14.997115,-7.3114424)(16.997116,-7.3114424)
\psline[linecolor=black, linewidth=0.08](14.997115,-7.3114424)(16.997116,-9.311442)(14.997115,-9.311442)(14.997115,-7.3114424)(14.997115,-7.3114424)
\psline[linecolor=black, linewidth=0.08](14.997115,-9.311442)(16.997116,-7.3114424)(16.997116,-9.311442)(16.997116,-9.311442)
\psline[linecolor=black, linewidth=0.08](14.997115,-7.3114424)(13.397116,-7.3114424)(13.397116,-7.3114424)
\psdots[linecolor=black, dotsize=0.4](12.997115,-7.3114424)
\psdots[linecolor=black, dotsize=0.4](14.997115,-7.3114424)
\psdots[linecolor=black, dotsize=0.4](16.997116,-7.3114424)
\psdots[linecolor=black, dotsize=0.4](16.997116,-9.311442)
\psdots[linecolor=black, dotsize=0.4](14.997115,-9.311442)
\psline[linecolor=black, linewidth=0.08](12.997115,-7.3114424)(13.397116,-7.3114424)(13.397116,-7.3114424)
\psline[linecolor=black, linewidth=0.08](9.397116,-9.311442)(11.397116,-9.311442)(11.397116,-9.311442)
\psline[linecolor=black, linewidth=0.08](16.997116,-9.311442)(18.997116,-9.311442)(18.997116,-9.311442)
\psdots[linecolor=black, dotsize=0.4](11.397116,-9.311442)
\psdots[linecolor=black, dotsize=0.4](18.997116,-9.311442)
\rput[bl](15.317116,-10.891442){$G_1\cup_{K_4}G_2$}
\end{pspicture}
}
\end{center}
\caption{$G_1$, $G_2$ and $G_1\cup_{K_4}G_2$, respectively.} \label{GGUK4}
\end{figure}

%\begin{figure}
%\begin{subfigure}{0.33\linewidth}
%\includegraphics[width=0.95\textwidth, height = 3cm]{g1union4.png}
%\caption{$G_1$.}
%\end{subfigure}
%\begin{subfigure}{0.33\linewidth}
%\includegraphics[width=0.95\textwidth, height = 3cm]{g2union4.png}
%\caption{$G_2$.}
%\end{subfigure}
%\begin{subfigure}{0.33\linewidth}
%\includegraphics[width=0.95\textwidth, height = 3cm]{g1union4g2.png}
%\caption{$G_1\cup_{K_4}G_2$.}
%\end{subfigure}
%\caption{Example for $4$-gluing of two graphs $G_1$ and $G_2$.}
%\label{GGUK4}
%\end{figure}

In the following remark, we show that the lower bound in 
Theorem~\ref{G1G2Kr-thm} is tight:

\begin{remark}
Let $r\geq 2$ and consider an $(r+1)$-vertex graph $G$ which is formed from the  
complete graph $K_r$ by attaching a single leaf to one of the vertices in $K_r$. 
We have $\gamma_{sp}(G)=r$. Let $G'$  be the graph  with two  
leaves attached to different vertices obtained by $r$-gluing $G\cup_{K_r}G$ (see Figure~\ref{GGUK4}). One  
can easily check that $\gamma_{sp}(G')=r=\gamma_{sp}(G)+\gamma_{sp}(G)-r$. So,  
the lower bound in Theorem~\ref{G1G2Kr-thm} is tight.
\end{remark}

We finish this section by showing that also the upper bound in 
Theorem~\ref{G1G2Kr-thm} is tight:

\begin{remark}  
Suppose that $G_1=G_2=G$ where G is formed from $K_r$ by attaching two leaves to 
each vertex in the $r$-clique (see Figure~\ref{GGUK4-lowersupport}). 
Then one can easily check that $\gamma _{sp}(G)=2r$ and $\gamma 
_{sp}(G_1\cup_{K_r}G_2)=4r=\gamma _{sp}(G_1)+\gamma _{sp}(G_2)$. So, the upper 
bound in Theorem~\ref{G1G2Kr-thm} is tight.
\end{remark}

\begin{figure}
\begin{center}
\psscalebox{0.6 0.6}
{
\begin{pspicture}(0,-10.931442)(16.794231,-3.0743268)
\psline[linecolor=black, linewidth=0.08](2.1971154,-5.2714424)(4.1971154,-5.2714424)
\psline[linecolor=black, linewidth=0.08](2.1971154,-5.2714424)(4.1971154,-7.2714424)(2.1971154,-7.2714424)(2.1971154,-5.2714424)(2.1971154,-5.2714424)
\psline[linecolor=black, linewidth=0.08](2.1971154,-7.2714424)(4.1971154,-5.2714424)(4.1971154,-7.2714424)(4.1971154,-7.2714424)
\psline[linecolor=black, linewidth=0.08](2.1971154,-5.2714424)(0.5971154,-5.2714424)(0.5971154,-5.2714424)
\psdots[linecolor=black, dotsize=0.4](0.1971154,-5.2714424)
\psdots[linecolor=black, dotsize=0.4](2.1971154,-5.2714424)
\psdots[linecolor=black, dotsize=0.4](4.1971154,-5.2714424)
\psdots[linecolor=black, dotsize=0.4](4.1971154,-7.2714424)
\psdots[linecolor=black, dotsize=0.4](2.1971154,-7.2714424)
\psline[linecolor=black, linewidth=0.08](0.1971154,-5.2714424)(0.5971154,-5.2714424)(0.5971154,-5.2714424)
\rput[bl](12.737116,-10.931442){$G_1\cup_{K_4}G_2$}
\psline[linecolor=black, linewidth=0.08](2.1971154,-5.2714424)(2.1971154,-3.2714422)(2.1971154,-3.2714422)
\psline[linecolor=black, linewidth=0.08](4.1971154,-5.2714424)(4.1971154,-3.2714422)(4.1971154,-3.2714422)
\psline[linecolor=black, linewidth=0.08](4.1971154,-5.2714424)(6.1971154,-5.2714424)(6.1971154,-5.2714424)
\psline[linecolor=black, linewidth=0.08](4.1971154,-7.2714424)(6.1971154,-7.2714424)(6.1971154,-7.2714424)
\psline[linecolor=black, linewidth=0.08](4.1971154,-7.2714424)(4.1971154,-9.271442)(4.1971154,-9.271442)
\psline[linecolor=black, linewidth=0.08](2.1971154,-7.2714424)(2.1971154,-9.271442)(2.1971154,-9.271442)
\psline[linecolor=black, linewidth=0.08](2.1971154,-7.2714424)(0.1971154,-7.2714424)(0.1971154,-7.2714424)
\psdots[linecolor=black, dotsize=0.4](2.1971154,-3.2714422)
\psdots[linecolor=black, dotsize=0.4](4.1971154,-3.2714422)
\psdots[linecolor=black, dotsize=0.4](6.1971154,-5.2714424)
\psdots[linecolor=black, dotsize=0.4](6.1971154,-7.2714424)
\psdots[linecolor=black, dotsize=0.4](4.1971154,-9.271442)
\psdots[linecolor=black, dotsize=0.4](2.1971154,-9.271442)
\psdots[linecolor=black, dotsize=0.4](0.1971154,-7.2714424)
\rput[bl](2.5771153,-10.831442){$G_1=G_2$}
\psline[linecolor=black, linewidth=0.08](12.5971155,-5.2714424)(14.5971155,-5.2714424)
\psline[linecolor=black, linewidth=0.08](12.5971155,-5.2714424)(14.5971155,-7.2714424)(12.5971155,-7.2714424)(12.5971155,-5.2714424)(12.5971155,-5.2714424)
\psline[linecolor=black, linewidth=0.08](12.5971155,-7.2714424)(14.5971155,-5.2714424)(14.5971155,-7.2714424)(14.5971155,-7.2714424)
\psline[linecolor=black, linewidth=0.08](12.5971155,-5.2714424)(10.997115,-5.2714424)(10.997115,-5.2714424)
\psdots[linecolor=black, dotsize=0.4](10.5971155,-5.2714424)
\psdots[linecolor=black, dotsize=0.4](12.5971155,-5.2714424)
\psdots[linecolor=black, dotsize=0.4](14.5971155,-5.2714424)
\psdots[linecolor=black, dotsize=0.4](14.5971155,-7.2714424)
\psdots[linecolor=black, dotsize=0.4](12.5971155,-7.2714424)
\psline[linecolor=black, linewidth=0.08](10.5971155,-5.2714424)(10.997115,-5.2714424)(10.997115,-5.2714424)
\psline[linecolor=black, linewidth=0.08](12.5971155,-5.2714424)(12.5971155,-3.2714422)(12.5971155,-3.2714422)
\psline[linecolor=black, linewidth=0.08](14.5971155,-5.2714424)(14.5971155,-3.2714422)(14.5971155,-3.2714422)
\psline[linecolor=black, linewidth=0.08](14.5971155,-5.2714424)(16.597115,-5.2714424)(16.597115,-5.2714424)
\psline[linecolor=black, linewidth=0.08](14.5971155,-7.2714424)(16.597115,-7.2714424)(16.597115,-7.2714424)
\psline[linecolor=black, linewidth=0.08](14.5971155,-7.2714424)(14.5971155,-9.271442)(14.5971155,-9.271442)
\psline[linecolor=black, linewidth=0.08](12.5971155,-7.2714424)(12.5971155,-9.271442)(12.5971155,-9.271442)
\psline[linecolor=black, linewidth=0.08](12.5971155,-7.2714424)(10.5971155,-7.2714424)(10.5971155,-7.2714424)
\psdots[linecolor=black, dotsize=0.4](12.5971155,-3.2714422)
\psdots[linecolor=black, dotsize=0.4](14.5971155,-3.2714422)
\psdots[linecolor=black, dotsize=0.4](16.597115,-5.2714424)
\psdots[linecolor=black, dotsize=0.4](16.597115,-7.2714424)
\psdots[linecolor=black, dotsize=0.4](14.5971155,-9.271442)
\psdots[linecolor=black, dotsize=0.4](12.5971155,-9.271442)
\psdots[linecolor=black, dotsize=0.4](10.5971155,-7.2714424)
\psline[linecolor=black, linewidth=0.08](12.5971155,-7.2714424)(10.5971155,-8.471442)(10.5971155,-8.471442)
\psline[linecolor=black, linewidth=0.08](12.5971155,-7.2714424)(11.397116,-9.271442)(11.397116,-9.271442)
\psline[linecolor=black, linewidth=0.08](12.5971155,-5.2714424)(10.5971155,-4.071442)(10.5971155,-4.071442)
\psline[linecolor=black, linewidth=0.08](12.5971155,-5.2714424)(11.397116,-3.2714422)(11.397116,-3.2714422)
\psline[linecolor=black, linewidth=0.08](14.5971155,-5.2714424)(15.797115,-3.2714422)(15.797115,-3.2714422)
\psline[linecolor=black, linewidth=0.08](14.5971155,-5.2714424)(16.597115,-4.071442)(16.597115,-4.071442)
\psline[linecolor=black, linewidth=0.08](14.5971155,-7.2714424)(16.597115,-8.471442)(16.597115,-8.471442)
\psline[linecolor=black, linewidth=0.08](14.5971155,-7.2714424)(15.797115,-9.271442)(15.797115,-9.271442)(15.797115,-9.271442)
\psdots[linecolor=black, dotsize=0.4](10.5971155,-4.071442)
\psdots[linecolor=black, dotsize=0.4](11.397116,-3.2714422)
\psdots[linecolor=black, dotsize=0.4](15.797115,-3.2714422)
\psdots[linecolor=black, dotsize=0.4](16.597115,-4.071442)
\psdots[linecolor=black, dotsize=0.4](16.597115,-8.471442)
\psdots[linecolor=black, dotsize=0.4](15.797115,-9.271442)
\psdots[linecolor=black, dotsize=0.4](11.397116,-9.271442)
\psdots[linecolor=black, dotsize=0.4](10.5971155,-8.471442)
\end{pspicture}
}
\end{center}
\caption{$G_1=G_2$ and $G_1\cup_{K_4}G_2$, respectively.} \label{GGUK4-lowersupport}
\end{figure}
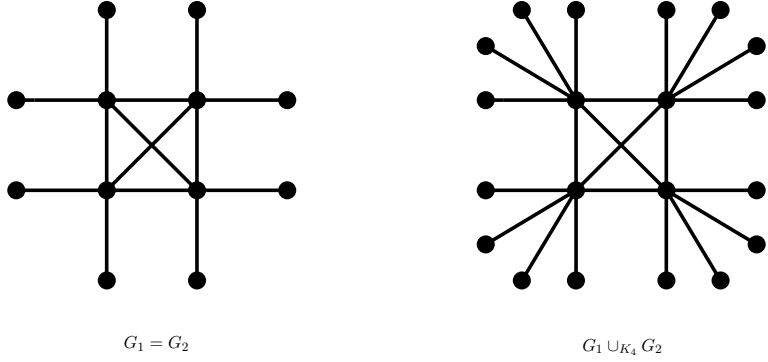

%\begin{figure}
%\begin{subfigure}{0.5\linewidth}
%\includegraphics[width=0.95\textwidth, height = 3cm]{g1equalg2.png}
%\caption{$G_1=G_2$.}
%\end{subfigure}
%\begin{subfigure}{0.5\linewidth}
%\includegraphics[width=0.95\textwidth, height = 3cm]{g1equalg2union4.png}
%\caption{$G_1\cup_{K_4}G_2$.}
%\end{subfigure}
%\caption{Example for $4$-gluing a graph with itself.}
%\label{GGUK4-lowersupport}
%\end{figure}

\section{Super domination number of the  Haj\'{o}s sum of graphs}\label{Sec:Hajos}

In this section, we consider the Haj\'{o}s sum  of two graphs. 
We start with a tight lower bound for the super domination number of the Haj\'{o}s sum  of two graphs.

\begin{theorem}\label{Hajos}
Let $G_1=(V_1,E_1)$ and $G_2=(V_2,E_2)$ be two graphs with disjoint 
vertex sets, $x_1y_1\in E_1$ and $x_2y_2\in E_2$. Then for the Haj\'{o}s sum 
$$G_3=G_1(x_1y_1)+_H G_2(x_2y_2),$$
it holds:
$$ \gamma _{sp} (G_1) + \gamma _{sp} (G_2) -2 \leq  \gamma _{sp}(G_3) 
\leq  \gamma _{sp} (G_1) + \gamma _{sp} (G_2).  $$
\end{theorem}

\begin{proof}
Suppose that we have formed $G_3$. Let $v:=v_H(x_1x_2)$ be the vertex 
identifying the vertices $x_1$ and $x_2$.
Again we divide the proof into two parts, namely the lower and the upper bound.

\noindent
{\bf Lower bound} $ \mathbf{\gamma _{sp} (G_1) + \gamma _{sp} (G_2) -2 
\leq  \gamma _{sp}(G_3)} $. 

\noindent Let $S$ be a super dominating set for 
$G_3$. With Lemma~\ref{Lem:Partition}, we may assume that $v\in S$. Moreover, 
the cases $v\in S$, $y_1\in S$, $y_2\in \overline{S}$, and $v\in S$, $y_1\in  
\overline{S}$, $y_2\in S$ are symmetrical and can be studied together. 
Hence, we have the following three cases. 
\begin{itemize}

\item[(i)]
$v\in {S}$, $y_1\in \overline{S}$, $y_2\in \overline{S}$. 

There might exist one vertex in $\overline{S}$ which is super dominated
by $v $. 
If such vertex exists, denote it by $v'$. 
Then it holds that $ \overline{S}\cap N(v)  = \{v'\} $.
We may assume, without loss of generality, that $v'\in V_1$. 
Observe that $ v' \not= y_1 $ holds.

Let 
$$S_1=(S \cup \{x_1,v'\}) \setminus (V_2\cup \{v\})$$ 

(or $S_1=(S \cup \{x_1\}) \setminus (V_2\cup \{v\})$ if 
$v'$ does not exist) 
and 
$$S_2=(S \cup \{x_2\}) \setminus (V_1\cup \{v\}).$$ 

We have $\overline{S_i}\subseteq \overline{S}$. Let $u\in \overline{S_i}$. If 
$u$ is super dominated in $G_3$ by some vertex $w\in S$, then $w\neq v$, since $u\neq v'$, 
and thus $w\in S_i$. Hence, $S_1$ is a super dominating set for $G_1$ and $S_2$ is a 
super dominating set for $G_2$. 
Thus,
$$ \gamma _{sp} (G_1) + \gamma _{sp} (G_2)\leq \gamma _{sp}(G_3)  +2.  $$	

%it is super dominated by 

%As there exist vertices $y_i'\in S \setminus V_j$, 
%where $\{i,j\}=\{1,2\}$,  such that $N(y_1')\cap 
%\overline{S}=\{y_1\}$ and  $N(y_2')\cap \overline{S}=\{y_2\}$, 
%$y_i'\in S_i$ for both $i\in\{1,2\}$, 
%and neither $x_1$ nor $x_2$ are needed to super dominate any 
%vertex, 
%$S_1$ is a super dominating set for $G_1$ and $S_2$ is a 
%super dominating set for $G_2$. 
%Thus,
%$$ \gamma _{sp} (G_1) + \gamma _{sp} (G_2)\leq \gamma _{sp}(G_3)  +2.  $$	

\item[(ii)]
$v\in {S}$, $y_1\in {S}$, $y_2\in \overline{S}$. 

Let $$S_1=(S \cup \{x_1\}) \setminus (V_2\cup \{v\}), \; \;
S_2=(S \cup \{x_2,y_2\}) \setminus (V_1\cup \{v\}).$$

In comparison to Case (i), $y_2$ is added to $S_2$,
as it could occur that $y_1$ super dominates
$y_2$ in $S$, but $y_2$ is not super dominated in $S_2$.
Observe that a vertex $v'$, which was super dominated by $v$ in $S$, as in Case (i), does not have to be added.
First, let $v' \in V_1$. Then because of $y_1\in S$, $x_1$ super dominates $v'$ in $S_1$.
Second, let $v' \in V_2$. Then because of $y_2\in S_2$, 
$x_2$ super dominates $v'$ in $S_2$.

So $S_1$ is a super dominating set for $G_1$ and $S_2$ is a 
super dominating set for $G_2$. 
Thus again,
$$ \gamma _{sp} (G_1) + \gamma _{sp} (G_2)\leq \gamma _{sp}(G_3)  +2.  $$

\item[(iii)]
$v\in {S}$, $y_1\in {S}$, $y_2\in S$. 

Let
$$S_1=(S \cup \{x_1\})  \setminus (V_2\cup \{v\}), \; \; S_2=(S \cup \{x_2\}) 
\setminus (V_1\cup \{v\}).$$ 
Observe that a vertex $v'$, which was super dominated by $v$ in $S$, as in Case (i), does not have to be added.
Because $y_1\in {S}$, $y_2\in S$, 
$x_1$ super dominates $v'$ in $S_1$ or $x_2$ super dominates $v'$ in $S_2$.

So $S_1$ is a super dominating set for $G_1$ and $S_2$ is a 
super dominating set for $G_2$. 
Thus,
$$ \gamma _{sp} (G_1) + \gamma _{sp} (G_2)\leq \gamma _{sp}(G_3)  +1.  $$	
\end{itemize}

\noindent
{\bf Upper bound} $ \mathbf{\gamma _{sp}(G_3)  
\leq  \gamma _{sp} (G_1) + \gamma _{sp} (G_2)} $.

\noindent Let $S_1$ and $S_2$ be super dominating sets for $G_1$ and $G_2$, 
respectively. 
Let $f_i:\overline{S_i'}\to\overline{S_i}$ be the bijective function 
introduced in  Lemma~\ref{Lem:Partition} where $S_i'$ is a super dominating set 
in $G_i$ for both $i\in\{1,2\}.$ In particular, we can now assume that $y_1\in 
S_1$ and $y_2\in 
S_2$. Moreover, we may assume that if $x_1\in \overline{S_1}$, then it is super 
dominated by $y_1$. Indeed, if $x_1$ is not super dominated by $y_1$, then $y_1$ 
does not super dominate any vertex in $V_1$ and so $y_1\in D$. Thus, 
$y_1,x_1\in S'_1$ and we could consider $S'_1$ instead of $S_1.$ This holds
analogously for~$x_2$.

We have the following three cases:
\begin{itemize}
\item[(i)]
$x_1\in S_1$, $x_2\in S_2$.  

There might be vertices $t_1\in V_1\cap 
\overline{S_1}$ and  $t_2\in V_2\cap\overline{S_2}$ such that $t_1$ is super 
dominated by $x_1$ and $t_2$ is super dominated by $x_2$. By the definition of a
super dominating set, if $t_1$ exists, then all neighbours of $x_1$ are in $S_1$ 
except $t_1$, and the same is true for $x_2$. Without loss of generality, if 
exactly one exists, we assume it is $t_1$. 
Let $$S=(S_1\cup 
S_2\cup\{t_1,v\}) \setminus \{x_1,x_2\}$$ 
(or $ (S_1 \cup S_2\cup\{v\}) \setminus \{x_1,x_2\}$, if $t_1$ does not exist).

We have $|S|\leq |S_1|+|S_2|$, and
 $S$ is a super dominating set in $G_3$ since each vertex in $\overline{S}$ is super dominated by the same 
vertex in $S_1 \cup S_2 $ as before except possibly $t_2$ which is super dominated by  the vertex $v$.

\item[(ii)]
$x_1\in \overline{S_1}$, $x_2\in {S_2}$. 

By assumption, $y_1$ now super dominates $x_1$ and thus, 
$f_1(y_1)=x_1$. Thus, by Lemma ~\ref{Lem:Partition}, for the super dominating set 
$S_1'$ of $G_1$ it holds that $x_1\in {S_1'}$, $y_1\in \overline{S_1'}$ and 
$N(x_1)\cap \overline{S_1'}=\{y_1\}$, i.e., $x_1$ super dominates $y_1$. 
%In this case there exists $y_1'$ such that $N(y_1')\cap \overline{S_1'}=\{y_1\}$. If $y_1'=x_1$, then all other neighbours of $x_1$ should be in $S_1'$ and 
Let us now consider the set $$S=(S_1'\cup 
S_2\cup\{y_1,v\}) \setminus \{x_1,x_2\}.$$
We have $|S|\leq |S_1'|+|S_2|= |S_1|+|S_2|$ and $S$ is a super dominating set in $G_3$ since each vertex in $\overline{S}$ is super dominated by the same 
vertex in $S'_1 \cup S_2 $ as before with the (possible) exception that the vertex which was previously super dominated by $x_2$ is now super dominated by $v$. 

\item[(iii)]
$x_1\in \overline{S_1}$, $x_2\in \overline{S_2}$. 

As in Case (ii) for $x_1$, by assumption $y_i$ now super 
dominates $x_i$ for both $i\in\{1,2\}$. Thus, in $S_1'$, $x_1$ super 
dominates $y_1$. Let us now consider the set $$S=(S_1'\cup 
S_2\cup\{v\}) \setminus \{x_1\}.$$
We have $|S|\leq |S_1|+|S_2|$. Moreover, $y_1$ is super dominated by $y_2$, $v$ 
does not super dominate any vertex and all other vertices in $\overline{S}$ 
are super dominated by the same vertices in $ S'_1 \cup S_2 $ as before.
\end{itemize}
%Hence, the claim follows and 
This finishes the proof.
\qed
\end{proof}

We finish this section by showing that both the lower and 
upper bound of Theorem~\ref{Hajos} is tight.

\begin{remark}
Consider $G_1=C_{4p+2}$ and $G_2=C_3$. Then one can easily check that $G_1(x_1y_1)+_H 
G_2(x_2y_2)=C_{4p+4}$, for any two edges $x_1y_1$ and $x_2y_2$ from $G_1$ and 
$G_2$, respectively. By Theorem~\ref{thm-2}(b), we have 
$\gamma _{sp}(G_1(x_1y_1)+_H G_2(x_2y_2))=2p+2$, $\gamma _{sp}(G_1)=2p+2$ and 
$\gamma _{sp}(G_2)=2$. 
Thus, the lower bound of Theorem~\ref{Hajos} is tight. 
\end{remark}

%Now we present a tight upper bound for 
%the super domination number of the Haj\'{o}s sum  of two graphs.

%\begin{theorem}\label{Hajosupper}
%Let $G_1=(V_1,E_1)$ and $G_2=(V_2,E_2)$ be two graphs with disjoint 
%vertex sets, $x_1y_1\in E_1$ and $x_2y_2\in E_2$. Then for the Haj\'{o}s sum 
%$$G_3=G_1(x_1y_1)+_H G_2(x_2y_2),$$
%it holds:
%\end{theorem}

\begin{remark}
Consider the cycles 
$G_1=C_{4p}$ and $G_2=C_3$. Then, we have $G_1(x_1y_1)+_H 
G_2(x_2y_2)=C_{4p+2}$, for any two edges $x_1y_1$ and $x_2y_2$ from $G_1$ and 
$G_2$, respectively. By Theorem~\ref{thm-2}(b),  we have 
$\gamma _{sp}(G_1(x_1y_1)+_H G_2(x_2y_2))=2p+2$, $\gamma _{sp}(G_1)=2p$ and 
$\gamma _{sp}(G_2)=2$. 
Thus, the upper bound of Theorem~\ref{Hajos} is tight. 
\end{remark}

\section{The number of minimum size super dominating sets of some graphs}\label{Sec:SPsetNumber}

In this section, we initiate the study of the number of minimum size super dominating  sets 
of a graph. Similar research has been conducted, for example for the 
domination number in multiple papers, see for example,~\cite{Connolly}.
Let ${\cal N}_{sp}(G)$ be the family of
super  dominating sets of a graph $G$ with cardinality $\gamma _{sp} (G)$ and let
$N_{sp}(G)=|{\cal N}_{sp}(G)|$. By 
Theorem~\ref{thm-1} and Lemma~\ref{Lem:Partition}, for every 
non-empty graph $G$ we have $N_{sp}(G)\geq 2$. In the following, we consider some 
special graph classes and compute their $N_{sp}$ values. Following 
Theorem~\ref{thm-2}, by an easy argument, we have the following result for 
$N_{sp}$ of  the complete graph, the complete bipartite graph and the star 
graph.
  
\begin{theorem}\label{number-s.d.set}
\begin{enumerate}
\item[(a)] If $K_n$ is the complete graph,  then $N_{sp}(K_n)= n$.

\item[(b)] If $K_{n,m}$ is the complete bipartite graph,  then $N_{sp}(K_{n,m})= nm$, where $\min\{n,m\}\geq 2$.

\item[(c)] If $K_{1,n}$ is the  star graph, then $N_{sp}(K_{1,n})= n+1$.

\end{enumerate}
\end{theorem}

\begin{proof}
\begin{itemize}
\item[(a)] 
By Theorem~\ref{thm-2}(c), 
$\gamma_{sp}(K_n)= n-1$  holds.
Thus, $|\overline{S}|=1 $ follows. 
Clearly, any single vertex of $K_n$ can be chosen 
as $\overline{S} $. As $K_n$ has exactly $n$ vertices, it follows that $N_{sp}(K_n)= n$.
\item[(b)]
Let $\min\{n,m\}\geq 2$.
By Theorem~\ref{thm-2}(d), 
$\gamma_{sp}(K_{n,m})= n+m-2$ holds.
Thus, $|\overline{S}|=2 $ follows. 
Clearly, these two vertices of $ \overline{S} $ have to be chosen from two 
different sides of the bipartition. 
It follows that $N_{sp}(K_{n,m})= nm$.
\item[(c)]
By Theorem~\ref{thm-2}(e), 
$\gamma_{sp}(K_{1,n})= n$ holds.
Thus, $|\overline{S}|=1 $ follows. 
Clearly, any single vertex of $K_{1,n}$ can be chosen 
as $\overline{S} $. As $K_{1,n}$ has exactly $n+1$ vertices, it follows that 
$N_{sp}(K_{1,n})= n+1$.
\qed
\end{itemize}
\end{proof}
  
In the following, we compute $N_{sp}$  of the friendship graph.

\begin{theorem}
Let $F_n$ be the friendship graph of order $n$.  Then 
$$N_{sp}(F_n)= 2^n.$$
\end{theorem}

\begin{proof}
By Theorem~\ref{Firend-thm}, we know that $\gamma_{sp}(F_n)=n+1$. Now 
consider Figure~\ref{friend}. For any dominating set $S$ of the friendship 
graph with cardinality less than $2n$, if we do not have 
$\{x,u_{2t-1}\}\subseteq S$ or $\{x,u_{2t}\}\subseteq S$, where $1\leq t 
\leq n$ and $x$ is the central vertex, then it is clear that $S$ is not a super dominating set. 
So we need $x$ in our super dominating set. Among $u_{2t-1}$ and $u_{2t}$,  
where $1\leq t \leq n$, we choose one of them.  So we have $2^n$ super 
dominating sets of size $n+1$, and we have the result.  		 \qed

\end{proof}
  
Now we consider the path graph and compute $N_{sp}(P_n)$.

\begin{theorem}
Let $P_n$ be the path graph of order $n\geq 2$.  Then 
\begin{displaymath}
N_{sp}(P_n)= \left\{ \begin{array}{ll}
2 & \textrm{if $n$ is even, }\\
\\
\frac{3}{2}(n-1) & \textrm{if $n$ is odd.}
\end{array} \right.
\end{displaymath}
Also we have $N_{sp}(P_1)=1$.
 %$N_{sp}(P_2)=2$, $N_{sp}(P_3)=3$ and $N_{sp}(P_4)=2$.
\end{theorem}

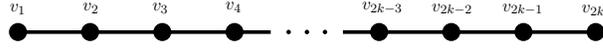
\begin{figure}
\begin{center}
\psscalebox{0.6 0.6}
{
\begin{pspicture}(0,-2.5535576)(13.219974,-1.6664423)
\psdots[linecolor=black, dotsize=0.4](0.19711533,-2.3564422)
\psdots[linecolor=black, dotsize=0.4](1.7971153,-2.3564422)
\psdots[linecolor=black, dotsize=0.4](3.3971152,-2.3564422)
\psdots[linecolor=black, dotsize=0.4](4.997115,-2.3564422)
\psdots[linecolor=black, dotsize=0.1](6.1971154,-2.3564422)
\psdots[linecolor=black, dotsize=0.1](6.5971155,-2.3564422)
\psdots[linecolor=black, dotsize=0.1](6.997115,-2.3564422)
\psdots[linecolor=black, dotsize=0.4](8.197115,-2.3564422)
\psdots[linecolor=black, dotsize=0.4](9.797115,-2.3564422)
\psdots[linecolor=black, dotsize=0.4](11.397116,-2.3564422)
\psdots[linecolor=black, dotsize=0.4](12.997115,-2.3564422)
\psline[linecolor=black, linewidth=0.08](5.7971153,-2.3564422)(0.19711533,-2.3564422)
\psline[linecolor=black, linewidth=0.08](7.397115,-2.3564422)(12.997115,-2.3564422)
\rput[bl](0.017115327,-1.9564422){${v_1}$}
\rput[bl](1.6371154,-1.9364423){${v_2}$}
\rput[bl](3.1771154,-1.9364423){${v_3}$}
\rput[bl](4.7971153,-1.9164423){${v_4}$}
\rput[bl](12.689973,-1.9812042){${v_{2k}}$}
\rput[bl](10.933306,-1.9812042){${v_{2k-1}}$}
\rput[bl](9.367592,-1.9812042){${v_{2k-2}}$}
\rput[bl](7.8209248,-1.9431089){${v_{2k-3}}$}
\end{pspicture}
}
\end{center}
\caption{ Path graph of order $2k\geq 2$.} \label{path2k}
\end{figure}

%\begin{figure}
%\centering
%\includegraphics[width=0.95\textwidth, height = 3cm]{path2k.png}
%\caption{ Path graph of order $2k\geq 6$ } 
%\label{path2k}
%\end{figure}

\begin{proof}
By Theorem~\ref{thm-2}(a),  we have  $\gamma_{sp}(P_n)=\lceil \frac{n}{2} \rceil$. 
We have two cases based on the parity of $n$:
\begin{itemize}
\item[(a)] $n$ even.

Let $n=2k$ with $ k \in \N $. 
Let $V=\{v_1,v_2,\ldots,v_{2k}\}$ be the vertex set 
of $P_{2k}$ (see Figure~\ref{path2k}), and $S$ be a super dominating 
set of $P_{2k}$ with $|S|=k$. Since $|S|=n/2$, in the partition of 
Lemma~\ref{Lem:Partition}(b), we have $D=\emptyset$ and $ V(G) = S \, \dot{\cup} 
\, S'$,
where $S$ and $S'$ have the same cardinality $n/2$.
By Lemma~\ref{Lem:Partition}(a), each vertex in $S$ is 
adjacent to exactly one vertex in $S'$ and vice versa. Thus, if we choose 
$v_1\in S$, then $v_2\in S'$, $v_3\in S'$, $v_4\in S$ and so on. After we 
know whether $v_1\in S$, all other vertices have their set decided. The case 
with $v_1\in S'$ is analogous; we can just swap $S$ and $S'$. 
Thus, we have $N_{sp}(P_{2k})=2$.

\begin{figure}
\begin{center}
\psscalebox{0.6 0.6}
{
\begin{pspicture}(0,-7.49)(20.9,1.63)
\psdots[linecolor=black, dotsize=0.4](10.450001,0.96)
\psdots[linecolor=black, dotsize=0.4](12.05,0.96)
\psdots[linecolor=black, dotsize=0.4](13.650001,0.96)
\psdots[linecolor=black, dotsize=0.4](15.250001,0.96)
\psdots[linecolor=black, dotsize=0.4](16.85,0.96)
\psdots[linecolor=black, dotsize=0.4](18.45,0.96)
\psdots[linecolor=black, dotsize=0.4](20.050001,0.96)
\rput[bl](10.330001,1.38){$v_1$}
\rput[bl](11.85,1.38){$v_2$}
\rput[bl](13.450001,1.36){$v_3$}
\rput[bl](15.030001,1.36){$v_4$}
\rput[bl](16.650002,1.34){$v_5$}
\rput[bl](18.27,1.34){$v_6$}
\rput[bl](19.810001,1.32){$v_7$}
\psline[linecolor=black, linewidth=0.08](10.450001,0.96)(16.85,0.96)(16.85,0.96)
\psline[linecolor=black, linewidth=0.08](16.85,0.96)(20.050001,0.96)(20.050001,0.96)
\psdots[linecolor=black, dotsize=0.4](0.8500006,0.96)
\psdots[linecolor=black, dotsize=0.4](2.4500005,0.96)
\psdots[linecolor=black, dotsize=0.4](4.0500007,0.96)
\psdots[linecolor=black, dotsize=0.4](5.6500006,0.96)
\psdots[linecolor=black, dotsize=0.4](7.2500005,0.96)
\rput[bl](0.7300006,1.38){$v_1$}
\rput[bl](2.2500007,1.38){$v_2$}
\rput[bl](3.8500006,1.36){$v_3$}
\rput[bl](5.430001,1.36){$v_4$}
\rput[bl](7.0500007,1.34){$v_5$}
\psline[linecolor=black, linewidth=0.08](0.8500006,0.96)(7.2500005,0.96)(7.2500005,0.96)
\psline[linecolor=black, linewidth=0.08](7.6500006,-0.24)(0.05000061,-0.24)(0.05000061,-5.04)(8.05,-5.04)(8.05,-0.24)(7.2500005,-0.24)(7.2500005,-0.24)
\psline[linecolor=black, linewidth=0.08](0.05000061,-1.04)(8.05,-1.04)(7.2500005,-1.04)
\psline[linecolor=black, linewidth=0.08](0.05000061,-1.84)(8.05,-1.84)(8.05,-1.84)
\psline[linecolor=black, linewidth=0.08](0.05000061,-2.64)(8.05,-2.64)(8.05,-2.64)
\psline[linecolor=black, linewidth=0.08](0.05000061,-3.44)(8.05,-3.44)(8.05,-3.44)
\psline[linecolor=black, linewidth=0.08](0.05000061,-4.24)(8.05,-4.24)(8.05,-4.24)
\psline[linecolor=black, linewidth=0.08](11.581035,-0.24)(9.650001,-0.24)(9.650001,-7.44)(20.85,-7.44)(20.85,-0.24)(10.808621,-0.24)(10.808621,-0.24)
\psline[linecolor=black, linewidth=0.08](9.650001,-1.04)(20.463793,-1.04)
\psline[linecolor=black, linewidth=0.08](9.650001,-1.84)(20.85,-1.84)
\psline[linecolor=black, linewidth=0.08](20.85,-1.04)(20.463793,-1.04)(20.463793,-1.04)
\psline[linecolor=black, linewidth=0.08](9.650001,-2.64)(20.85,-2.64)
\psline[linecolor=black, linewidth=0.08](9.650001,-3.44)(20.85,-3.44)
\psline[linecolor=black, linewidth=0.08](9.650001,-4.24)(20.85,-4.24)
\psline[linecolor=black, linewidth=0.08](9.650001,-5.04)(20.85,-5.04)
\psline[linecolor=black, linewidth=0.08](9.650001,-5.84)(20.85,-5.84)
\psline[linecolor=black, linewidth=0.08](9.650001,-6.64)(20.85,-6.64)
\psdots[linecolor=black, dotsize=0.4](10.450001,-0.64)
\psdots[linecolor=black, dotsize=0.4](10.450001,-1.44)
\psdots[linecolor=black, dotsize=0.4](10.450001,-2.24)
\psdots[linecolor=black, dotsize=0.4](10.450001,-3.04)
\psdots[linecolor=black, dotsize=0.4](10.450001,-3.84)
\psdots[linecolor=black, dotsize=0.4](12.05,-4.64)
\psdots[linecolor=black, dotsize=0.4](12.05,-5.44)
\psdots[linecolor=black, dotsize=0.4](12.05,-6.24)
\psdots[linecolor=black, dotsize=0.4](12.05,-7.04)
\psdots[linecolor=black, dotsize=0.4](12.05,-0.64)
\psdots[linecolor=black, fillstyle=solid, dotstyle=o, dotsize=0.4, fillcolor=white](10.450001,-4.64)
\psdots[linecolor=black, fillstyle=solid, dotstyle=o, dotsize=0.4, fillcolor=white](10.450001,-5.44)
\psdots[linecolor=black, fillstyle=solid, dotstyle=o, dotsize=0.4, fillcolor=white](10.450001,-6.24)
\psdots[linecolor=black, fillstyle=solid, dotstyle=o, dotsize=0.4, fillcolor=white](10.450001,-7.04)
\psdots[linecolor=black, fillstyle=solid, dotstyle=o, dotsize=0.4, fillcolor=white](12.05,-1.44)
\psdots[linecolor=black, fillstyle=solid, dotstyle=o, dotsize=0.4, fillcolor=white](12.05,-2.24)
\psdots[linecolor=black, dotsize=0.4](13.650001,-1.44)
\psdots[linecolor=black, dotsize=0.4](15.250001,-3.04)
\psdots[linecolor=black, dotsize=0.4](15.250001,-3.84)
\psdots[linecolor=black, dotsize=0.4](13.650001,-2.24)
\psdots[linecolor=black, fillstyle=solid, dotstyle=o, dotsize=0.4, fillcolor=white](12.05,-3.04)
\psdots[linecolor=black, fillstyle=solid, dotstyle=o, dotsize=0.4, fillcolor=white](13.650001,-3.04)
\psdots[linecolor=black, fillstyle=solid, dotstyle=o, dotsize=0.4, fillcolor=white](12.05,-3.84)
\psdots[linecolor=black, fillstyle=solid, dotstyle=o, dotsize=0.4, fillcolor=white](13.650001,-3.84)
\psdots[linecolor=black, dotsize=0.4](16.85,-3.04)
\psdots[linecolor=black, dotsize=0.4](16.85,-3.84)
\psdots[linecolor=black, dotsize=0.4](18.45,-3.04)
\psdots[linecolor=black, dotsize=0.4](20.050001,-3.84)
\psdots[linecolor=black, fillstyle=solid, dotstyle=o, dotsize=0.4, fillcolor=white](20.050001,-3.04)
\psdots[linecolor=black, fillstyle=solid, dotstyle=o, dotsize=0.4, fillcolor=white](18.45,-3.84)
\psdots[linecolor=black, fillstyle=solid, dotstyle=o, dotsize=0.4, fillcolor=white](13.650001,-0.64)
\psdots[linecolor=black, fillstyle=solid, dotstyle=o, dotsize=0.4, fillcolor=white](15.250001,-0.64)
\psdots[linecolor=black, dotsize=0.4](16.85,-0.64)
\psdots[linecolor=black, dotsize=0.4](18.45,-0.64)
\psdots[linecolor=black, dotsize=0.4](16.85,-1.44)
\psdots[linecolor=black, dotsize=0.4](18.45,-1.44)
\psdots[linecolor=black, dotsize=0.4](15.250001,-2.24)
\psdots[linecolor=black, dotsize=0.4](20.050001,-2.24)
\psdots[linecolor=black, fillstyle=solid, dotstyle=o, dotsize=0.4, fillcolor=white](20.050001,-0.64)
\psdots[linecolor=black, fillstyle=solid, dotstyle=o, dotsize=0.4, fillcolor=white](20.050001,-1.44)
\psdots[linecolor=black, fillstyle=solid, dotstyle=o, dotsize=0.4, fillcolor=white](18.45,-2.24)
\psdots[linecolor=black, fillstyle=solid, dotstyle=o, dotsize=0.4, fillcolor=white](16.85,-2.24)
\psdots[linecolor=black, fillstyle=solid, dotstyle=o, dotsize=0.4, fillcolor=white](15.250001,-1.44)
\psdots[linecolor=black, dotsize=0.4](13.650001,-4.64)
\psdots[linecolor=black, dotsize=0.4](13.650001,-5.44)
\psdots[linecolor=black, dotsize=0.4](13.650001,-6.24)
\psdots[linecolor=black, dotsize=0.4](13.650001,-7.04)
\psdots[linecolor=black, fillstyle=solid, dotstyle=o, dotsize=0.4, fillcolor=white](15.250001,-4.64)
\psdots[linecolor=black, fillstyle=solid, dotstyle=o, dotsize=0.4, fillcolor=white](15.250001,-5.44)
\psdots[linecolor=black, fillstyle=solid, dotstyle=o, dotsize=0.4, fillcolor=white](15.250001,-6.24)
\psdots[linecolor=black, dotsize=0.4](16.85,-4.64)
\psdots[linecolor=black, dotsize=0.4](18.45,-4.64)
\psdots[linecolor=black, dotsize=0.4](16.85,-5.44)
\psdots[linecolor=black, dotsize=0.4](20.050001,-5.44)
\psdots[linecolor=black, dotsize=0.4](18.45,-6.24)
\psdots[linecolor=black, dotsize=0.4](20.050001,-6.24)
\psdots[linecolor=black, dotsize=0.4](20.050001,-7.04)
\psdots[linecolor=black, dotsize=0.4](15.250001,-7.04)
\psdots[linecolor=black, fillstyle=solid, dotstyle=o, dotsize=0.4, fillcolor=white](20.050001,-4.64)
\psdots[linecolor=black, fillstyle=solid, dotstyle=o, dotsize=0.4, fillcolor=white](18.45,-5.44)
\psdots[linecolor=black, fillstyle=solid, dotstyle=o, dotsize=0.4, fillcolor=white](16.85,-6.24)
\psdots[linecolor=black, fillstyle=solid, dotstyle=o, dotsize=0.4, fillcolor=white](16.85,-7.04)
\psdots[linecolor=black, fillstyle=solid, dotstyle=o, dotsize=0.4, fillcolor=white](18.45,-7.04)
\psdots[linecolor=black, dotsize=0.4](0.8500006,-0.64)
\psdots[linecolor=black, dotsize=0.4](2.4500005,-0.64)
\psdots[linecolor=black, dotsize=0.4](7.2500005,-0.64)
\psdots[linecolor=black, dotsize=0.4](0.8500006,-1.44)
\psdots[linecolor=black, dotsize=0.4](4.0500007,-1.44)
\psdots[linecolor=black, dotsize=0.4](7.2500005,-1.44)
\psdots[linecolor=black, dotsize=0.4](0.8500006,-2.24)
\psdots[linecolor=black, dotsize=0.4](5.6500006,-2.24)
\psdots[linecolor=black, dotsize=0.4](7.2500005,-2.24)
\psdots[linecolor=black, dotstyle=o, dotsize=0.4, fillcolor=white](2.4500005,-3.04)
\psdots[linecolor=black, dotsize=0.4](4.0500007,-4.64)
\psdots[linecolor=black, dotsize=0.4](5.6500006,-4.64)
\psdots[linecolor=black, dotsize=0.4](0.8500006,-3.04)
\psdots[linecolor=black, dotsize=0.4](5.6500006,-3.04)
\psdots[linecolor=black, dotsize=0.4](2.4500005,-3.84)
\psdots[linecolor=black, dotstyle=o, dotsize=0.4, fillcolor=white](5.6500006,-3.84)
\psdots[linecolor=black, dotsize=0.4](7.2500005,-3.84)
\psdots[linecolor=black, dotsize=0.4](2.4500005,-4.64)
\psdots[linecolor=black, fillstyle=solid, dotstyle=o, dotsize=0.4, fillcolor=white](4.0500007,-0.64)
\psdots[linecolor=black, fillstyle=solid, dotstyle=o, dotsize=0.4, fillcolor=white](5.6500006,-0.64)
\psdots[linecolor=black, fillstyle=solid, dotstyle=o, dotsize=0.4, fillcolor=white](5.6500006,-1.44)
\psdots[linecolor=black, fillstyle=solid, dotstyle=o, dotsize=0.4, fillcolor=white](2.4500005,-1.44)
\psdots[linecolor=black, fillstyle=solid, dotstyle=o, dotsize=0.4, fillcolor=white](2.4500005,-2.24)
\psdots[linecolor=black, fillstyle=solid, dotstyle=o, dotsize=0.4, fillcolor=white](4.0500007,-2.24)
\psdots[linecolor=black, fillstyle=solid,fillcolor=black, dotsize=0.4](4.0500007,-3.04)
\psdots[linecolor=black, fillstyle=solid, dotstyle=o, dotsize=0.4, fillcolor=white](7.2500005,-3.04)
\psdots[linecolor=black, fillstyle=solid,fillcolor=black, dotsize=0.4](4.0500007,-3.84)
\psdots[linecolor=black, fillstyle=solid, dotstyle=o, dotsize=0.4, fillcolor=white](0.8500006,-3.84)
\psdots[linecolor=black, fillstyle=solid, dotstyle=o, dotsize=0.4, fillcolor=white](0.8500006,-4.64)
\psdots[linecolor=black, fillstyle=solid, dotstyle=o, dotsize=0.4, fillcolor=white](7.2500005,-4.64)
\end{pspicture}
}
\end{center}
\caption{Minimum size super dominating sets of $P_5$ and $P_7$, 
respectively.} 
\label{path5and7}
\end{figure}
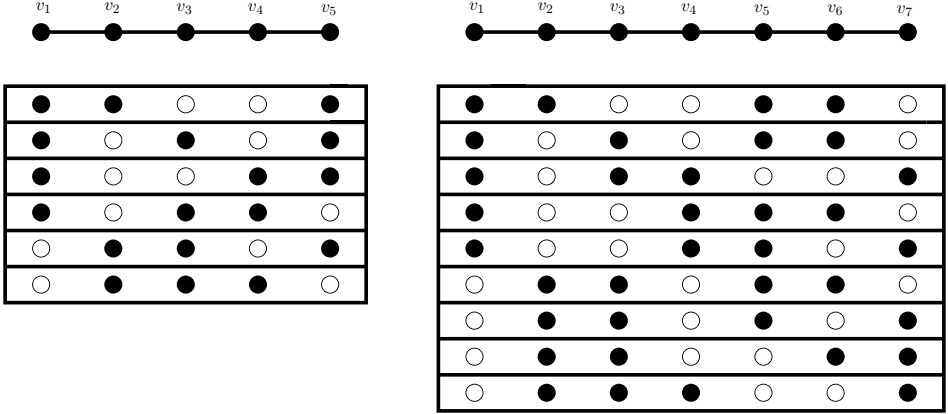

\item[(b)] $n$ odd.

Let $n=2k+1$ with $ k \in \N_0 $. 
It is easy to see that $N_{sp}(P_1)=1$, $N_{sp}(P_3)=3$. 
Now we consider $n\ge 5$ and thus $ k \ge 2$.
Let $V=\{v_1,v_2,\ldots,v_{2k+1}\}$ be the vertex set 
of $P_{2k+1}$. 

By Theorem~\ref{thm-2}(a), we have 
$\gamma_{sp}(P_{2k+1})=k+1$. Let $S$ be a super dominating set of cardinality 
$k+1$ and $S'$ be another super dominating set of the same cardinality with 
$\overline{S}\subseteq S'$ and $\overline{S'}\subseteq S$. Moreover, 
by Lemma~\ref{Lem:Partition}, let 
$f:\overline{S'}\to \overline{S}$ be a bijective function for which $f(a)=b$ if 
and only if $a$ super dominates $b$ for $S$ and $b$ super dominates $a$ for $S'$. 
%The set $S'$ and the function $f$ exist by Lemma \ref{Lem:Partition}. 
Observe that $|\overline{S}|=|\overline{S'}|=k$. Again,
by Lemma~\ref{Lem:Partition}(b), $|D|=|S\cap S'|=1$ holds. Let us denote by $w$ 
the single vertex in $S\cap S'$. Since $w$ is not necessary for super dominating 
any other vertices, $S\setminus\{w\}$ is a super dominating set for 
the induced subgraph $P'$ of $V(P_{2k+1})\setminus \{w\}$. Notice that $P'$ 
consists of either two paths or a single even-length path (when $w$ is the start 
or end vertex of the original path). We have $|S\setminus \{w\}|=k$. Thus, each path in 
$P'$ has even-length since odd paths have more than half of their vertices in 
any super dominating set. Hence, $w=v_i$ where $i$ is odd. Thus, we have $k+1$ 
possible choices for $w$. Furthermore, each of these choices for $w$ yields a 
super dominating set of the smallest cardinality.

When $w=v_1$, we have $P'=P_{2k}$. Since $N_{sp}(P_{2k})=2$, 
based on the proof of $n$ even, 
we may choose in this case $S$ in two different ways based on whether $v_2\in S$. 
Moreover, if we choose $S$ in one way, then $S'$ 
is the other super dominating set of the 
smallest cardinality in $P'$. Thus, $w=v_1$ contributes two super dominating 
sets. Furthermore, the choice $w=v_{2k+1}$ is symmetrical.

Let us now consider the case $w=v_{2i+1}$ where $i \in \{1,2,\dots,k-1\} $. Now, $P'$ 
consists of two even paths of lengths $2i$ and $2k-2i$. Hence, for each choice 
of $w$ there are four different smallest super dominating sets (two for both 
even paths). 
In summary, we can classify the super dominating sets in
eight classes, where by definition, the vertices $w$ are always black.
\begin{description}
\item[$\mathbf{SW(0)}$:] $ w= v_1 $, $v_2$ is white (i.e., does not lie in $S$),
\item[$\mathbf{SB(0)}$:] $ w= v_1, $ $v_2$ is black (i.e., lies in $S$),
\item[$\mathbf{EW(k)}$:] $ w= v_{2k+1} $, $v_{2k}$ is white,
\item[$\mathbf{EB(k)}$:] $ w= v_{2k+1} $, $v_{2k}$ is black,
\item[$\mathbf{MBB(i)}$] $ w= v_{2i+1} $, $v_{2i}$ is black,
$v_{2i+2}$ is black, 
where $ i \in \{1,2,\dots,k-1\} $,
\item[$\mathbf{MWW(i)}$:] $ w= v_{2i+1} $, $v_{2i}$ is white,
$v_{2i+2}$ is white,
where $ i \in \{1,2,\dots,k-1\} $,
\item[$\mathbf{MBW(i)}$:] $ w= v_{2i+1} $, $v_{2i}$ is black,
$v_{2i+2}$ is white,
where $ i \in \{1,2,\dots,k-1\} $,
\item[$\mathbf{MWB(i)}$:] $ w= v_{2i+1} $, $v_{2i}$ is white,
$v_{2i+2}$ is black,
where $ i \in \{1,2,\dots,k-1\} $.
\end{description}
This leads to $$ 2 \cdot 2 + 4 \cdot (k-1) = 4 k $$ super dominating sets. 
However, some super dominating sets appear in multiple  classes. 
%We encourage the reader to consider Figure \ref{path5and7} for better 
%understanding of these classes and their overlaps. 
Below we give a formal explanation on the number 
these classes overlap. 

Recall that the eight cases above are formed by first choosing the set $D=\{w\}$. 
Thus, a super dominating set $S$ is counted in multiple classes if and only if 
there are multiple ways to choose the set $D$ (and the super dominating set $S'$) for 
$S$. Moreover, since $|D|=1$, we can choose the set $D$ in multiple ways if and only 
if there is a white vertex which is super dominated by two different black 
vertices. Since the maximum degree of a vertex is $2$, we may choose the set $D$ in 
these cases in exactly two ways, say $D$ and $D'$, where 
the corresponding vertices $w$ and $w'$
have distance $2$ in the path. Hence, in each such case, a super dominating 
set is included in two classes with consecutive indices. 

Finally, if we look into the eight classes 
above, we notice that in the classes SB(0), EB(k), MBB(i) and MWW(i), the 
vertex $w$ does 
not super dominate any vertex. Thus, we do not have any white vertex, which 
is dominated twice in those classes. On the other hand, we have a white vertex that
is super dominated by two black vertices in the classes SW(0), EW(k), MBW(i), and 
MWB(i). This twice super dominated white vertex is $v_2$, $v_{2k}$, 
$v_{2i+2}$ and $v_{2i}$, respectively. 

We illustrate these overlapping super dominating sets 
in Figure~\ref{path5and7} for $P_5$ and $P_7$.  For $P_5$, 
in row IV, SW(0) and vertex $ w = v_1$ overlap with MWB(1) and vertex $w=v_3$, 
and in row V, EW(2) and vertex $ w = v_5$ overlap with MBW(1) and vertex $w=v_3$. 
For $P_7$, 
in row III, SW(0) and vertex $ w = v_1$ overlap with MWB(1) and vertex $w=v_3$, 
in row V, EW(3) and vertex $ w = v_7$ overlap with MBW(2) and vertex $w=v_5$, 
and in row VI, MBW(1) and vertex $ w = v_3$ overlap with MWB(2) and vertex $w=v_5$. 

That is, we are counting the super dominating 
sets twice in $1+1+(k-1)+(k-1)=2k$ classes. Hence, $N_{sp}(P_{2k+1})=4k-(2k/2)=3k$ as 
claimed. \qed

\end{itemize}

%$ S=\hat{S} $ can only hold, if two vertices with distance two in the path
%$ P_{2k+1} $ super dominate the same vertex. The equivalent condition follows.
%\end{itemize}
%By the claim, for each $ i \in \{1,2,\dots,k\} $ exactly one (of four) 
%super dominating set is counted twice. 

%If a super dominating set occurs for multiple 
%different choices of $w$, then it means that $w$ super dominates vertex $x$ but 
%$x$ is also super dominated by another vertex $y$. This occurs exactly when 
%$v_{2i}\in\overline{S}$ and $v_{2i+2}\in S$, or $v_{2i}\in S$ and $v_{2i+2}\in 
%\overline{S}$. Hence, for each choice of $w$, there are two cases which are 
%counted twice. Furthermore, when $w=v_{3}$ one of these cases was already 
%considered with the choice $w=v_1$ and when $w=v_{2k-1}$ one of these cases was 
%considered with the choice $w=v_{2k+1}$. Thus, in total between this and 
%previous case, where $w$ was an end-vertex, we have $4k$ super dominating sets. 
%However, 2k of these super dominating sets two are counted twice. 
\end{proof}

In Theorem \ref{Cycle-counting}, we determine the value $N_{sp}(C_n)$ for each $n$. 
Interestingly, the value varies quite a lot based on $n \pmod 4$. We will utilize necklace combinatorics for the proof of  Case (c). See \cite{NecklaceAlg} for an algorithm to calculate the value $N_n(q_1,q_2,\dots,q_k)$ in the following definition as well as some connections between necklaces and combinatorics on words.

\begin{definition}\label{def:necklace}
Let $ k, n, q_1, q_2,\dots,q_k \in \N $ with 
$ n = \sum_{i=1}^k q_i $.
\item[(a)]
A $(q_1,q_2,\dots,q_k)$-\textit{necklace} of length 
$n$ consists of a total of $n$ beads where we have $q_i$ beads of type $i$. The beads are placed into a cycle (necklace).

\item[(b)]
$N_n(q_1,q_2,\dots,q_k)$ is defined as
the number of different $(q_1,q_2,\dots,q_k)$-necklaces of length $n$, 
when the $n$ rotations of a necklace are considered as the same necklace. 
\end{definition}

The following relation is easy to see.
\begin{eqnarray}
\label{eq1_necklace}
N_{k+1}(k,1) & = & 1 
\text{ for } k \in \N.
\end{eqnarray}
Furthermore, we have 
\begin{eqnarray}
\label{eq2_necklace}
N_{2k+2}(2k,2)=k+1 
\text{ for } k \in \N
\text{ and } N_{2k+3}(2k+1,2)=k+1
\text{ for } k \in \N_0.
\end{eqnarray} 
Indeed, in Eq.~\eqref{eq2_necklace},
let us say that we have two red beads in the necklace and 
the other beads are blue. Since the rotations of the necklaces are not counted 
multiple times, each necklace is uniquely determined by the smallest number of 
blue beads between the two red beads.

\begin{theorem}\label{Cycle-counting}
Let $C_n$ be the cycle graph of order $n\geq 3$.  Then 
\begin{displaymath}
N_{sp}(C_n)= \left\{ \begin{array}{ll}
4 & \textrm{if $n \equiv 0 \pmod 4$, }\\
\\
2n & \textrm{if $n \equiv 1 \pmod 4$,}\\
\\
\frac{5n^2-10n}{8} & \textrm{if $n \equiv 2 \pmod 4$,}\\
\\
n & \textrm{if $n \equiv 3 \pmod 4$.}
\end{array} \right.
\end{displaymath}
\end{theorem}

\begin{proof}
Let $V=\{v_1,v_2,\ldots,v_n\}$ be the vertex set of $C_n$ and $S$ be a minimum 
size super dominating set for that. We consider the following 
cases:
\begin{itemize}
\item[(a)] $ n \equiv 0 \pmod 4 $.

Let $n=4k$ with $ k\in \N $. 
By Theorem~\ref{thm-2}(b), we have  $\gamma_{sp}(C_{4k})=2k$. 
Let $S$ be a super dominating set of cardinality $2k$ in $C_n$. By 
Lemma~\ref{Lem:Partition}, there is another super dominating set $ \overline{S}$ 
with same cardinality as $S$, $D=\emptyset $, 
and there is a bijective function $ f:\overline{S'} \to \overline{S} $.
Thus, we cannot have $3$ consecutive vertices in $S$ and 
for each $v\in S =\overline{S'}$, there exists a unique $u\in \overline{S} = S'$ such that $N(v)\cap 
\overline{S}=\{u\}$ and for each $ v \in S'= \overline{S} $ a unique $u \in \overline{S'} = S $ 
such that $N(v)\cap \overline{S'}=\{u\}$. Hence, 
the following four sets are the only super dominating sets in $C_{4k}$:
\begin{align*}
S_1&=\{v_1,v_2,v_5,v_6,\ldots,v_{4k-3},v_{4k-2}\},\\
S_2&=\{v_2,v_3,v_6,v_7,\ldots,v_{4k-2},v_{4k-1}\},\\
S_3&=\{v_3,v_4,v_7,v_8,\ldots,v_{4k-1},v_{4k}\},\\
S_4&=\{v_1,v_4,v_5,v_8,\ldots,v_{4k-3},v_{4k}\}.
\end{align*}

Hence, $N_{sp}(C_{4k})=4$.

\item[(b)] $ n \equiv 1 \pmod 4 $.

Let $n=4k+1$ with $k \in \N $. 
By Theorem~\ref{thm-2}(b),  we have  $\gamma_{sp}(C_{4k+1})=2k+1$. Let 
$S$ be a super dominating set of size $2k+1$ in $C_n$. First, notice that 
$|S|=|\overline{S}|+1$. Thus, if we consider Lemma~\ref{Lem:Partition} and the set 
$D$, we have $|D|=1$. Without loss of generality, $v_{4k+1}\in D$ holds. If we now 
contract the edge $v_{4k} v_{4k+1}$, then we get a cycle $C_{4k}$ in which 
$S \setminus \{v_{4k+1}\}$ is a super dominating set of cardinality $2k$. If we now recall 
the structure of a minimum size super dominating set for the cycle $C_{4k}$, then we 
notice that there are two possibilities for $v_{4k+1}\in D$, namely either 
$N(v_{4k+1})\subseteq S$ or $N(v_{4k+1})\subseteq \overline{S}$. Hence, 
$v_{4k+1}$ is either one of $3$ consecutive vertices in $S$ or a single vertex 
in $S$ surrounded by vertices in $\overline{S}$. Moreover, in both cases, we may 
rotate the super dominating set around the cycle in $n$ different ways. Thus, 
$N_{sp}(C_{4k+1})=2n$ as claimed.

\item[(c)] $ n \equiv 2 \pmod 4 $.

Let  $n=4k+2=8q+4p+2$ where $q \in \N_0, k \in \N $ and $p \in \{0,1\} $. 
By Theorem~\ref{thm-2}(b),  we have  $\gamma_{sp}(C_{4k+2})=2k+2$.  Since 
choosing any four vertices from $C_6$ gives us a super dominating set, 
$\binom{6}{4} = 15 =  (5 \cdot 6^2 - 10 \cdot 6) /8 $ is the number of super dominating sets, and it shows that 
the formula holds for $n=6$.
So, let now $n\geq 10$.

In the following, we write about $k$ \textit{consecutive vertices} in a minimum 
size super 
dominating set  $S$ implying that there is not a subset of $k+1$ 
consecutive vertices in~$S$. We claim that the following cases cannot occur for vertices in $S$: 
\begin{itemize}
\item At least $5$ consecutive vertices, 
\item twice $4$ consecutive vertices, 
\item once $4$ consecutive vertices and once $1$ consecutive vertex,
\item once $4$ consecutive vertices and once $3$ consecutive vertices, 
\item three times $3$ consecutive vertices, 
\item twice $3$ consecutive vertices and once $1$ consecutive vertex, 
\item once $3$ consecutive vertices and twice $1$ consecutive vertex, 
%\item $A_1\subseteq S$, $A_2\subseteq S$, 
%$A_3\subseteq S$ when  $A_i \cap A_j = \emptyset$,  $A_i$ has 
%consecutive vertices and $|A_i|\geq 3$ for  each $i\neq j$ where $i,j=1,2,3$, 
\item three times $1$ consecutive vertex.
%i.e., there are at most two vertices in $S$ which have their neighbourhoods in $\overline{S}$. 
\end{itemize}
The non-existence of all these cases can be shown in the same way.
By Lemma~\ref{Lem:Partition}, we have $|D| = 2$.
On the other hand, if one of these cases occurred, then we would have at least three black 
vertices surrounded either by two black vertices or by two white vertices. 
As these vertices do not super dominate other vertices, it follows 
that $ |D| \ge 3 $ leading to a contradiction.

As we are in the case $ n \equiv 2 \pmod 4 $,
it remains to consider the following cases for vertices 
in $S$, where we do not list the occurrences
of $2$ consecutive vertices.
\begin{itemize}
\item[(i)] Exactly once $4$ consecutive vertices, 
\item[(ii)] exactly twice $3$ consecutive vertices, 
\item[(iii)] exactly once $3$ consecutive vertices and exactly once $1$ consecutive vertex,
\item[(iv)] exactly twice $1$ consecutive vertex,
%\item exactly once $1$ consecutive vertex,
\item[(v)] none of that, i.e., only $2$ consecutive vertices.
\end{itemize}

We have illustrated the minimum size super dominating sets for $C_{10}$ in 
Figure~\ref{Super-C10}.

\begin{figure}
\begin{center}
\psscalebox{0.6 0.6}
{
\begin{pspicture}(0,-6.835347)(15.69,5.295347)
\psdots[linecolor=black, dotsize=0.2](1.9,3.6653473)
\psdots[linecolor=black, dotsize=0.2](2.3,3.6653473)
\psdots[linecolor=black, dotsize=0.2](2.7,3.6653473)
\psdots[linecolor=black, dotsize=0.2](3.1,3.6653473)
\psdots[linecolor=black, dotsize=0.2](2.3,3.2653472)
\psdots[linecolor=black, dotsize=0.2](2.7,3.2653472)
\psdots[linecolor=black, dotsize=0.2](3.1,3.2653472)
\psdots[linecolor=black, dotsize=0.2](3.5,3.2653472)
\psdots[linecolor=black, dotsize=0.2](2.7,2.8653474)
\psdots[linecolor=black, dotsize=0.2](3.1,2.8653474)
\psdots[linecolor=black, dotsize=0.2](3.5,2.8653474)
\psdots[linecolor=black, dotsize=0.2](3.9,2.8653474)
\psdots[linecolor=black, dotsize=0.2](3.1,2.4653473)
\psdots[linecolor=black, dotsize=0.2](3.5,2.4653473)
\psdots[linecolor=black, dotsize=0.2](3.9,2.4653473)
\psdots[linecolor=black, dotsize=0.2](4.3,2.4653473)
\psdots[linecolor=black, dotsize=0.2](3.5,2.0653472)
\psdots[linecolor=black, dotsize=0.2](3.9,2.0653472)
\psdots[linecolor=black, dotsize=0.2](4.3,2.0653472)
\psdots[linecolor=black, dotsize=0.2](4.7,2.0653472)
\psdots[linecolor=black, dotsize=0.2](3.9,1.6653473)
\psdots[linecolor=black, dotsize=0.2](4.3,1.6653473)
\psdots[linecolor=black, dotsize=0.2](4.7,1.6653473)
\psdots[linecolor=black, dotsize=0.2](5.1,1.6653473)
\psdots[linecolor=black, dotsize=0.2](4.3,1.2653472)
\psdots[linecolor=black, dotsize=0.2](4.7,1.2653472)
\psdots[linecolor=black, dotsize=0.2](5.1,1.2653472)
\psdots[linecolor=black, dotsize=0.2](5.5,1.2653472)
\psdots[linecolor=black, dotsize=0.2](4.7,0.86534727)
\psdots[linecolor=black, dotsize=0.2](5.1,0.86534727)
\psdots[linecolor=black, dotsize=0.2](5.5,0.86534727)
\psdots[linecolor=black, dotsize=0.2](1.9,0.86534727)
\psdots[linecolor=black, dotsize=0.2](5.1,0.4653473)
\psdots[linecolor=black, dotsize=0.2](5.5,0.4653473)
\psdots[linecolor=black, dotsize=0.2](1.9,0.4653473)
\psdots[linecolor=black, dotsize=0.2](2.3,0.4653473)
\psdots[linecolor=black, dotsize=0.2](5.5,0.06534729)
\psdots[linecolor=black, dotsize=0.2](1.9,0.06534729)
\psdots[linecolor=black, dotsize=0.2](2.3,0.06534729)
\psdots[linecolor=black, dotsize=0.2](2.7,0.06534729)
\psdots[linecolor=black, dotsize=0.2](4.3,3.6653473)
\psdots[linecolor=black, dotsize=0.2](4.7,3.6653473)
\psdots[linecolor=black, dotsize=0.2](4.7,3.2653472)
\psdots[linecolor=black, dotsize=0.2](5.1,3.2653472)
\psdots[linecolor=black, dotsize=0.2](5.1,2.8653474)
\psdots[linecolor=black, dotsize=0.2](5.5,2.8653474)
\psdots[linecolor=black, dotsize=0.2](5.5,2.4653473)
\psdots[linecolor=black, dotsize=0.2](1.9,2.4653473)
\psdots[linecolor=black, dotsize=0.2](1.9,2.0653472)
\psdots[linecolor=black, dotsize=0.2](2.3,2.0653472)
\psdots[linecolor=black, dotsize=0.2](2.3,1.6653473)
\psdots[linecolor=black, dotsize=0.2](2.7,1.6653473)
\psdots[linecolor=black, dotsize=0.2](2.7,1.2653472)
\psdots[linecolor=black, dotsize=0.2](3.1,1.2653472)
\psdots[linecolor=black, dotsize=0.2](3.1,0.86534727)
\psdots[linecolor=black, dotsize=0.2](3.5,0.86534727)
\psdots[linecolor=black, dotsize=0.2](3.5,0.4653473)
\psdots[linecolor=black, dotsize=0.2](3.9,0.4653473)
\psdots[linecolor=black, dotsize=0.2](3.9,0.06534729)
\psdots[linecolor=black, dotsize=0.2](4.3,0.06534729)
\psdots[linecolor=black, dotsize=0.2](1.9,-0.7346527)
\psdots[linecolor=black, dotsize=0.2](2.3,-0.7346527)
\psdots[linecolor=black, dotsize=0.2](2.7,-0.7346527)
\psdots[linecolor=black, dotsize=0.2](3.9,-0.7346527)
\psdots[linecolor=black, dotsize=0.2](4.3,-0.7346527)
\psdots[linecolor=black, dotsize=0.2](4.7,-0.7346527)
\psdots[linecolor=black, dotsize=0.2](2.3,-1.1346527)
\psdots[linecolor=black, dotsize=0.2](2.7,-1.1346527)
\psdots[linecolor=black, dotsize=0.2](3.1,-1.1346527)
\psdots[linecolor=black, dotsize=0.2](4.3,-1.1346527)
\psdots[linecolor=black, dotsize=0.2](4.7,-1.1346527)
\psdots[linecolor=black, dotsize=0.2](5.1,-1.1346527)
\psdots[linecolor=black, dotsize=0.2](2.7,-1.5346527)
\psdots[linecolor=black, dotsize=0.2](3.1,-1.5346527)
\psdots[linecolor=black, dotsize=0.2](3.5,-1.5346527)
\psdots[linecolor=black, dotsize=0.2](4.7,-1.5346527)
\psdots[linecolor=black, dotsize=0.2](5.1,-1.5346527)
\psdots[linecolor=black, dotsize=0.2](5.5,-1.5346527)
\psdots[linecolor=black, dotsize=0.2](3.1,-1.9346527)
\psdots[linecolor=black, dotsize=0.2](3.5,-1.9346527)
\psdots[linecolor=black, dotsize=0.2](3.9,-1.9346527)
\psdots[linecolor=black, dotsize=0.2](5.1,-1.9346527)
\psdots[linecolor=black, dotsize=0.2](5.5,-1.9346527)
\psdots[linecolor=black, dotsize=0.2](1.9,-1.9346527)
\psdots[linecolor=black, dotsize=0.2](3.5,-2.3346527)
\psdots[linecolor=black, dotsize=0.2](3.9,-2.3346527)
\psdots[linecolor=black, dotsize=0.2](4.3,-2.3346527)
\psdots[linecolor=black, dotsize=0.2](5.5,-2.3346527)
\psdots[linecolor=black, dotsize=0.2](1.9,-2.3346527)
\psdots[linecolor=black, dotsize=0.2](2.3,-2.3346527)
\psdots[linecolor=black, dotsize=0.2](9.9,3.6653473)
\psdots[linecolor=black, dotsize=0.2](10.3,3.6653473)
\psdots[linecolor=black, dotsize=0.2](10.7,3.6653473)
\psdots[linecolor=black, dotsize=0.2](10.3,3.2653472)
\psdots[linecolor=black, dotsize=0.2](10.7,3.2653472)
\psdots[linecolor=black, dotsize=0.2](11.1,3.2653472)
\psdots[linecolor=black, dotsize=0.2](10.7,2.8653474)
\psdots[linecolor=black, dotsize=0.2](11.1,2.8653474)
\psdots[linecolor=black, dotsize=0.2](11.5,2.8653474)
\psdots[linecolor=black, dotsize=0.2](11.1,2.4653473)
\psdots[linecolor=black, dotsize=0.2](11.5,2.4653473)
\psdots[linecolor=black, dotsize=0.2](11.9,2.4653473)
\psdots[linecolor=black, dotsize=0.2](11.5,2.0653472)
\psdots[linecolor=black, dotsize=0.2](11.9,2.0653472)
\psdots[linecolor=black, dotsize=0.2](12.3,2.0653472)
\psdots[linecolor=black, dotsize=0.2](11.9,1.6653473)
\psdots[linecolor=black, dotsize=0.2](12.3,1.6653473)
\psdots[linecolor=black, dotsize=0.2](12.7,1.6653473)
\psdots[linecolor=black, dotsize=0.2](12.3,1.2653472)
\psdots[linecolor=black, dotsize=0.2](12.7,1.2653472)
\psdots[linecolor=black, dotsize=0.2](13.1,1.2653472)
\psdots[linecolor=black, dotsize=0.2](12.7,0.86534727)
\psdots[linecolor=black, dotsize=0.2](13.1,0.86534727)
\psdots[linecolor=black, dotsize=0.2](13.5,0.86534727)
\psdots[linecolor=black, dotsize=0.2](13.1,0.4653473)
\psdots[linecolor=black, dotsize=0.2](13.5,0.4653473)
\psdots[linecolor=black, dotsize=0.2](9.9,0.4653473)
\psdots[linecolor=black, dotsize=0.2](13.5,0.06534729)
\psdots[linecolor=black, dotsize=0.2](9.9,0.06534729)
\psdots[linecolor=black, dotsize=0.2](10.3,0.06534729)
\psdots[linecolor=black, dotsize=0.2](1.9,-3.1346526)
\psdots[linecolor=black, dotsize=0.2](2.3,-3.1346526)
\psdots[linecolor=black, dotsize=0.2](2.7,-3.1346526)
\psdots[linecolor=black, dotsize=0.2](2.3,-3.5346527)
\psdots[linecolor=black, dotsize=0.2](2.7,-3.5346527)
\psdots[linecolor=black, dotsize=0.2](3.1,-3.5346527)
\psdots[linecolor=black, dotsize=0.2](2.7,-3.9346528)
\psdots[linecolor=black, dotsize=0.2](3.1,-3.9346528)
\psdots[linecolor=black, dotsize=0.2](3.5,-3.9346528)
\psdots[linecolor=black, dotsize=0.2](3.1,-4.334653)
\psdots[linecolor=black, dotsize=0.2](3.5,-4.334653)
\psdots[linecolor=black, dotsize=0.2](3.9,-4.334653)
\psdots[linecolor=black, dotsize=0.2](3.5,-4.7346525)
\psdots[linecolor=black, dotsize=0.2](3.9,-4.7346525)
\psdots[linecolor=black, dotsize=0.2](4.3,-4.7346525)
\psdots[linecolor=black, dotsize=0.2](3.9,-5.1346526)
\psdots[linecolor=black, dotsize=0.2](4.3,-5.1346526)
\psdots[linecolor=black, dotsize=0.2](4.7,-5.1346526)
\psdots[linecolor=black, dotsize=0.2](4.3,-5.5346527)
\psdots[linecolor=black, dotsize=0.2](4.7,-5.5346527)
\psdots[linecolor=black, dotsize=0.2](5.1,-5.5346527)
\psdots[linecolor=black, dotsize=0.2](4.7,-5.934653)
\psdots[linecolor=black, dotsize=0.2](5.1,-5.934653)
\psdots[linecolor=black, dotsize=0.2](5.5,-5.934653)
\psdots[linecolor=black, dotsize=0.2](5.1,-6.334653)
\psdots[linecolor=black, dotsize=0.2](5.5,-6.334653)
\psdots[linecolor=black, dotsize=0.2](1.9,-6.334653)
\psdots[linecolor=black, dotsize=0.2](5.5,-6.7346525)
\psdots[linecolor=black, dotsize=0.2](1.9,-6.7346525)
\psdots[linecolor=black, dotsize=0.2](2.3,-6.7346525)
\psdots[linecolor=black, dotsize=0.2](11.9,3.6653473)
\psdots[linecolor=black, dotsize=0.2](12.3,3.6653473)
\psdots[linecolor=black, dotsize=0.2](12.3,3.2653472)
\psdots[linecolor=black, dotsize=0.2](12.7,3.2653472)
\psdots[linecolor=black, dotsize=0.2](12.7,2.8653474)
\psdots[linecolor=black, dotsize=0.2](13.1,2.8653474)
\psdots[linecolor=black, dotsize=0.2](13.1,2.4653473)
\psdots[linecolor=black, dotsize=0.2](13.5,2.4653473)
\psdots[linecolor=black, dotsize=0.2](13.5,2.0653472)
\psdots[linecolor=black, dotsize=0.2](9.9,2.0653472)
\psdots[linecolor=black, dotsize=0.2](9.9,1.6653473)
\psdots[linecolor=black, dotsize=0.2](10.3,1.6653473)
\psdots[linecolor=black, dotsize=0.2](10.3,1.2653472)
\psdots[linecolor=black, dotsize=0.2](10.7,1.2653472)
\psdots[linecolor=black, dotsize=0.2](10.7,0.86534727)
\psdots[linecolor=black, dotsize=0.2](11.1,0.86534727)
\psdots[linecolor=black, dotsize=0.2](11.1,0.4653473)
\psdots[linecolor=black, dotsize=0.2](11.5,0.4653473)
\psdots[linecolor=black, dotsize=0.2](11.5,0.06534729)
\psdots[linecolor=black, dotsize=0.2](11.9,0.06534729)
\psdots[linecolor=black, dotsize=0.2](4.7,-3.1346526)
\psdots[linecolor=black, dotsize=0.2](4.3,-3.1346526)
\psdots[linecolor=black, dotsize=0.2](3.5,-3.1346526)
\psdots[linecolor=black, dotsize=0.2](5.1,-3.5346527)
\psdots[linecolor=black, dotsize=0.2](4.7,-3.5346527)
\psdots[linecolor=black, dotsize=0.2](3.9,-3.5346527)
\psdots[linecolor=black, dotsize=0.2](5.5,-3.9346528)
\psdots[linecolor=black, dotsize=0.2](5.1,-3.9346528)
\psdots[linecolor=black, dotsize=0.2](4.3,-3.9346528)
\psdots[linecolor=black, dotsize=0.2](1.9,-4.334653)
\psdots[linecolor=black, dotsize=0.2](5.5,-4.334653)
\psdots[linecolor=black, dotsize=0.2](4.7,-4.334653)
\psdots[linecolor=black, dotsize=0.2](2.3,-4.7346525)
\psdots[linecolor=black, dotsize=0.2](1.9,-4.7346525)
\psdots[linecolor=black, dotsize=0.2](5.1,-4.7346525)
\psdots[linecolor=black, dotsize=0.2](2.7,-5.1346526)
\psdots[linecolor=black, dotsize=0.2](2.3,-5.1346526)
\psdots[linecolor=black, dotsize=0.2](5.5,-5.1346526)
\psdots[linecolor=black, dotsize=0.2](3.1,-5.5346527)
\psdots[linecolor=black, dotsize=0.2](2.7,-5.5346527)
\psdots[linecolor=black, dotsize=0.2](1.9,-5.5346527)
\psdots[linecolor=black, dotsize=0.2](3.5,-5.934653)
\psdots[linecolor=black, dotsize=0.2](3.1,-5.934653)
\psdots[linecolor=black, dotsize=0.2](2.3,-5.934653)
\psdots[linecolor=black, dotsize=0.2](3.9,-6.334653)
\psdots[linecolor=black, dotsize=0.2](3.5,-6.334653)
\psdots[linecolor=black, dotsize=0.2](2.7,-6.334653)
\psdots[linecolor=black, dotsize=0.2](4.3,-6.7346525)
\psdots[linecolor=black, dotsize=0.2](3.9,-6.7346525)
\psdots[linecolor=black, dotsize=0.2](3.1,-6.7346525)
\psdots[linecolor=black, dotsize=0.2](13.1,3.6653473)
\psdots[linecolor=black, dotsize=0.2](13.5,3.2653472)
\psdots[linecolor=black, dotsize=0.2](9.9,2.8653474)
\psdots[linecolor=black, dotsize=0.2](10.3,2.4653473)
\psdots[linecolor=black, dotsize=0.2](10.7,2.0653472)
\psdots[linecolor=black, dotsize=0.2](11.1,1.6653473)
\psdots[linecolor=black, dotsize=0.2](11.5,1.2653472)
\psdots[linecolor=black, dotsize=0.2](11.9,0.86534727)
\psdots[linecolor=black, dotsize=0.2](12.3,0.4653473)
\psdots[linecolor=black, dotsize=0.2](12.7,0.06534729)
\psdots[linecolor=black, dotsize=0.2](9.9,-3.1346526)
\psdots[linecolor=black, dotsize=0.2](10.3,-3.1346526)
\psdots[linecolor=black, dotsize=0.2](11.1,-3.1346526)
\psdots[linecolor=black, dotsize=0.2](11.5,-3.1346526)
\psdots[linecolor=black, dotsize=0.2](12.3,-3.1346526)
\psdots[linecolor=black, dotsize=0.2](12.7,-3.1346526)
\psdots[linecolor=black, dotsize=0.2](10.3,-3.5346527)
\psdots[linecolor=black, dotsize=0.2](10.7,-3.5346527)
\psdots[linecolor=black, dotsize=0.2](11.5,-3.5346527)
\psdots[linecolor=black, dotsize=0.2](11.9,-3.5346527)
\psdots[linecolor=black, dotsize=0.2](12.7,-3.5346527)
\psdots[linecolor=black, dotsize=0.2](13.1,-3.5346527)
\psdots[linecolor=black, dotsize=0.2](10.7,-3.9346528)
\psdots[linecolor=black, dotsize=0.2](11.1,-3.9346528)
\psdots[linecolor=black, dotsize=0.2](11.9,-3.9346528)
\psdots[linecolor=black, dotsize=0.2](12.3,-3.9346528)
\psdots[linecolor=black, dotsize=0.2](13.1,-3.9346528)
\psdots[linecolor=black, dotsize=0.2](13.5,-3.9346528)
\psdots[linecolor=black, dotsize=0.2](11.1,-4.334653)
\psdots[linecolor=black, dotsize=0.2](11.5,-4.334653)
\psdots[linecolor=black, dotsize=0.2](12.3,-4.334653)
\psdots[linecolor=black, dotsize=0.2](12.7,-4.334653)
\psdots[linecolor=black, dotsize=0.2](13.5,-4.334653)
\psdots[linecolor=black, dotsize=0.2](9.9,-4.334653)
\psdots[linecolor=black, dotsize=0.2](11.5,-4.7346525)
\psdots[linecolor=black, dotsize=0.2](11.9,-4.7346525)
\psdots[linecolor=black, dotsize=0.2](12.7,-4.7346525)
\psdots[linecolor=black, dotsize=0.2](13.1,-4.7346525)
\psdots[linecolor=black, dotsize=0.2](9.9,-4.7346525)
\psdots[linecolor=black, dotsize=0.2](10.3,-4.7346525)
\psdots[linecolor=black, dotsize=0.2](11.9,-5.1346526)
\psdots[linecolor=black, dotsize=0.2](12.3,-5.1346526)
\psdots[linecolor=black, dotsize=0.2](13.1,-5.1346526)
\psdots[linecolor=black, dotsize=0.2](13.5,-5.1346526)
\psdots[linecolor=black, dotsize=0.2](10.3,-5.1346526)
\psdots[linecolor=black, dotsize=0.2](10.7,-5.1346526)
\psdots[linecolor=black, dotsize=0.2](12.3,-5.5346527)
\psdots[linecolor=black, dotsize=0.2](12.7,-5.5346527)
\psdots[linecolor=black, dotsize=0.2](13.5,-5.5346527)
\psdots[linecolor=black, dotsize=0.2](9.9,-5.5346527)
\psdots[linecolor=black, dotsize=0.2](10.7,-5.5346527)
\psdots[linecolor=black, dotsize=0.2](11.1,-5.5346527)
\psdots[linecolor=black, dotsize=0.2](12.7,-5.934653)
\psdots[linecolor=black, dotsize=0.2](13.1,-5.934653)
\psdots[linecolor=black, dotsize=0.2](9.9,-5.934653)
\psdots[linecolor=black, dotsize=0.2](10.3,-5.934653)
\psdots[linecolor=black, dotsize=0.2](11.1,-5.934653)
\psdots[linecolor=black, dotsize=0.2](11.5,-5.934653)
\psdots[linecolor=black, dotsize=0.2](13.1,-6.334653)
\psdots[linecolor=black, dotsize=0.2](13.5,-6.334653)
\psdots[linecolor=black, dotsize=0.2](10.3,-6.334653)
\psdots[linecolor=black, dotsize=0.2](10.7,-6.334653)
\psdots[linecolor=black, dotsize=0.2](11.5,-6.334653)
\psdots[linecolor=black, dotsize=0.2](11.9,-6.334653)
\psdots[linecolor=black, dotsize=0.2](13.5,-6.7346525)
\psdots[linecolor=black, dotsize=0.2](9.9,-6.7346525)
\psdots[linecolor=black, dotsize=0.2](10.7,-6.7346525)
\psdots[linecolor=black, dotsize=0.2](11.1,-6.7346525)
\psdots[linecolor=black, dotsize=0.2](11.9,-6.7346525)
\psdots[linecolor=black, dotsize=0.2](12.3,-6.7346525)
\psline[linecolor=black, linewidth=0.04](1.1,-0.33465272)(6.3,-0.33465272)(6.3,-0.33465272)
\psline[linecolor=black, linewidth=0.04](9.1,-2.7346528)(14.3,-2.7346528)(14.3,-2.7346528)
\psline[linecolor=black, linewidth=0.04](9.1,-0.33465272)(14.3,-0.33465272)(14.3,-0.33465272)
\psdots[linecolor=black, fillstyle=solid, dotstyle=o, dotsize=0.2, fillcolor=white](3.5,3.6653473)
\psdots[linecolor=black, fillstyle=solid, dotstyle=o, dotsize=0.2, fillcolor=white](3.9,3.6653473)
\psdots[linecolor=black, fillstyle=solid, dotstyle=o, dotsize=0.2, fillcolor=white](3.9,3.2653472)
\psdots[linecolor=black, fillstyle=solid, dotstyle=o, dotsize=0.2, fillcolor=white](4.3,3.2653472)
\psdots[linecolor=black, fillstyle=solid, dotstyle=o, dotsize=0.2, fillcolor=white](4.3,2.8653474)
\psdots[linecolor=black, fillstyle=solid, dotstyle=o, dotsize=0.2, fillcolor=white](4.7,2.8653474)
\psdots[linecolor=black, fillstyle=solid, dotstyle=o, dotsize=0.2, fillcolor=white](4.7,2.4653473)
\psdots[linecolor=black, fillstyle=solid, dotstyle=o, dotsize=0.2, fillcolor=white](5.1,2.4653473)
\psdots[linecolor=black, fillstyle=solid, dotstyle=o, dotsize=0.2, fillcolor=white](5.1,2.0653472)
\psdots[linecolor=black, fillstyle=solid, dotstyle=o, dotsize=0.2, fillcolor=white](5.5,2.0653472)
\psdots[linecolor=black, fillstyle=solid, dotstyle=o, dotsize=0.2, fillcolor=white](5.5,1.6653473)
\psdots[linecolor=black, fillstyle=solid, dotstyle=o, dotsize=0.2, fillcolor=white](1.9,1.6653473)
\psdots[linecolor=black, fillstyle=solid, dotstyle=o, dotsize=0.2, fillcolor=white](5.1,3.6653473)
\psdots[linecolor=black, fillstyle=solid, dotstyle=o, dotsize=0.2, fillcolor=white](5.5,3.6653473)
\psdots[linecolor=black, fillstyle=solid, dotstyle=o, dotsize=0.2, fillcolor=white](5.5,3.2653472)
\psdots[linecolor=black, fillstyle=solid, dotstyle=o, dotsize=0.2, fillcolor=white](1.9,3.2653472)
\psdots[linecolor=black, fillstyle=solid, dotstyle=o, dotsize=0.2, fillcolor=white](1.9,2.8653474)
\psdots[linecolor=black, fillstyle=solid, dotstyle=o, dotsize=0.2, fillcolor=white](2.3,2.8653474)
\psdots[linecolor=black, fillstyle=solid, dotstyle=o, dotsize=0.2, fillcolor=white](2.3,2.4653473)
\psdots[linecolor=black, fillstyle=solid, dotstyle=o, dotsize=0.2, fillcolor=white](2.7,2.4653473)
\psdots[linecolor=black, fillstyle=solid, dotstyle=o, dotsize=0.2, fillcolor=white](2.7,2.0653472)
\psdots[linecolor=black, fillstyle=solid, dotstyle=o, dotsize=0.2, fillcolor=white](3.1,2.0653472)
\psdots[linecolor=black, fillstyle=solid, dotstyle=o, dotsize=0.2, fillcolor=white](3.1,1.6653473)
\psdots[linecolor=black, fillstyle=solid, dotstyle=o, dotsize=0.2, fillcolor=white](3.5,1.6653473)
\psdots[linecolor=black, fillstyle=solid, dotstyle=o, dotsize=0.2, fillcolor=white](3.5,1.2653472)
\psdots[linecolor=black, fillstyle=solid, dotstyle=o, dotsize=0.2, fillcolor=white](3.9,1.2653472)
\psdots[linecolor=black, fillstyle=solid, dotstyle=o, dotsize=0.2, fillcolor=white](3.9,0.86534727)
\psdots[linecolor=black, fillstyle=solid, dotstyle=o, dotsize=0.2, fillcolor=white](4.3,0.86534727)
\psdots[linecolor=black, fillstyle=solid, dotstyle=o, dotsize=0.2, fillcolor=white](4.3,0.4653473)
\psdots[linecolor=black, fillstyle=solid, dotstyle=o, dotsize=0.2, fillcolor=white](4.7,0.4653473)
\psdots[linecolor=black, fillstyle=solid, dotstyle=o, dotsize=0.2, fillcolor=white](4.7,0.06534729)
\psdots[linecolor=black, fillstyle=solid, dotstyle=o, dotsize=0.2, fillcolor=white](5.1,0.06534729)
\psdots[linecolor=black, fillstyle=solid, dotstyle=o, dotsize=0.2, fillcolor=white](1.9,1.2653472)
\psdots[linecolor=black, fillstyle=solid, dotstyle=o, dotsize=0.2, fillcolor=white](2.3,1.2653472)
\psdots[linecolor=black, fillstyle=solid, dotstyle=o, dotsize=0.2, fillcolor=white](2.3,0.86534727)
\psdots[linecolor=black, fillstyle=solid, dotstyle=o, dotsize=0.2, fillcolor=white](2.7,0.86534727)
\psdots[linecolor=black, fillstyle=solid, dotstyle=o, dotsize=0.2, fillcolor=white](2.7,0.4653473)
\psdots[linecolor=black, fillstyle=solid, dotstyle=o, dotsize=0.2, fillcolor=white](3.1,0.4653473)
\psdots[linecolor=black, fillstyle=solid, dotstyle=o, dotsize=0.2, fillcolor=white](3.1,0.06534729)
\psdots[linecolor=black, fillstyle=solid, dotstyle=o, dotsize=0.2, fillcolor=white](3.5,0.06534729)
\psdots[linecolor=black, fillstyle=solid, dotstyle=o, dotsize=0.2, fillcolor=white](3.1,-0.7346527)
\psdots[linecolor=black, fillstyle=solid, dotstyle=o, dotsize=0.2, fillcolor=white](3.5,-0.7346527)
\psdots[linecolor=black, fillstyle=solid, dotstyle=o, dotsize=0.2, fillcolor=white](3.5,-1.1346527)
\psdots[linecolor=black, fillstyle=solid, dotstyle=o, dotsize=0.2, fillcolor=white](3.9,-1.1346527)
\psdots[linecolor=black, fillstyle=solid, dotstyle=o, dotsize=0.2, fillcolor=white](3.9,-1.5346527)
\psdots[linecolor=black, fillstyle=solid, dotstyle=o, dotsize=0.2, fillcolor=white](4.3,-1.5346527)
\psdots[linecolor=black, fillstyle=solid, dotstyle=o, dotsize=0.2, fillcolor=white](4.3,-1.9346527)
\psdots[linecolor=black, fillstyle=solid, dotstyle=o, dotsize=0.2, fillcolor=white](4.7,-1.9346527)
\psdots[linecolor=black, fillstyle=solid, dotstyle=o, dotsize=0.2, fillcolor=white](4.7,-2.3346527)
\psdots[linecolor=black, fillstyle=solid, dotstyle=o, dotsize=0.2, fillcolor=white](5.1,-2.3346527)
\psdots[linecolor=black, fillstyle=solid, dotstyle=o, dotsize=0.2, fillcolor=white](1.9,-1.5346527)
\psdots[linecolor=black, fillstyle=solid, dotstyle=o, dotsize=0.2, fillcolor=white](2.3,-1.5346527)
\psdots[linecolor=black, fillstyle=solid, dotstyle=o, dotsize=0.2, fillcolor=white](2.3,-1.9346527)
\psdots[linecolor=black, fillstyle=solid, dotstyle=o, dotsize=0.2, fillcolor=white](2.7,-1.9346527)
\psdots[linecolor=black, fillstyle=solid, dotstyle=o, dotsize=0.2, fillcolor=white](2.7,-2.3346527)
\psdots[linecolor=black, fillstyle=solid, dotstyle=o, dotsize=0.2, fillcolor=white](3.1,-2.3346527)
\psdots[linecolor=black, fillstyle=solid, dotstyle=o, dotsize=0.2, fillcolor=white](5.1,-0.7346527)
\psdots[linecolor=black, fillstyle=solid, dotstyle=o, dotsize=0.2, fillcolor=white](5.5,-0.7346527)
\psdots[linecolor=black, fillstyle=solid, dotstyle=o, dotsize=0.2, fillcolor=white](5.5,-1.1346527)
\psdots[linecolor=black, fillstyle=solid, dotstyle=o, dotsize=0.2, fillcolor=white](1.9,-1.1346527)
\psdots[linecolor=black, fillstyle=solid, dotstyle=o, dotsize=0.2, fillcolor=white](11.1,3.6653473)
\psdots[linecolor=black, fillstyle=solid, dotstyle=o, dotsize=0.2, fillcolor=white](11.5,3.6653473)
\psdots[linecolor=black, fillstyle=solid, dotstyle=o, dotsize=0.2, fillcolor=white](11.5,3.2653472)
\psdots[linecolor=black, fillstyle=solid, dotstyle=o, dotsize=0.2, fillcolor=white](11.9,3.2653472)
\psdots[linecolor=black, fillstyle=solid, dotstyle=o, dotsize=0.2, fillcolor=white](11.9,2.8653474)
\psdots[linecolor=black, fillstyle=solid, dotstyle=o, dotsize=0.2, fillcolor=white](12.3,2.8653474)
\psdots[linecolor=black, fillstyle=solid, dotstyle=o, dotsize=0.2, fillcolor=white](12.3,2.4653473)
\psdots[linecolor=black, fillstyle=solid, dotstyle=o, dotsize=0.2, fillcolor=white](12.7,2.4653473)
\psdots[linecolor=black, fillstyle=solid, dotstyle=o, dotsize=0.2, fillcolor=white](12.7,2.0653472)
\psdots[linecolor=black, fillstyle=solid, dotstyle=o, dotsize=0.2, fillcolor=white](13.1,2.0653472)
\psdots[linecolor=black, fillstyle=solid, dotstyle=o, dotsize=0.2, fillcolor=white](13.1,1.6653473)
\psdots[linecolor=black, fillstyle=solid, dotstyle=o, dotsize=0.2, fillcolor=white](13.5,1.6653473)
\psdots[linecolor=black, fillstyle=solid, dotstyle=o, dotsize=0.2, fillcolor=white](13.5,1.2653472)
\psdots[linecolor=black, fillstyle=solid, dotstyle=o, dotsize=0.2, fillcolor=white](9.9,1.2653472)
\psdots[linecolor=black, fillstyle=solid, dotstyle=o, dotsize=0.2, fillcolor=white](9.9,0.86534727)
\psdots[linecolor=black, fillstyle=solid, dotstyle=o, dotsize=0.2, fillcolor=white](10.3,0.86534727)
\psdots[linecolor=black, fillstyle=solid, dotstyle=o, dotsize=0.2, fillcolor=white](10.3,0.4653473)
\psdots[linecolor=black, fillstyle=solid, dotstyle=o, dotsize=0.2, fillcolor=white](10.7,0.4653473)
\psdots[linecolor=black, fillstyle=solid, dotstyle=o, dotsize=0.2, fillcolor=white](10.7,0.06534729)
\psdots[linecolor=black, fillstyle=solid, dotstyle=o, dotsize=0.2, fillcolor=white](11.1,0.06534729)
\psdots[linecolor=black, fillstyle=solid, dotstyle=o, dotsize=0.2, fillcolor=white](12.7,3.6653473)
\psdots[linecolor=black, fillstyle=solid, dotstyle=o, dotsize=0.2, fillcolor=white](13.5,3.6653473)
\psdots[linecolor=black, fillstyle=solid, dotstyle=o, dotsize=0.2, fillcolor=white](13.1,3.2653472)
\psdots[linecolor=black, fillstyle=solid, dotstyle=o, dotsize=0.2, fillcolor=white](13.5,2.8653474)
\psdots[linecolor=black, fillstyle=solid, dotstyle=o, dotsize=0.2, fillcolor=white](9.9,3.2653472)
\psdots[linecolor=black, fillstyle=solid, dotstyle=o, dotsize=0.2, fillcolor=white](10.3,2.8653474)
\psdots[linecolor=black, fillstyle=solid, dotstyle=o, dotsize=0.2, fillcolor=white](10.7,2.4653473)
\psdots[linecolor=black, fillstyle=solid, dotstyle=o, dotsize=0.2, fillcolor=white](11.1,2.0653472)
\psdots[linecolor=black, fillstyle=solid, dotstyle=o, dotsize=0.2, fillcolor=white](11.5,1.6653473)
\psdots[linecolor=black, fillstyle=solid, dotstyle=o, dotsize=0.2, fillcolor=white](11.9,1.2653472)
\psdots[linecolor=black, fillstyle=solid, dotstyle=o, dotsize=0.2, fillcolor=white](12.3,0.86534727)
\psdots[linecolor=black, fillstyle=solid, dotstyle=o, dotsize=0.2, fillcolor=white](12.7,0.4653473)
\psdots[linecolor=black, fillstyle=solid, dotstyle=o, dotsize=0.2, fillcolor=white](13.1,0.06534729)
\psdots[linecolor=black, fillstyle=solid, dotstyle=o, dotsize=0.2, fillcolor=white](9.9,2.4653473)
\psdots[linecolor=black, fillstyle=solid, dotstyle=o, dotsize=0.2, fillcolor=white](10.3,2.0653472)
\psdots[linecolor=black, fillstyle=solid, dotstyle=o, dotsize=0.2, fillcolor=white](10.7,1.6653473)
\psdots[linecolor=black, fillstyle=solid, dotstyle=o, dotsize=0.2, fillcolor=white](11.1,1.2653472)
\psdots[linecolor=black, fillstyle=solid, dotstyle=o, dotsize=0.2, fillcolor=white](11.5,0.86534727)
\psdots[linecolor=black, fillstyle=solid, dotstyle=o, dotsize=0.2, fillcolor=white](11.9,0.4653473)
\psdots[linecolor=black, fillstyle=solid, dotstyle=o, dotsize=0.2, fillcolor=white](12.3,0.06534729)
\psdots[linecolor=black, fillstyle=solid, dotstyle=o, dotsize=0.2, fillcolor=white](5.1,-3.1346526)
\psdots[linecolor=black, fillstyle=solid, dotstyle=o, dotsize=0.2, fillcolor=white](5.5,-3.1346526)
\psdots[linecolor=black, fillstyle=solid, dotstyle=o, dotsize=0.2, fillcolor=white](5.5,-3.5346527)
\psdots[linecolor=black, fillstyle=solid, dotstyle=o, dotsize=0.2, fillcolor=white](1.9,-3.5346527)
\psdots[linecolor=black, fillstyle=solid, dotstyle=o, dotsize=0.2, fillcolor=white](1.9,-3.9346528)
\psdots[linecolor=black, fillstyle=solid, dotstyle=o, dotsize=0.2, fillcolor=white](2.3,-3.9346528)
\psdots[linecolor=black, fillstyle=solid, dotstyle=o, dotsize=0.2, fillcolor=white](2.3,-4.334653)
\psdots[linecolor=black, fillstyle=solid, dotstyle=o, dotsize=0.2, fillcolor=white](2.7,-4.334653)
\psdots[linecolor=black, fillstyle=solid, dotstyle=o, dotsize=0.2, fillcolor=white](2.7,-4.7346525)
\psdots[linecolor=black, fillstyle=solid, dotstyle=o, dotsize=0.2, fillcolor=white](3.1,-4.7346525)
\psdots[linecolor=black, fillstyle=solid, dotstyle=o, dotsize=0.2, fillcolor=white](3.1,-5.1346526)
\psdots[linecolor=black, fillstyle=solid, dotstyle=o, dotsize=0.2, fillcolor=white](3.5,-5.1346526)
\psdots[linecolor=black, fillstyle=solid, dotstyle=o, dotsize=0.2, fillcolor=white](3.5,-5.5346527)
\psdots[linecolor=black, fillstyle=solid, dotstyle=o, dotsize=0.2, fillcolor=white](3.9,-5.5346527)
\psdots[linecolor=black, fillstyle=solid, dotstyle=o, dotsize=0.2, fillcolor=white](3.9,-5.934653)
\psdots[linecolor=black, fillstyle=solid, dotstyle=o, dotsize=0.2, fillcolor=white](4.3,-5.934653)
\psdots[linecolor=black, fillstyle=solid, dotstyle=o, dotsize=0.2, fillcolor=white](4.3,-6.334653)
\psdots[linecolor=black, fillstyle=solid, dotstyle=o, dotsize=0.2, fillcolor=white](4.7,-6.334653)
\psdots[linecolor=black, fillstyle=solid, dotstyle=o, dotsize=0.2, fillcolor=white](4.7,-6.7346525)
\psdots[linecolor=black, fillstyle=solid, dotstyle=o, dotsize=0.2, fillcolor=white](5.1,-6.7346525)
\psdots[linecolor=black, fillstyle=solid, dotstyle=o, dotsize=0.2, fillcolor=white](3.1,-3.1346526)
\psdots[linecolor=black, fillstyle=solid, dotstyle=o, dotsize=0.2, fillcolor=white](3.5,-3.5346527)
\psdots[linecolor=black, fillstyle=solid, dotstyle=o, dotsize=0.2, fillcolor=white](3.9,-3.9346528)
\psdots[linecolor=black, fillstyle=solid, dotstyle=o, dotsize=0.2, fillcolor=white](4.3,-4.334653)
\psdots[linecolor=black, fillstyle=solid, dotstyle=o, dotsize=0.2, fillcolor=white](4.7,-4.7346525)
\psdots[linecolor=black, fillstyle=solid, dotstyle=o, dotsize=0.2, fillcolor=white](5.1,-5.1346526)
\psdots[linecolor=black, fillstyle=solid, dotstyle=o, dotsize=0.2, fillcolor=white](5.5,-5.5346527)
\psdots[linecolor=black, fillstyle=solid, dotstyle=o, dotsize=0.2, fillcolor=white](3.9,-3.1346526)
\psdots[linecolor=black, fillstyle=solid, dotstyle=o, dotsize=0.2, fillcolor=white](4.3,-3.5346527)
\psdots[linecolor=black, fillstyle=solid, dotstyle=o, dotsize=0.2, fillcolor=white](4.7,-3.9346528)
\psdots[linecolor=black, fillstyle=solid, dotstyle=o, dotsize=0.2, fillcolor=white](5.1,-4.334653)
\psdots[linecolor=black, fillstyle=solid, dotstyle=o, dotsize=0.2, fillcolor=white](5.5,-4.7346525)
\psdots[linecolor=black, fillstyle=solid, dotstyle=o, dotsize=0.2, fillcolor=white](1.9,-5.1346526)
\psdots[linecolor=black, fillstyle=solid, dotstyle=o, dotsize=0.2, fillcolor=white](2.3,-5.5346527)
\psdots[linecolor=black, fillstyle=solid, dotstyle=o, dotsize=0.2, fillcolor=white](2.7,-5.934653)
\psdots[linecolor=black, fillstyle=solid, dotstyle=o, dotsize=0.2, fillcolor=white](3.1,-6.334653)
\psdots[linecolor=black, fillstyle=solid, dotstyle=o, dotsize=0.2, fillcolor=white](3.5,-6.7346525)
\psdots[linecolor=black, fillstyle=solid, dotstyle=o, dotsize=0.2, fillcolor=white](1.9,-5.934653)
\psdots[linecolor=black, fillstyle=solid, dotstyle=o, dotsize=0.2, fillcolor=white](2.3,-6.334653)
\psdots[linecolor=black, fillstyle=solid, dotstyle=o, dotsize=0.2, fillcolor=white](2.7,-6.7346525)
\psdots[linecolor=black, fillstyle=solid, dotstyle=o, dotsize=0.2, fillcolor=white](10.7,-3.1346526)
\psdots[linecolor=black, fillstyle=solid, dotstyle=o, dotsize=0.2, fillcolor=white](11.1,-3.5346527)
\psdots[linecolor=black, fillstyle=solid, dotstyle=o, dotsize=0.2, fillcolor=white](11.5,-3.9346528)
\psdots[linecolor=black, fillstyle=solid, dotstyle=o, dotsize=0.2, fillcolor=white](11.9,-4.334653)
\psdots[linecolor=black, fillstyle=solid, dotstyle=o, dotsize=0.2, fillcolor=white](12.3,-4.7346525)
\psdots[linecolor=black, fillstyle=solid, dotstyle=o, dotsize=0.2, fillcolor=white](12.7,-5.1346526)
\psdots[linecolor=black, fillstyle=solid, dotstyle=o, dotsize=0.2, fillcolor=white](13.1,-5.5346527)
\psdots[linecolor=black, fillstyle=solid, dotstyle=o, dotsize=0.2, fillcolor=white](13.5,-5.934653)
\psdots[linecolor=black, fillstyle=solid, dotstyle=o, dotsize=0.2, fillcolor=white](9.9,-6.334653)
\psdots[linecolor=black, fillstyle=solid, dotstyle=o, dotsize=0.2, fillcolor=white](10.3,-6.7346525)
\psdots[linecolor=black, fillstyle=solid, dotstyle=o, dotsize=0.2, fillcolor=white](11.9,-3.1346526)
\psdots[linecolor=black, fillstyle=solid, dotstyle=o, dotsize=0.2, fillcolor=white](12.3,-3.5346527)
\psdots[linecolor=black, fillstyle=solid, dotstyle=o, dotsize=0.2, fillcolor=white](12.7,-3.9346528)
\psdots[linecolor=black, fillstyle=solid, dotstyle=o, dotsize=0.2, fillcolor=white](13.1,-4.334653)
\psdots[linecolor=black, fillstyle=solid, dotstyle=o, dotsize=0.2, fillcolor=white](13.5,-4.7346525)
\psdots[linecolor=black, fillstyle=solid, dotstyle=o, dotsize=0.2, fillcolor=white](9.9,-5.1346526)
\psdots[linecolor=black, fillstyle=solid, dotstyle=o, dotsize=0.2, fillcolor=white](10.3,-5.5346527)
\psdots[linecolor=black, fillstyle=solid, dotstyle=o, dotsize=0.2, fillcolor=white](10.7,-5.934653)
\psdots[linecolor=black, fillstyle=solid, dotstyle=o, dotsize=0.2, fillcolor=white](11.1,-6.334653)
\psdots[linecolor=black, fillstyle=solid, dotstyle=o, dotsize=0.2, fillcolor=white](11.5,-6.7346525)
\psdots[linecolor=black, fillstyle=solid, dotstyle=o, dotsize=0.2, fillcolor=white](13.1,-3.1346526)
\psdots[linecolor=black, fillstyle=solid, dotstyle=o, dotsize=0.2, fillcolor=white](13.5,-3.1346526)
\psdots[linecolor=black, fillstyle=solid, dotstyle=o, dotsize=0.2, fillcolor=white](13.5,-3.5346527)
\psdots[linecolor=black, fillstyle=solid, dotstyle=o, dotsize=0.2, fillcolor=white](9.9,-3.5346527)
\psdots[linecolor=black, fillstyle=solid, dotstyle=o, dotsize=0.2, fillcolor=white](9.9,-3.9346528)
\psdots[linecolor=black, fillstyle=solid, dotstyle=o, dotsize=0.2, fillcolor=white](10.3,-3.9346528)
\psdots[linecolor=black, fillstyle=solid, dotstyle=o, dotsize=0.2, fillcolor=white](10.3,-4.334653)
\psdots[linecolor=black, fillstyle=solid, dotstyle=o, dotsize=0.2, fillcolor=white](10.7,-4.334653)
\psdots[linecolor=black, fillstyle=solid, dotstyle=o, dotsize=0.2, fillcolor=white](10.7,-4.7346525)
\psdots[linecolor=black, fillstyle=solid, dotstyle=o, dotsize=0.2, fillcolor=white](11.1,-4.7346525)
\psdots[linecolor=black, fillstyle=solid, dotstyle=o, dotsize=0.2, fillcolor=white](11.1,-5.1346526)
\psdots[linecolor=black, fillstyle=solid, dotstyle=o, dotsize=0.2, fillcolor=white](11.5,-5.1346526)
\psdots[linecolor=black, fillstyle=solid, dotstyle=o, dotsize=0.2, fillcolor=white](11.5,-5.5346527)
\psdots[linecolor=black, fillstyle=solid, dotstyle=o, dotsize=0.2, fillcolor=white](11.9,-5.5346527)
\psdots[linecolor=black, fillstyle=solid, dotstyle=o, dotsize=0.2, fillcolor=white](11.9,-5.934653)
\psdots[linecolor=black, fillstyle=solid, dotstyle=o, dotsize=0.2, fillcolor=white](12.3,-5.934653)
\psdots[linecolor=black, fillstyle=solid, dotstyle=o, dotsize=0.2, fillcolor=white](12.3,-6.334653)
\psdots[linecolor=black, fillstyle=solid, dotstyle=o, dotsize=0.2, fillcolor=white](12.7,-6.334653)
\psdots[linecolor=black, fillstyle=solid, dotstyle=o, dotsize=0.2, fillcolor=white](12.7,-6.7346525)
\psdots[linecolor=black, fillstyle=solid, dotstyle=o, dotsize=0.2, fillcolor=white](13.1,-6.7346525)
\psdots[linecolor=black, dotsize=0.2](9.9,-0.7346527)
\psdots[linecolor=black, dotsize=0.2](10.3,-0.7346527)
\psdots[linecolor=black, dotsize=0.2](11.1,-0.7346527)
\psdots[linecolor=black, dotsize=0.2](11.9,-0.7346527)
\psdots[linecolor=black, dotsize=0.2](12.3,-0.7346527)
\psdots[linecolor=black, dotsize=0.2](13.1,-0.7346527)
\psdots[linecolor=black, fillstyle=solid, dotstyle=o, dotsize=0.2, fillcolor=white](10.7,-0.7346527)
\psdots[linecolor=black, fillstyle=solid, dotstyle=o, dotsize=0.2, fillcolor=white](11.5,-0.7346527)
\psdots[linecolor=black, fillstyle=solid, dotstyle=o, dotsize=0.2, fillcolor=white](12.7,-0.7346527)
\psdots[linecolor=black, fillstyle=solid, dotstyle=o, dotsize=0.2, fillcolor=white](13.5,-0.7346527)
\psdots[linecolor=black, dotsize=0.2](10.3,-1.1346527)
\psdots[linecolor=black, dotsize=0.2](10.7,-1.1346527)
\psdots[linecolor=black, dotsize=0.2](11.5,-1.1346527)
\psdots[linecolor=black, dotsize=0.2](12.3,-1.1346527)
\psdots[linecolor=black, dotsize=0.2](12.7,-1.1346527)
\psdots[linecolor=black, dotsize=0.2](13.5,-1.1346527)
\psdots[linecolor=black, fillstyle=solid, dotstyle=o, dotsize=0.2, fillcolor=white](11.1,-1.1346527)
\psdots[linecolor=black, fillstyle=solid, dotstyle=o, dotsize=0.2, fillcolor=white](11.9,-1.1346527)
\psdots[linecolor=black, fillstyle=solid, dotstyle=o, dotsize=0.2, fillcolor=white](13.1,-1.1346527)
\psdots[linecolor=black, fillstyle=solid, dotstyle=o, dotsize=0.2, fillcolor=white](9.9,-1.1346527)
\psdots[linecolor=black, dotsize=0.2](10.7,-1.5346527)
\psdots[linecolor=black, dotsize=0.2](11.1,-1.5346527)
\psdots[linecolor=black, dotsize=0.2](11.9,-1.5346527)
\psdots[linecolor=black, dotsize=0.2](12.7,-1.5346527)
\psdots[linecolor=black, dotsize=0.2](13.1,-1.5346527)
\psdots[linecolor=black, dotsize=0.2](9.9,-1.5346527)
\psdots[linecolor=black, fillstyle=solid, dotstyle=o, dotsize=0.2, fillcolor=white](11.5,-1.5346527)
\psdots[linecolor=black, fillstyle=solid, dotstyle=o, dotsize=0.2, fillcolor=white](12.3,-1.5346527)
\psdots[linecolor=black, fillstyle=solid, dotstyle=o, dotsize=0.2, fillcolor=white](13.5,-1.5346527)
\psdots[linecolor=black, fillstyle=solid, dotstyle=o, dotsize=0.2, fillcolor=white](10.3,-1.5346527)
\psdots[linecolor=black, dotsize=0.2](11.1,-1.9346527)
\psdots[linecolor=black, dotsize=0.2](11.5,-1.9346527)
\psdots[linecolor=black, dotsize=0.2](12.3,-1.9346527)
\psdots[linecolor=black, dotsize=0.2](13.1,-1.9346527)
\psdots[linecolor=black, dotsize=0.2](13.5,-1.9346527)
\psdots[linecolor=black, dotsize=0.2](10.3,-1.9346527)
\psdots[linecolor=black, fillstyle=solid, dotstyle=o, dotsize=0.2, fillcolor=white](11.9,-1.9346527)
\psdots[linecolor=black, fillstyle=solid, dotstyle=o, dotsize=0.2, fillcolor=white](12.7,-1.9346527)
\psdots[linecolor=black, fillstyle=solid, dotstyle=o, dotsize=0.2, fillcolor=white](9.9,-1.9346527)
\psdots[linecolor=black, fillstyle=solid, dotstyle=o, dotsize=0.2, fillcolor=white](10.7,-1.9346527)
\psdots[linecolor=black, dotsize=0.2](11.5,-2.3346527)
\psdots[linecolor=black, dotsize=0.2](11.9,-2.3346527)
\psdots[linecolor=black, dotsize=0.2](12.7,-2.3346527)
\psdots[linecolor=black, dotsize=0.2](13.5,-2.3346527)
\psdots[linecolor=black, dotsize=0.2](9.9,-2.3346527)
\psdots[linecolor=black, dotsize=0.2](10.7,-2.3346527)
\psdots[linecolor=black, fillstyle=solid, dotstyle=o, dotsize=0.2, fillcolor=white](12.3,-2.3346527)
\psdots[linecolor=black, fillstyle=solid, dotstyle=o, dotsize=0.2, fillcolor=white](13.1,-2.3346527)
\psdots[linecolor=black, fillstyle=solid, dotstyle=o, dotsize=0.2, fillcolor=white](10.3,-2.3346527)
\psdots[linecolor=black, fillstyle=solid, dotstyle=o, dotsize=0.2, fillcolor=white](11.1,-2.3346527)
\psline[linecolor=black, linewidth=0.04](5.9,4.4653473)(1.5,4.4653473)(1.5,5.2653475)(5.9,5.2653475)(5.9,4.4653473)
\psdots[linecolor=black, fillstyle=solid, dotstyle=o, dotsize=0.2, fillcolor=white](1.9,4.4653473)
\psdots[linecolor=black, fillstyle=solid, dotstyle=o, dotsize=0.2, fillcolor=white](2.3,4.4653473)
\psdots[linecolor=black, fillstyle=solid, dotstyle=o, dotsize=0.2, fillcolor=white](2.7,4.4653473)
\psdots[linecolor=black, fillstyle=solid, dotstyle=o, dotsize=0.2, fillcolor=white](3.1,4.4653473)
\psdots[linecolor=black, fillstyle=solid, dotstyle=o, dotsize=0.2, fillcolor=white](3.5,4.4653473)
\psdots[linecolor=black, fillstyle=solid, dotstyle=o, dotsize=0.2, fillcolor=white](3.9,4.4653473)
\psdots[linecolor=black, fillstyle=solid, dotstyle=o, dotsize=0.2, fillcolor=white](4.3,4.4653473)
\psdots[linecolor=black, fillstyle=solid, dotstyle=o, dotsize=0.2, fillcolor=white](4.7,4.4653473)
\psdots[linecolor=black, fillstyle=solid, dotstyle=o, dotsize=0.2, fillcolor=white](5.1,4.4653473)
\psdots[linecolor=black, fillstyle=solid, dotstyle=o, dotsize=0.2, fillcolor=white](5.5,4.4653473)
\rput[bl](1.74,4.605347){1}
\rput[bl](2.18,4.645347){2}
\rput[bl](2.62,4.625347){3}
\rput[bl](3.02,4.625347){4}
\rput[bl](3.38,4.585347){5}
\rput[bl](3.76,4.605347){6}
\rput[bl](4.22,4.585347){7}
\rput[bl](4.62,4.605347){8}
\rput[bl](4.98,4.605347){9}
\rput[bl](5.26,4.605347){10}
\psline[linecolor=black, linewidth=0.04](13.9,4.4653473)(9.5,4.4653473)(9.5,5.2653475)(13.9,5.2653475)(13.9,4.4653473)
\psdots[linecolor=black, fillstyle=solid, dotstyle=o, dotsize=0.2, fillcolor=white](9.9,4.4653473)
\psdots[linecolor=black, fillstyle=solid, dotstyle=o, dotsize=0.2, fillcolor=white](10.3,4.4653473)
\psdots[linecolor=black, fillstyle=solid, dotstyle=o, dotsize=0.2, fillcolor=white](10.7,4.4653473)
\psdots[linecolor=black, fillstyle=solid, dotstyle=o, dotsize=0.2, fillcolor=white](11.1,4.4653473)
\psdots[linecolor=black, fillstyle=solid, dotstyle=o, dotsize=0.2, fillcolor=white](11.5,4.4653473)
\psdots[linecolor=black, fillstyle=solid, dotstyle=o, dotsize=0.2, fillcolor=white](11.9,4.4653473)
\psdots[linecolor=black, fillstyle=solid, dotstyle=o, dotsize=0.2, fillcolor=white](12.3,4.4653473)
\psdots[linecolor=black, fillstyle=solid, dotstyle=o, dotsize=0.2, fillcolor=white](12.7,4.4653473)
\psdots[linecolor=black, fillstyle=solid, dotstyle=o, dotsize=0.2, fillcolor=white](13.1,4.4653473)
\psdots[linecolor=black, fillstyle=solid, dotstyle=o, dotsize=0.2, fillcolor=white](13.5,4.4653473)
\rput[bl](9.74,4.605347){1}
\rput[bl](10.18,4.645347){2}
\rput[bl](10.62,4.625347){3}
\rput[bl](11.02,4.625347){4}
\rput[bl](11.38,4.585347){5}
\rput[bl](11.76,4.605347){6}
\rput[bl](12.22,4.585347){7}
\rput[bl](12.62,4.605347){8}
\rput[bl](12.98,4.605347){9}
\rput[bl](13.26,4.605347){10}
\psline[linecolor=black, linewidth=0.04](1.1,-2.7346528)(6.3,-2.7346528)(6.3,-2.7346528)
\rput[bl](0.2,1.7653472){\LARGE{$(i)$}}
\rput[bl](0.1,-1.9346527){\LARGE{$(ii)$}}
\rput[bl](0.0,-5.2346525){\LARGE{$(iii)$}}
\rput[bl](14.6,1.7653472){\LARGE{$(iii)$}}
\rput[bl](14.7,-1.7346528){\LARGE{$(iv)$}}
\rput[bl](14.8,-5.2346525){\LARGE{$(v)$}}
\end{pspicture}
}
\end{center}
\caption{All super dominating sets of $C_{10}$.} \label{Super-C10}
\end{figure}
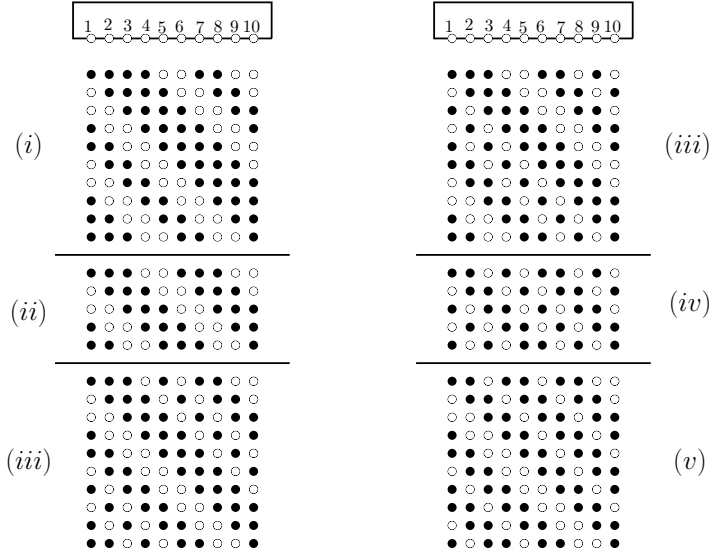
  
%The cases where $p$ is $0$ or $1$ are slightly different but yield 
%the same result and are considered together. 
As mentioned, we will utilize necklaces (see Definition \ref{def:necklace}).
We transform some sequences of vertices into beads according to the following 
list. After that, we consider the resulting cycle as a necklace.
\begin{itemize}
\item {\bf B-bead} (black bead): B,
\item {\bf DB-bead} (double black bead): BB,
\item {\bf P-bead} (pair bead): WBBW,
\item {\bf T-bead} (triple bead): WBBBW,
\item {\bf Q-bead} (quadruple bead): WBBBBW,
\item {\bf DP-bead} (double pair bead): WBBWBBW.
\end{itemize}

We distinguish between the five cases (i)--(v) mentioned above.

\begin{itemize}
\item[(i)]
The  $4$ consecutive vertices form a Q-bead. 
If we convert the cycle together with the super dominating set into a necklace 
consisting of $2q-1+p$ P-beads and a single Q-bead, then by Eq.~\eqref{eq1_necklace} 
we attain $N_{2q+p}(2q-1+p,1)=1$ different necklaces. Now, this 
single necklace provides $n$ different super dominating sets with rotations.

\item[(ii)]
We attain a necklace consisting of $2q-2+p$ 
P-beads and two T-beads. We have $N_{2q+p}(2q-2+p,2)=q$ by Eq.~\eqref{eq2_necklace}.
Furthermore, we can again rotate these $n$ times. However, if $p=0$ then the necklace with $q-1$ P-beads between the two T-beads yields only $n/2$ different super dominating 
sets. Thus, in this case we have $qn-(1-p)n/2$ different super dominating sets.

\item[(iii)]
We attain a necklace consisting of $2q-1+p$ P-beads, one 
T-bead and one B-bead. We have $N_{2q+p+1}(2q-1+p,1,1)=2q+p$. 
Indeed, the number of P-beads $h$ which we can have  between the B-bead and 
the T-bead is $0\leq h\leq 2q-1+p$. Moreover, $n$ rotations give $n$ 
different super dominating sets for each of these necklaces. Thus, we have 
$n(2q+p)$ different super dominating sets in this case.

\item[(iv)]
We attain a necklace consisting of $2q+p$ P-beads and two B-beads. We have 
$N_{2q+p+2}(2q+p,2)=q+1$ by Eq.~\eqref{eq2_necklace}. However, we do not 
consider the case where we have two adjacent B-beads since in that case we 
would actually have $2$ consecutive black vertices and we would be considering 
Case (v). As in Case (ii), if we have $p=0$, then the necklace which has $q$ 
P-beads between the two B-beads give only $n/2$ different super 
dominating sets with rotations. In all other cases, we get $n$ different super 
dominating sets. Thus, we have $qn-(1-p)n/2$ different super dominating sets in 
this Case (iv).

\item[(v)]
In this case, we have two vertices $u_1,u_2\in\overline{S}$ such that they 
are both super dominated by two different vertices. 

Let us first consider the case where $u_1$ and $u_2$ have distance $3$. 
We attain a necklace consisting of $2q+p$ P-beads and 
one DB-bead. By Eq.~\eqref{eq1_necklace}, we have $N_{2q+p}(2q+p,1)=1$.
We can rotate this case $n$ times and thus, we have $n$ different 
super dominating sets.

Let us then assume that $u_1$ and $u_2$ have distance of at least $7$. 
Here the white vertex in the middle corresponds to $u_1$ or 
$u_2$. In this case, we attain a necklace consisting of $2q-3+p$ DB-beads and 
two DP-beads. We have $N_{2q-1+p}(2q-3+p,2)=q-1+p$ by Eq.~\eqref{eq2_necklace}. Notice that when $p=1$ the necklace in which we have $q-1$ 
DB-beads between the two DP-beads gives us only $n/2$ different super 
dominating sets. All other necklaces give $n$ different super dominating sets. 
Thus, we have $n(q-1)+pn/2$ different super dominating sets in this Case (v).
\end{itemize}

By summing all these different cases together, we get
\begin{eqnarray*}
& & N_{sp}(C_n)
\\
& = &  n + qn-(1-p)n/2+ n(2q+p) + qn-(1-p)n/2 + n + n(q-1)+pn/2  
\\
& = & 5qn+(5/2)pn
\\
& = & \frac{5n^2-10n}{8} 
\end{eqnarray*}
super dominating sets, as claimed.

\item[(d)] $ n \equiv 3 \pmod 4 $.

Let $n=4k+3$ with $ k\in \N_0 $. 
It is easy to see that $N_{sp}(C_{3})=3$. Now let $n\geq 7$. 
By Theorem~\ref{thm-2}(b), $\gamma_{sp}(C_{4k+3})=2k+2$. So we need to choose 
$2k+2$ vertices in a proper way to have a super dominating set.  First, we show 
that it is not possible to have $3$ consecutive vertices in a minimum size super 
dominating set $S$. Suppose that we have $3$ consecutive vertices $v$, $v'$, 
$v''$ and contract one of the two corresponding edges ($v v'$ or $v' v''$). 
Then we have a super dominating set of size $2k+1$ in 
$C_{4k+2}$, a contradiction. Second, with the same contraction technique, we 
notice that  it is not possible among $3$ consecutive vertices for the middle one 
to be in $S$ and for the two others to be in $\overline{S}$. Hence, we can only 
have $2$ consecutive vertices in $S$. 

Since $\gamma_{sp}(C_{4k+3})=2k+2$, we have $k+1$ sets of size two with 
consecutive vertices in $S$. Moreover, we have $|\overline{S}|=2k+1$. Thus, we 
have one vertex $v \in \overline{S} $ with $N(v)\subseteq S$ and $2k$ sets of size two with 
consecutive vertices in $\overline{S}$. Assume first that $v=v_1$. Now we can 
construct $n$ different super dominating sets by rotating them around the cycle. 
Thus, $N_{sp}(C_{4k+3})=n$.
\end{itemize}  
Therefore, we have the result. \qed
\end{proof}

%We end this section by the following example:

%\begin{example}\label{ExampleC10}

%Consider $C_{10}$. We  present all super dominating sets of it to show how the 
%proof of Case (3) in Theorem~\ref{Cycle-counting} works.  Consider 
%Figure~\ref{Super-C10}, which shows all cases for super dominating sets of 
%cardinality 6. Note that filled vertices are in the super dominating set and 
%empty vertices are not, and the red cases follow the rows $\frac{k}{2}$ and 
%$\frac{k}{2}+2k$ in Figure~\ref{cycle2mod8} by shifting the indices. Similarly, 
%we can find all super dominating sets for $C_n=C_{4k+2}$ where $n \equiv 6 \pmod 8$.

%\end{example}

%\begin{figure}
%\centering
%\includegraphics[width=0.95\textwidth, height = 12cm]{superdomsetc10.png}
%\caption{ All super dominating sets of $C_{10}$.} 
%\label{Super-C10}
%\end{figure}

\section{Conclusions and future work} \label{Sec:conclusion}

In this paper, we obtained tight results on the super domination number of 
graphs, particularly on the neighbourhood corona product, $r$-gluing and Haj\'{o}s sum 
of two graphs. For each of these, we presented tight lower and upper bounds 
together with constructions attaining these upper bounds. Moreover, we gave the 
exact number of minimum size super dominating sets of some graph classes such as paths and 
cycles. Finally, we provide the following suggestions for future research.

\begin{itemize}
\item[(a)] The formula $\gamma_{sp}(G \star H)= n(\gamma_{sp}(H)+1)$ of
Corollary~\ref{corneighbour} does not cover the case $ H=K_1$.
Thus, it would be interesting to compute $\gamma_{sp}(G \star K_1) $. By 
Eq.~(\ref{eqmn}), $\gamma_{sp}(G \star K_1) \leq \gamma_{sp}(G)+n$ 
is a (trivial) upper bound, but it is tight only for some graphs $G$.
\item[(b)] We found a connection for the super domination number between the 
usual corona product of two graphs~\cite{Kle} and the neighbourhood corona product 
of two graphs. Does a generalization exist for the corona products that 
preserve the result of Corollary~\ref{corneighbour}?
%\item[(c)]
%Count the number of super dominating sets with  cardinality at least  
%$\gamma_{sp}(G)$, where $G$ is an arbitrary  graph.
\item[(c)]
Give general upper and lower bounds for the number of minimum size super dominating sets with help of $\gamma_{sp}(G)$ and  possibly some other graph parameters.
\item[(d)] Count the number of (minimum size) super dominating sets in trees.

\end{itemize}

\section{Acknowledgements} 

The  first author would like to thank the Research Council of Norway and 
Department of Informatics, University of Bergen, for their support. The third 
author was funded by the Academy of Finland, grant 338797.

%\COMG{Maybe use a bib-file.} 

\end{document}